\documentclass[]{interact}

\usepackage{epstopdf}
\usepackage[caption=false]{subfig}

\usepackage[numbers,sort&compress]{natbib}
\bibpunct[, ]{[}{]}{,}{n}{,}{,}
\makeatletter
\def\NAT@def@citea{\def\@citea{\NAT@separator}}
\makeatother

\theoremstyle{plain}
\newtheorem{theorem}{Theorem}[section]
\newtheorem{lemma}[theorem]{Lemma}

\newtheorem{proposition}[theorem]{Proposition}

\theoremstyle{definition}
\newtheorem{definition}[theorem]{Definition}
\newtheorem{example}[theorem]{Example}

\theoremstyle{remark}
\newtheorem{remark}{Remark}

\begin{document}


\title{Granular fractional Caputo-Katugampola derivatives and their applications in optimality conditions for fuzzy fractional variational problems}

\author{
\name{Le Thanh Tung\textsuperscript{a}\thanks{CONTACT L.~T. Tung. Email: lttung@ctu.edu.vn}, Tran Thien Khai\textsuperscript{b}, Pham Le Bach Ngoc\textsuperscript{c,d} and Trinh Tung\textsuperscript{e,f}}
\affil{\textsuperscript{a} Faculty of Mathematics, College of Natural Sciences, Can Tho University, Can Tho, Vietnam;
\textsuperscript{b} Tra Vinh University, Vinh Long, Vietnam;
\textsuperscript{c} Kien Giang University, An Giang, Vietnam;
\textsuperscript{d} University of Szeged, Szeged, Hungary;
\textsuperscript{e} Faculty
of Mathematics and Computer Science, University of Science, Ho Chi Minh City, Vietnam;
\textsuperscript{f} Vietnam National University, Ho Chi Minh City, Vietnam}
}
\date{}
\maketitle
\begin{center}
{\it
Dedicated to Professor Phan Quoc Khanh on the occasion of his 80th birthday}
\end{center}

\begin{abstract}
In this article, we present the new definition of the fuzzy Caputo-Katugampola derivative and its related concepts for the class of fuzzy functions and apply them to establish both necessary and sufficient optimality conditions for fractional fuzzy variational problems. By adopting the  horizontal membership functional representation of fuzzy numbers, the granular Caputo-Katugampola derivatives offer higher computational efficiency than the previous approaches in certain scenarios. The optimality conditions for the fuzzy
fractional variational problems under granular fuzzy Caputo
fractional derivatives are also examined and illustrated with detailed examples. Another novel aspect of our results, even in the non-fuzzy case, is the application of special functions to represent the exact solutions of the fuzzy fractional variational problems.
\end{abstract}


\section{Introduction}
Fractional calculus	is one of the important fields in mathematics and has received much attention recently due to its diverse applications in many different areas of science, technology, and life,  see e.g. the books \cite{KST06,MDM18} for detailed applications. Since fractional calculus generalizes classic integration and differentiation to arbitrary orders, it could be employed to model complex systems such as viscoelasticity, anomalous diffusion by choosing a suitable order of integrals and derivatives.
The Riemann-Liouville fractional integrals, Riemann-Liouville  fractional derivatives and Caputo fractional derivatives are  three important and frequently used concepts in
fractional calculus. In \cite{K11,K14}, Riemann-Liouville fractional integrals and Riemann-Liouville  fractional derivatives were generalized by using a scale transformation of the form $t^\rho$ in the kernel.
The Caputo-type fractional derivatives was  extended in a similar way in \cite{AMO16,A17}. Since the additional $\rho$ parameter in the Caputo-Katugampola (CK, in brief) derivative contributes to flexible adjustment for more accurate description of complex systems in fractional calculus, it has attracted significant research interest in recent years. Optimization using fractional calculus, alternatively termed ``Fractional-order optimization", is one of the important research directions in optimization and has been studied extensively recently. The necessary and sufficient optimality conditions in  fractional-order optimization problems are essential conditions in both solving fractional-order optimization  problems analytically and building algorithms to solve fractional-order optimization  problems via numerical schemes.

Real-world optimization problems frequently feature imprecise objective and constraint functions, driven by measurement errors or unexpected variables. Addressing this uncertainty quantitatively via interval concepts or fuzzy concepts has proven highly effective for decision-making. For the interval-valued and fuzzy optimization problems in recent times, see e.g. the interval-valued smooth nonlinear programming \cite{A24,T25}, the fuzzy nonlinear programming on manifold \cite{BI26}, the interval-valued convex semi-infinite programming  \cite{T20},  interval-valued variational control problems \cite{T21a} and references therein. The variational problems, used in formulating problems in the various applied and theoretical aspects of mathematics, were also generalized to the fuzzy variational problems. In \cite{F11}, the necessary optimality conditions for fuzzy variational problems were
investigated by using the level-set differentiability.  The optimality conditions for the fuzzy variational problems of several dependent variables
were established by utilizing the granular fuzzy derivatives in \cite{TT23}. In \cite{FS14}, some necessary optimality conditions for fuzzy fractional variational problems were established via Caputo-type fuzzy fractional derivatives. Under the granular fuzzy Cuputo fractional derivatives in \cite{NZ17}, the optimality conditions for the fuzzy variational problems were examined.

We observe that the granular fuzzy Caputo derivatives in \cite{NZ17} could be generalized to the granular fuzzy Caputo-Katugampola derivatives and apply to fuzzy fractional-order optimization  problems. That motivates us to introduce a novel definition of the fuzzy Caputo-Katugampola  derivative and its related concepts and utilize them to obtain both necessary and sufficient
optimality conditions for the fuzzy fractional variational problems in this paper. The structure of the paper is as follows. We first recall basic concepts and some preliminaries about the fuzzy functions, the fractional derivatives and special functions in Section 2. Then, we propose the granular  Caputo-Katugampola derivative and its related concepts in Section 3. A comprehensive example is given to illustrate the definitions of both left and right granular  Caputo-Katugampola derivatives as well as to compare with the results in \cite{NZ17} and  \cite{HVD19}. In Section 4, both necessary and sufficient optimality conditions for the fundamental fuzzy fractional variational problems and the fuzzy variational problems  with fractional isoperimetric constraints are then examined. The examples illuminating the results of the paper by using special functions are also new and verified by using Octave and Matlab.
\section{Preliminaries}
	\label{sect2}
	In this paper, the class of all closed and bounded intervals in $\mathbb{R}$, denoted by $\mathbf{K}_C$, is
	$$\mathbf{K}_C=\{U=[u^L,u^R]\mid u^L,u^R\in\mathbb{R}, u^L\le u^R\}.$$
	A fuzzy set $\widetilde{u}$ on $\mathbb{R}$ is defined by a function $\rho_{\widetilde{u}}:\mathbb{R}\to [0,1]$, which is called a membership function.
	The $\mu$-level set of $\widetilde{X}$, indicated by $[\widetilde{u}]^{\mu}$, is defined as $[\widetilde{u}]^{\mu}:=\{x\in\mathbb{R}\mid \rho_{\widetilde{u}}(x)\ge \mu\},$ $\forall \mu\in(0,1]$. The support of $\widetilde{u}$ is the set ${\rm supp}(\widetilde{u}):=\{x\in\mathbb{R}\mid \rho_{\widetilde{u}}(x)> 0\}$. The zero-level set of $\widetilde{X}$ is defined as the closure of the support of $\widetilde{u}$, i.e., $[\widetilde{u}]^0={\rm cl}\{x\in\mathbb{R}\mid \rho_{\widetilde{u}}(x)> 0\}$.
	\begin{definition}\label{defn2.1} A fuzzy number is a fuzzy set $\widetilde{u}$ with membership function $\rho_{\widetilde{u}}$ satisfying the following conditions:
		\begin{enumerate}
			\item[(i)] $\rho_{\widetilde{u}}$ is normal, that is, there exists $\bar{x}\in\mathbb{R}$ such that $\rho_{\widetilde{u}}(\bar{x})=1$;
			\item[(ii)] $\rho_{\widetilde{u}}$ is quasiconcave, i.e., $$\rho_{\widetilde{u}}(\lambda x+(1-\lambda)x')\ge \max\{ \rho_{\widetilde{u}}(x),\rho_{\widetilde{u}}(x')\}$$ for all $\lambda\in [0,1]$, for all $x,x'\in \widetilde{u}$;
			\item[(iii)] $\rho_{\widetilde{u}}$ is upper semicontinuous, i.e., $\{x\in\mathbb{R}\mid \rho_{\widetilde{u}}(x)\ge \mu\}$ is a closed subset of $\mathbb{R}$ for each $\mu\in (0,1]$;
			\item[(iv)] $[\widetilde{u}]^0$ is a compact subset of $\mathbb{R}$.
		\end{enumerate}	
	\end{definition}
The set of all fuzzy numbers on $\mathbb{R}$ is signified by $\mathbf{F}(\mathbb{R})$. The condition (ii) leads that $[\widetilde{u}]^{\mu}$ is a convex set for each $\mu\in [0,1]$. Combining this with conditions (iii) and (iv) tells us that $[\widetilde{u}]^{\mu}$ is a compact and convex subset of $\mathbb{R}$ for each $\mu\in [0,1]$. In other words, $[\widetilde{u}]^{\mu}=[u_{\mu}^L,u_{\mu}^R]\in \mathbf{K}_C$, which is called 	the parametric form of a fuzzy number $\widetilde{u}$. For instance, the triangular fuzzy number $\widetilde{a}=(a^L,a,a^R)$ with $a^L\le a\le a^R$ has the parametric form $[\widetilde{a}]^{\mu}=[a_{\mu}^L,a_{\mu}^R]=[a^L+(a-a^L)\mu,
a^R-(a^R-a)\mu]$.

The pair $[\widetilde{u}]^{\mu}=[u^L_{\mu},u^R_{\mu}]$ of functions $u^L_{\mu}=\phi_1(\mu)$ and $u^R_{\mu}=\phi_2(\mu)$ ($\mu\in[0, 1]$) satisfies the following conditions: $\phi_1(\mu)$ is a monotonically increasing left continuous function, $\phi_2(\mu)$ is a monotonically decreasing left continuous function and $u^L_{\mu}\le u^R_{\mu}$ ($\mu\in[0, 1]$).
A fuzzy number $\widetilde{u}$ is said to be a canonical number in the case when the functions $\phi_1(\mu)=u^L_{\mu}$ and $\phi_2(\mu)=u^R_{\mu}$ are continuous on $[0,1]$. The set of all canonical fuzzy numbers on $\mathbb{R}$ is denoted by $\mathbf{F}_C(\mathbb{R})$.

	\begin{definition} \cite{MPK17,PL15} Let $\widetilde{u}:[a,b]\to [0,1]$ be in $\mathbf{F}_C(\mathbb{R})$. Then, the horizontal membership function of $X$ is defined as follows
		\begin{eqnarray*}
			u^{gr}:[0,1]\times [0,1]&\to& [a,b]\\
			(\mu,\beta_u)&\mapsto&t:=u^{gr}(\mu,\beta_u)=u_\mu^L
			+(u_\mu^R-u_\mu^L)\beta_u.
		\end{eqnarray*}
		The horizontal membership function of $\widetilde{u}\in \mathbf{F}_C(\mathbb{R})$ is also denoted by $\mathcal{H}(\widetilde{u}):=u^{gr}(\mu,\beta_u)$.
	\end{definition}
	\begin{proposition} \cite{MPK17,PL15} The $\mu$-level sets of $\widetilde{u}$ can be obtained by the formula
		$$[u^L_{\mu},u^R_{\mu}]=[\mathcal{H}^{-1}(u^{gr}(\mu,\beta_u))]^{\mu}=
		\left[\inf\limits_{\gamma\ge \mu}\min\limits_{\beta_u}
		u^{gr}(\gamma,\beta_u),\sup\limits_{\gamma\ge \mu}\max\limits_{\beta_u}
		u^{gr}(\gamma,\beta_u)\right].$$
	\end{proposition}
	\begin{definition}  Let $\widetilde{u},\widetilde{v}:[a,b]\to [0,1]$ be in $\mathbf{F}_C(\mathbb{R})$.
		\begin{enumerate}
			\item[(i)] \cite{MPK17,SDL19} $\widetilde{u}=\widetilde{v}$ if $\mathcal{H}(\widetilde{u})=\mathcal{H}(\widetilde{v})$ ($u^{gr}(\mu,\beta_u)=v^{gr}(\mu,\beta_v)$)  for all $\beta_u=\beta_v\in [0,1]$ and $\mu\in [0,1]$.
			\item[(ii)] \cite{MPK17,SDL19} $\widetilde{u}\ge\widetilde{v}$ if $\mathcal{H}(\widetilde{u})\ge\mathcal{H}(\widetilde{v})$ ($u^{gr}(\mu,\beta_u)\ge v^{gr}(\mu,\beta_v)$)  for all $\beta_u=\beta_v\in [0,1]$ and $\mu\in [0,1]$.
			\item[(iii)] \cite{TT23} $\widetilde{u}>\widetilde{v}$ if $\mathcal{H}(\widetilde{u})>\mathcal{H}(\widetilde{v})$ ($u^{gr}(\mu,\beta_u)> v^{gr}(\mu,\beta_v)$)  for all $\beta_u=\beta_v\in [0,1]$ and $\mu\in [0,1]$.
		\end{enumerate}
	\end{definition}
	\begin{definition} (see \cite{GH96,KST06}) Let $ c\in \mathbb{R}, 1\le p\le \infty$ and $\phi:[a,b]\to \mathbb{R}$ be a real-valued function.
		\begin{enumerate}
\item[(i)]  $X^p_c[a,b]$ is the space of Lebesgue measurable functions $\phi$ on $[a,b]$ for which $\|\phi\|_{X^p_c}<\infty$, where
$$\|\phi\|_{X^p_c}=\left\{\begin{array}{ll}
\left(\int_a^b|t^c\phi(t)\frac{dt}{t}|\right)^{1/p}& \mbox{if}\;
1\le p<\infty\\
\mbox{ess}\sup\limits_{a\le t\le b}\{t^c|\phi(t)|\}& \mbox{if}\;
p=\infty.
\end{array}
\right.$$
When $c = 1/p$, the space $X^p_{1/p}$ coincides with the $L^p[a,b]$.
			\item[(ii)] $C[a,b]$ is the space of all continuous functions on $[a,b]$, i.e.,
			$$C[a,b]=\{\phi\mid \phi \;\mbox{is continuous on}\; [a,b]\}.$$			
\item[(iii)] $AC[a,b]$ is the space of all absolutely continuous functions on $[a,b]$, i.e.,
			$$AC[a,b]=\left\{\phi\mid  \phi(t)=C_0+\int_a^b\phi_1(t)dt, \phi_1\in L[a,b]\right\}.$$
\item[(iv)] $C^1[a,b]$ is the space of all continuously differentiable functions on $[a,b]$, i.e.,
			$$C^1[a,b]=\{\phi\mid \exists \phi':[a,b] \to \mathbb{R}\;\mbox{and}\; \phi' \;\mbox{is continuous on}\; [a,b]\}.$$
		\end{enumerate}
	\end{definition}
By the definition, we get that
$$C^1[a,b]\subset AC[a,b]\subset C[a,b]\subset X^p_c[a,b],$$
see also \cite{GH96,KST06} for more details. Hence, the assumptions that $\psi\in X^p_c[a,b]$ is also satisfied for $\psi\in AC[a,b]$, $\psi\in C[a,b]$ or $\psi\in C^1[a,b]$, since these weaker
statements are strong enough for the assumptions.

	\begin{definition}\cite{MPK17} Let $\widetilde{\phi} :[a, b] \to \mathbf{F}_C(\mathbb{R})$ be a fuzzy function including $n$ distinct fuzzy numbers $\widetilde{u}_1,\widetilde{u}_2,...,$ $\widetilde{u}_n$. The horizontal membership function of $\widetilde{\phi}(t)$ at the point $t\in [a,b]$ is denoted by $\mathcal{H}(\widetilde{\phi}(t)):=\phi^{gr}(t,\mu,\beta_{\psi})$,
		where $\phi^{gr}:[a, b]\times [0,1]\times [0,1]^n\to [c,d]\subset\mathbb{R}$ and $\beta_{\phi}=(\beta_{u_1},...,\beta_{u_n})$.
	\end{definition}
	\begin{definition} \cite{MPK17} Let $\widetilde{\phi} :[a,b] \to\mathbf{F}_C(\mathbb{R})$ be a fuzzy-valued function.
		
		\noindent (i) The granular metric on $\mathbf{F}_C(\mathbb{R})$ is a mapping $D^{gr}:\mathbf{F}_C(\mathbb{R})\times \mathbf{F}_C(\mathbb{R})\to \mathbb{R}_+$, is given by
		$$D^{gr}(\widetilde{u},\widetilde{v}):=\sup\limits_{\mu}
		\max\limits_{\beta_u,\beta_v}\mid u^{gr}(\mu ,\beta_u)
		-v^{gr}(\mu ,\beta_v)\mid .$$
		
		\noindent (ii) The fuzzy function $\widetilde{\phi} :[a, b] \to \mathbf{F}_C(\mathbb{R})$ is said to be continuous on $[a,b]$ if $\lim\limits_{t\to \bar{t}}\widetilde{\phi}(t)=\widetilde{\phi}(\bar{t}) \forall \bar{t}\in (a,b)$, $\lim\limits_{t\to a^+}\widetilde{\phi}(t)=\widetilde{\phi}(a)$, $\lim\limits_{t\to b^-}\widetilde{\phi}(t)=\widetilde{\phi}(b)$.
	\end{definition}
	
	\begin{definition}
		\cite{MPK17} The fuzzy function $\widetilde{\phi} :(a,b) \to\mathbf{F}_C(\mathbb{R})$ is said to be granular differentiable (gr-differentiable) at
		$\bar{t}\in (a,b)$, if there exists a fuzzy number $\widetilde{\dot{\phi}}(\bar{t})\in \mathbf{F}_C(\mathbb{R})$, such
		that the following limit exists:
		$$\lim\limits_{h\to 0}\frac{\widetilde{\phi}(\bar{t}+h)\ominus_{gr}
\widetilde{\phi}(\bar{t})}{h}=
		\widetilde{\dot{\phi}}(\bar{t})$$
		and the value $\widetilde{\dot{\phi}}(\bar{t})$ is then called gr-derivative of fuzzy-valued function $\widetilde{\phi}$ at $\bar{t}$. We say that $\widetilde{\phi}$ is gr-differentiable on $(a,b)$ if the gr-derivative
		$\widetilde{\dot{\phi}}(t)$ exists for all points $t\in (a,b)$ and $\frac{d_{gr}\phi}{dt}:=\widetilde{\dot{\phi}}:(a,b)\to \mathbf{F}_C(\mathbb{R})$
		is called gr-derivative of $\widetilde{\phi}$ on $(a,b)$.
	\end{definition}

	\begin{proposition}
		\cite{MPK17}
		The fuzzy function $\widetilde{\phi} :(a, b)\to \mathbf{F}_C(\mathbb{R})$  is granular differentiable at
		$\bar{t}\in (a,b)$ if and only if $\mathcal{H}(\widetilde{\phi}(.))=\phi^{gr}(.,\mu,\beta_{\phi})$ is  differentiable at $\bar{t}$ and $$\mathcal{H}
		\left(\widetilde{\dot{\phi}}(\bar{t})\right)=\frac{\partial\mathcal{H}
			(\widetilde{\phi}(\bar{t}))}{\partial t}
		=\frac{\partial \phi^{gr}(\bar{t},\mu,\beta_{\phi})}{\partial t}.$$
	\end{proposition}
	\begin{proposition}
		The fuzzy function $\widetilde{\phi} :[a, b]\to \mathbf{F}_C(\mathbb{R})$  is granular differentiable at
		$\bar{t}=a (\bar{t}=b)$ if and only if its
		horizontal membership function $\mathcal{H}(\widetilde{\phi}(t))=\phi^{gr}(t,\mu,\beta_{\phi})$ is  differentiable with respect
		to $t$ at that point. Moreover,
		$$\mathcal{H}
		\left(\widetilde{\dot{\phi}}(a)\right)
		=\frac{\partial \mathcal{H}(\widetilde{\phi}(a))}{\partial t}
		\left(\mathcal{H}\left(\widetilde{\dot{\phi}}(b)\right)
		=\frac{\partial \mathcal{H}(\widetilde{\phi}(b))}{\partial t}\right).$$
	\end{proposition}
	
	\begin{definition}Let $\widetilde{\phi}:[a,b]\to \mathbf{F}_C(\mathbb{R})$ be a fuzzy-valued function.
		\begin{enumerate}
			\item[(i)] $\mathcal{C}([a,b])$ is the space of all continuous fuzzy functions  on $[a,b]$, i.e.,
			$$\mathcal{C}([a,b])=\{\widetilde{\phi}\mid \widetilde{\phi} \;\mbox{is continuous on}\; [a,b]\}.$$
			\item[(ii)] $\mathcal{C}^1([a,b])$ is the space of all continuously gr-differentiable fuzzy functions \cite{DLK20,SDL19} on $[a,b]$, i.e.,
			$$\mathcal{C}^1([a,b])=\{\widetilde{\phi}\mid \exists \widetilde{\dot{\phi}}:[a,b] \to \mathbf{F}_C(\mathbb{R})\;\mbox{and}\; \widetilde{\dot{\psi}} \;\mbox{is continuous on}\; [a,b]\}.$$
		\end{enumerate}
	\end{definition}
\begin{remark}\label{rem2.1}
		The fuzzy function $\widetilde{\phi} :[a, b] \to \mathbf{F}_C(\mathbb{R})$  is continuously $gr$-differentiable on
		$[a,b]$ if and only if its
		horizontal membership function $\mathcal{H}(\widetilde{\phi}(t))=\phi^{gr}(t,\mu,\beta_{\psi})$ is  continuously
differentiable on $[a,b]$.
	\end{remark}
\begin{definition} \cite{SB12} Let $\tilde{u}, \tilde{v}\in \mathbf{F}_C(\mathbb{R})$,
\begin{enumerate}
	\item[(i)]  The generalized Hukuhara difference ($gH$-difference, in brief) of   $\tilde{u}, \tilde{v}$  is defined as follows:
$$\tilde{u} \ominus_{gH}\tilde{v}=\tilde{w}
\Leftrightarrow \left\{\begin{array}{c}
(i)\; \tilde{u}=\tilde{v}+\tilde{w}\\
\mbox{or}\; (ii)\;\tilde{v}=\tilde{u}+(-1)\tilde{w}.
\end{array}
\right. $$
The $\tilde{u} \ominus_{gH}\tilde{v}$  is
given levelwise by:
$$[\tilde{u} \ominus_{gH}\tilde{v}]^\mu=\left[\min\{u_{\mu}^R-v_{\mu}^R,
u_{\mu}^L-v_{\mu}^L\},\max \{u_{\mu}^R-v_{\mu}^R,
u_{\mu}^L-v_{\mu}^L\}\right].$$
\item[(ii)] The fuzzy function $\widetilde{\phi} :(a,b) \to\mathbf{F}_C(\mathbb{R})$ is said to be $gH$-differentiable at		$\bar{t}\in (a,b)$, if there exists a fuzzy number $\widetilde{\dot{\phi}}(\bar{t})\in \mathbf{F}_C(\mathbb{R})$, such
		that the following limit exists:
		$$\lim\limits_{h\to 0}\frac{\widetilde{\phi}(\bar{t}+h)\ominus_{gH}
\widetilde{\phi}(\bar{t})}{h}=
		\widetilde{\dot{\phi}}^{gH}(\bar{t}).$$
\end{enumerate}
\end{definition}
\begin{definition} \cite{MPK17}
		Let $\widetilde{\phi} :[a, b] \to \mathbf{F}_C(\mathbb{R})$ be a fuzzy function and $\mathcal{H}(\widetilde{\phi}(t))=\phi^{gr}(t,\mu,\beta_{\phi})$
		be integrable on $[a,b]$ with the integral $\widetilde{\int}_a^b\widetilde{\phi}(t)dt$. Then, the fuzzy
		function $\widetilde{\phi}$ is said to be gr-integrable on $[a,b]$ if there
		exists a fuzzy number $\widetilde{p}=\widetilde{\int}_a^b\widetilde{\phi}(t)dt$ such that $$\mathcal{H}\left(\widetilde{\int}_a^b\widetilde{\phi}(t)dt\right)
		=\mathcal{H}(\widetilde{p})=
		\int_a^b\mathcal{H}(\widetilde{\phi}(t))dt
		=\int_a^b\phi^{gr}(t,\mu,\beta_{\phi})dt.$$
	\end{definition}

\begin{definition}\label{defn2.14}
		Let $\phi:[a,b]\to \mathbb{R}$ be of class $X^p_c[a,b]$, $0<\beta<1$ and $\rho>0$.
		\begin{enumerate}
			\item[(i)] \cite{K11,K14} The left (right) Katugampola-Riemann-Liouville (KRL, in brief) integral of $\phi$ of order $\beta$ is defined by
			$$I^{\beta,\rho}_{a^+}\phi(t):=\frac{\rho^{1-\beta}}{\Gamma(\beta)}
\int_a^t\frac{\tau^{\rho-1}}{(t^{\rho}-\tau^{\rho})^{1-\beta}}
\phi(\tau)d\tau, t\ge a$$
			$$\left(I^{\beta,\rho}_{b^-}\phi(t):=\frac{\rho^{1-\beta}}{\Gamma(\beta)}		\int_t^b\frac{\tau^{\rho-1}}{(\tau^{\rho}-t^{\rho})^{1-\beta}}
\phi(\tau)d\tau, t\le b\right).$$
			\item[(ii)] \cite{K11,K14} Let $\phi\in C^1[a,b]$. The left (right) KRL fractional derivative of $\phi$ of order $\beta$ is
			$$D^{\beta,\rho}_{a^+}\phi(t):=\frac{\rho^{\beta}}{\Gamma(1-\beta)}
t^{1-\rho}				.\frac{d}{dt}\left(\int_a^t\frac{\tau^{\rho-1}}{(t^{\rho}
-\tau^{\rho})^{\beta}}\phi(\tau)d\tau\right), t\ge a$$
			$$\left(D^{\beta,\rho}_{b^-}\phi(t):=
-\frac{\rho^{\beta}}{\Gamma(1-\beta)}
t^{1-\rho}				.\frac{d}{dt}\left(\int_t^b\frac{\tau^{\rho-1}}
{(\tau^{\rho}-t^{\rho})^{\beta}}\phi(\tau)d\tau\right), t\le b\right).$$
			\item[(iii)] \cite{AMO16,A17} Let $\phi\in C^1[a,b]$. The left (right) Caputo-Katugampola (CK, in brief) fractional derivative of $\phi$ of order $\beta$ is defined by
			$${}^{C}D^{\beta,\rho}_{a^+}\phi(t):=
\frac{\rho^{\beta}}{\Gamma(1-\beta)}			\int_a^t\frac{\phi'(\tau)}{(\tau^{\rho}-t^{\rho})^{\beta}}d\tau, t\ge a$$
			$$\left({}^{C}D^{\beta,\rho}_{b^-}\phi(t):=
-\frac{\rho^{\beta}}{\Gamma(1-\beta)}			\int_t^b\frac{\phi'(\tau)}{(\tau^{\rho}-t^{\rho})^{\beta}}d\tau, t\le b\right).$$
		\end{enumerate}
	\end{definition}

	\begin{remark}
		\begin{enumerate}
			\item[(i)] $D^{\beta,\rho}_{a^+}\phi(t)=t^{1-\rho}	\frac{d}{dt}(I^{1-\beta,\rho}_{a^+}\phi(t)), D^{\beta,\rho}_{b^-}\phi(t)=-
t^{1-\rho}\frac{d}{dt}(I^{1-\beta,\rho}_{b^-}\phi(t)).$
			\item[(ii)] ${^CD}^{\beta,\rho}_{a^+}\phi(t)
=I^{1-\beta,\rho}_{a^+}\left(\tau^{1-\rho}\frac{d}{d\tau}\phi\right)(t), {^CD}^{\beta,\rho}_{b^-}\phi(t)=I^{1-\beta,\rho}_{b^-}
\left(-\tau^{1-\rho}\frac{d}{d\tau}\phi\right)(t).$
		\end{enumerate}
	\end{remark}
When $\rho=1$, the concepts in Definition \ref{defn2.14} reduce to
the Riemann-Liouville integral,  Riemann-Liouville derivative and Caputo derivative, see e.g. \cite{KST06}.

Now, we recall some properties of special functions in \cite{AAR99,SC12} and the references therein as follows.
\begin{remark}\label{rem2.6}
\begin{enumerate}
\item[(i)] The gamma function
$\Gamma:\mathbb{R}  \to\mathbb{R}$, beta function $B:\mathbb{R}\times \mathbb{R}  \to\mathbb{R}$ and digamma function $\psi:\mathbb{R}  \to\mathbb{R}$ are defined, respectively, by
$$\Gamma(x)=\int_0^{\infty}e^{-t} t^{x-1}dt, B(x,y)=\int_0^1(1-t)^{x-1} t^{y-1}dt, \psi(x)=\frac{\Gamma'(x)}{\Gamma(x)}.$$
The Pochhammer symbol of $\lambda\in\mathbb{R}$ is defined by $(\lambda)_n=\lambda(\lambda+1)...(\lambda+n-1)$ for $n>0$ and $(\lambda)_0=1$.
\item[(ii)] $\Gamma(x+1)=x\Gamma(x)$, $B(x,y)=\frac{\Gamma(x)\Gamma(y)}{\Gamma(x+y)}$ and
    $\psi(x)=\frac{d}{dx} (\ln (\Gamma (x))$.
\item[(iii)] $\frac{\Gamma(x+a)}{\Gamma(x)}=x^a\left(1+\frac{a(a-1)}{2x}
    +o(\frac{1}{x^2})\right)$, i.e., $\frac{\Gamma(x+a)}{\Gamma(x)}\thicksim x^a$ when $x\to\infty$.
\item[(iv)] $\psi'(x)=\frac{d^2}{dx^2}( \ln (\Gamma (x))=\sum\limits_{n=0}^{\infty}\left(\frac{1}{x+n}\right)^2>0$.
\end{enumerate}
\end{remark}
\begin{remark}\label{rem2.7}
\begin{enumerate}
\item[(i)] The hypergeometric function
$_2F_1:\mathbb{R}^3\times[0,1]\to \mathbb{R}$ and $_3F_2:\mathbb{R}^5\times[0,1]\to \mathbb{R}$ are defined, respectively, by
$$_2F_1(a,b;c;t)=\sum\limits_{n=0}^{\infty}\frac{(a)_n(b)_n}{(c)_n}.
\frac{t^n}{n!}\;\left({_3F_2}(a,b,c;d,e;t)=
\sum\limits_{n=0}^{\infty}\frac{(a)_n(b)_n(c)_n}{(d)_n(e)_n}.
\frac{t^n}{n!}\right).$$
The functions $_2F_1$ and $_3F_2$ converge for $|t|<1$.
\item[(ii)] $_2F_1(a,b;c;t)=_2F_1(b,a;c;t)$ and $_2F_1(0,b;c;t)=_2F_1(a,0;c;t)=1$.
\item[(iii)] The Euler’s integral
representation of hypergeometric function
$_2F_1:\mathbb{R}^3\times[0,1]\to \mathbb{R}$ is 			$$_2F_1(a,b;c;t)=\frac{\Gamma(a)\Gamma(c)}{\Gamma(c-a)}
			\int_0^1\tau^{a-1}(1-\tau)^{c-a-1}
			(1-t\tau)^{-b}d\tau.$$
The function $_2F_1$ converges when $c>a>0,\mid t\mid <1$.
The derivative of $_2F_1$ is
$$\frac{d}{dt}({_2F_1}(a,b;c;t))=\frac{ab}{c}{_2F_1}(a+1,b+1;c+1;t).$$
\item[(iv)] If $c>a+b$ and $c\not\in \{0,-1,-2,...\}$, then
$_2F_1(a,b;c;1)=\frac{\Gamma(c)\Gamma(c-a-b)}
{\Gamma(c-a)\Gamma(c-b)}.$
Especially,
$$_2F_1(1,b;c;1)=\frac{\Gamma(c)\Gamma(c-1-b)}
{\Gamma(c-1)\Gamma(c-b)}=\frac{c-1}{c-b-1}.$$
If $d+e>a+b+c>0$ and $d+e\not\in \{0,-1,-2,...\}$, then ${_3F_2}(a,b,c;d,e;1)$ converges.
\item[(v)] If $a=1$ and $c-b-1\le 0$,
\begin{eqnarray*}
_2F_1(1,b;c;1)&=&\frac{\Gamma(1)\Gamma(c)}{\Gamma(c-1)}
			\int_0^1\tau^{1-1}(1-\tau)^{c-1-1}
			(1-\tau)^{-b}d\tau\\
&=&-\frac{\Gamma(c)}{\Gamma(c-1)}
			\lim\limits_{\epsilon\to 0^+}\int_0^{1-\epsilon}(1-\tau)^{(c-b-1)-1}
			d(1-\tau)\\
&=&\left\{\begin{array}{ll}
-\frac{\Gamma(c)}{\Gamma(c-1)}\lim\limits_{\epsilon\to 0^+}			\left.\ln (1-\tau)\right|_0^{1-\epsilon}, & \mbox{if}\; c-b-1=0\\
-\frac{\Gamma(c)}{\Gamma(c-1)}\lim\limits_{\epsilon\to 0^+}			\left.\frac{(1-\tau)^{c-b-1}}{c-b-1}\right|_0^{1-\epsilon}, &  \mbox{if}\;c-b-1<0
\end{array}\right.\\
&=&+\infty.
\end{eqnarray*}
\end{enumerate}
\end{remark}
We generalize the results in Lemma 7 with $\rho=1$ in  \cite{TTKN25}  into the case of an arbitrary $\rho>0$ as follows.

\begin{lemma}\label{lem2.15} Let $\beta\in (0,1)$ and $\rho>0$.
		\begin{enumerate}\item[(i)] We have
			$$u_1(t):=\int_0^t(t^\rho-\tau^\rho)^{\beta-1}			\tau^{\rho-1}(1-\tau^\rho)^{\beta-1}d\tau=\frac{t^{\beta \rho}}{\beta\rho}.
			{_2F_1}(1,1-\beta;1+\beta;t^\rho),$$
			$$u_2(t):=\int_0^t(t^\rho-\tau^\rho)^{\beta-1}			\tau^{\rho-1}(1-\tau^\rho)^{\beta}d\tau=\frac{t^{\beta\rho}}{\beta\rho}.
			{_2F_1}(1,-\beta;1+\beta;t^\rho).$$

			\item[(ii)] If $\beta>\frac{1}{2}$, then
			$$u_1(1)=\frac{1}{\beta\rho}.
{_2F_1}(1,1-\beta;1+\beta;1)=\frac{1}{(2\beta-1)\rho},$$
			$$u_2(1)=\frac{1}{\beta\rho}.{_2F_1}(1,-\beta;1+\beta;1)
=\frac{1}{2\beta\rho}.$$
			\item[(iii)]
$$\int_0^1u_1(t)dt
=\frac{1}{\beta\rho(\beta\rho+1)}. {_3F_2}(1,1-\beta,\beta+\frac{1}{\rho};1+\beta,\beta+\frac{1}{\rho}+1;1),$$
$$\int_0^1u_2(t)dt
=\frac{1}{\beta\rho(\beta\rho+1)}. {_3F_2}(1,-\beta,\beta+\frac{1}{\rho};1+\beta,\beta+\frac{1}{\rho}+1;1).$$
		
		\end{enumerate}
	\end{lemma}
\begin{proof}
		(i) By changing variable $\tau_1=\frac{\tau^\rho}{t^\rho}$, we have $\tau^\rho=t^\rho\tau_1\to \rho\tau^{\rho-1}d\tau=t^\rho\tau_1d\tau_1$ and
		\begin{eqnarray*}
			\int_0^t(t^\rho-\tau^\rho)^{\beta-1}			\tau^{\rho-1}(1-\tau^\rho)^{\beta-1}d\tau
&=&\int_0^1(t^\rho-t^\rho\tau_1)^{\beta-1}
			(1-t^\rho\tau_1)^{\beta-1}.\frac{t^\rho}{\rho}d\tau_1\\
			&=& \frac{t^{\beta \rho}}{\rho}\int_0^1(1-\tau_1)^{\beta-1}
			(1-t^\rho\tau_1)^{\beta-1}d\tau_1\\
			&=& \frac{t^{\beta \rho}}{\rho}\int_0^1\tau_1^{1-1}(1-\tau_1)^{(1+\beta)-1-1}
			(1-t^\rho\tau_1)^{-(1-\beta)}d\tau_1\\
			&=&\frac{t^{\beta \rho}}{\rho}.\frac{\Gamma((1+\beta)-1)}{\Gamma(1)\Gamma(1+\beta)}
			{_2F_1}(1,1-\beta;1+\beta;t^\rho)\\
			&=&\frac{t^{\beta \rho}}{\rho}.\frac{\Gamma(\beta)}{\Gamma(1+\beta)}.
			{_2F_1}(1,1-\beta;1+\beta;t^\rho)\\
			&=&\frac{t^{\beta \rho}}{\beta\rho}.
			{_2F_1}(1,1-\beta;1+\beta;t^\rho).
		\end{eqnarray*}
		Similarly,
		\begin{eqnarray*}
			\int_0^t(t^\rho-\tau^\rho)^{\beta-1}			\tau^{\rho-1}(1-\tau^\rho)^{\beta}d\tau
&=&\int_0^1(t^\rho-t^\rho\tau_1)^{\beta-1}
			(1-t^\rho\tau_1)^{\beta}.\frac{t^\rho}{\rho}d\tau_1\\
			&=& \frac{t^{\beta \rho}}{\rho}\int_0^1(1-\tau_1)^{\beta}
			(1-t^\rho\tau_1)^{\beta}d\tau_1\\
			&=& \frac{t^{\beta \rho}}{\rho}\int_0^1\tau_1^{1-1}(1-\tau_1)^{(1+\beta)-1-1}
			(1-t^\rho\tau_1)^{-(-\beta)}d\tau_1\\
			&=&\frac{t^{\beta \rho}}{\rho}.\frac{\Gamma((1+\beta)-1)}{\Gamma(1)\Gamma(1+\beta)}
			{_2F_1}(1,-\beta;1+\beta;t^\rho)\\
			&=&\frac{t^{\beta \rho}}{\rho}.\frac{\Gamma(\beta)}{\Gamma(1+\beta)}.
			{_2F_1}(1,-\beta;1+\beta;t^\rho)\\
			&=&\frac{t^{\beta \rho}}{\beta\rho}.
			{_2F_1}(1,-\beta;1+\beta;t^\rho).
		\end{eqnarray*}
		\noindent(ii) Let $\beta>\frac{1}{2}$. Then, $c-b-a=2\beta-1>0$ and by Remark \ref{rem2.7} (iv),
		\begin{eqnarray*}			u_1(1)=\frac{1}{\beta\rho}{_2F_1}(1,1-\beta;1+\beta;1) &=&\frac{1}{\beta\rho}.\frac{1+\beta-1}{1+\beta-(1-\beta)-1}\\
			&=&\frac{1}{(2\beta-1)\rho},
		\end{eqnarray*}
		and
		\begin{eqnarray*}
			u_2(1)=\frac{1}{\beta\rho}{_2F_1}(1,-\beta;1+\beta;1)
&=&\frac{1}{\beta\rho}.\frac{1+\beta-1}{1+\beta-(-\beta)-1}\\
			&=&\frac{1}{2\beta\rho}.
		\end{eqnarray*}
		\noindent (iii) From the definition and $\frac{(1)_n}{n!}=\frac{n!}{n!}=1$, we get that
			\begin{eqnarray*}
			\int_0^1u_1(t)dt&=&\frac{1}{\beta\rho}\int_0^1
t^{\beta\rho}{_2F_1}(1,1-\beta;1+\beta;t^\rho)dt\\
			&=&\frac{1}{\beta\rho}\int_0^1
t^{\beta\rho}\left(\sum\limits_{n=0}^{\infty}
\frac{(1)_n(1-\beta)_n}{(1+\beta)_n}.
\frac{(t^\rho)^n}{n!}\right)dt\\
	&=&\frac{1}{\beta\rho}\int_0^1
\sum\limits_{n=0}^{\infty}
\frac{(1)_n(1-\beta)_n}{(1+\beta)_n}.
\frac{t^{\rho n+\beta\rho}}{n!}dt	\\
&=&\frac{1}{\beta\rho}.\sum\limits_{n=0}^{\infty}
\frac{(1)_n(1-\beta)_n}{(1+\beta)_n n!}.
\left.\frac{t^{\rho( n+\beta)+1}}{\rho( n+\beta)+1}
			\right|_0^{1} \\
&=&\frac{1}{\beta\rho}.\sum\limits_{n=0}^{\infty}
\frac{(1)_n(1-\beta)_n}{(1+\beta)_n n!}.
\frac{1}{\rho( n+\beta+\frac{1}{\rho})}	\\
&=&\frac{1}{\beta\rho(\beta\rho+1)}.\sum\limits_{n=0}^{\infty}
\frac{(1)_n(1-\beta)_n(\beta+\frac{1}{\rho})_n}{(1+\beta)_n
(\beta+1+\frac{1}{\rho})_n}.
\frac{1}{n!}\\
&=&	\frac{1}{\beta\rho(\beta\rho+1)}. {_3F_2}(1,1-\beta,\beta+\frac{1}{\rho};1+\beta,
\beta+\frac{1}{\rho}+1;1),
		\end{eqnarray*}
since
$$\frac{1}{\rho( n+\beta+\frac{1}{\rho})}=\frac{1}{\rho(\beta+\frac{1}{\rho})}.
\frac{\beta+\frac{1}{\rho}}{n+\beta+\frac{1}{\rho}}=
\frac{1}{\beta\rho +1}.
\frac{(\beta+\frac{1}{\rho})_n}{(\beta+1+\frac{1}{\rho})_n}.$$
Similarly,
		\begin{eqnarray*}
			\int_0^1u_2(t)dt&=&\frac{1}{\beta\rho}\int_0^1
t^{\beta\rho}{_2F_1}(1,-\beta;1+\beta;t^\rho)dt\\
			&=&\frac{1}{\beta\rho}\int_0^1
\sum\limits_{n=0}^{\infty}
\frac{(1)_n(-\beta)_n}{(1+\beta)_n}.
\frac{t^{\rho n+\beta\rho}}{n!}dt\\
&=&	\frac{1}{\beta\rho(\beta\rho+1)}. {_3F_2}(1,-\beta,\beta+\frac{1}{\rho};1+\beta,\beta+\frac{1}{\rho};1).
		\end{eqnarray*}
	\end{proof}
\section{Granular fractional Caputo-Katugampola derivatives}
	In the line of \cite{AMO16,A17,K11,NZ17}, we propose the granular  KRL integral, granular  KRL fractional derivatives and granular  CK fractional derivatives as follows.

\begin{definition}
		Let $\widetilde{\phi}:[a,b]\to \mathbf{F}_C(\mathbb{R})$ be of class $\mathcal{C}([a,b])$ and $0<\beta<1, \rho>0$.
		\begin{enumerate}
			\item[(i)] The left (right) granular  KRL fractional integral of $\phi$ of order $\beta$ is
			$${_{gr}\mathcal{I}}^{\beta,\rho}_{a^+}\widetilde{\phi}(t):=
			\frac{\rho^{1-\beta}}{\Gamma(\beta)}
			\widetilde{\int}_a^t
\frac{\tau^{\rho-1}}{(t^{\rho}-\tau^{\rho})^{1-\beta}}
\widetilde{\phi}(\tau)d\tau, t\ge a$$
			$$\left({_{gr}\mathcal{I}}^{\beta,\rho}_{b^-}\widetilde{\phi}(t):=
			\frac{\rho^{1-\beta}}{\Gamma(\beta)}
			\widetilde{\int}_t^b
\frac{\tau^{\rho-1}}{(\tau^{\rho}-t^{\rho})^{1-\beta}}
\widetilde{\phi}(\tau)d\tau, t\le b\right).$$
			\item[(ii)] The left (right) granular  KRL fractional derivative of $\phi$ of order $\beta$ is defined by
			$${_{gr}\mathcal{D}}^{\beta,\rho}_{a^+}\widetilde{\phi}(t):=
\frac{\rho^{\beta}}{\Gamma(1-\beta)}
t^{1-\rho}.				\frac{d_{gr}}{dt}\left(\widetilde{\int}_a^t
\frac{\tau^{\rho-1}}{(t^{\rho}-\tau^{\rho})^{1-\beta}}
			\widetilde{\phi}(\tau)d\tau\right), t\ge a$$
			$$\left({_{gr}\mathcal{D}}^{\beta,\rho}_{b^-}\widetilde{\phi}(t):=
-\frac{\rho^{\beta}}{\Gamma(1-\beta)}
t^{1-\rho}.		
\frac{d_{gr}}{dt}\left(\widetilde{\int}_t^b
\frac{\tau^{\rho-1}}{(\tau^{\rho}-t^{\rho})^{1-\beta}}
			\widetilde{\phi}(\tau)\right), t\le b\right).$$
			\item[(iii)] Let $\phi\in \mathcal{C}^1([a,b])$.  The left (right) Caputo fractional derivative of $\phi$ of order $\beta$ is
			$${_{gr}^C\mathcal{D}}^{\beta,\rho}_{a^+}\widetilde{\phi}(t):=
			\frac{\rho^{\beta}}{\Gamma(1-\beta)}
			\widetilde{\int}_a^t			\frac{1}{(t^{\rho}-\tau^{\rho})^{\beta}}\widetilde{\dot{\phi}}(\tau)d\tau, t\ge a$$
			$$\left({_{gr}^C\mathcal{D}}^{\beta,\rho}_{b^-}
\widetilde{\phi}(t):=
			-\frac{\rho^{\beta}}{\Gamma(1-\beta)}
			\widetilde{\int}_t^b
   \frac{1}{(\tau^{\rho}-t^{\rho})^{\beta}}
			\widetilde{\dot{\phi}}(\tau)d\tau, t\le b\right).$$
		\end{enumerate}
	\end{definition}
\begin{remark}
			 When $\rho=1$, the above definitions become the definitions of granular  RL integral, granular  RL derivative and granular  Caputo derivative in \cite{NZ17}.

	\end{remark}
By the similar proof to the proof of \cite{NZ17}, we have the following properties.
	\begin{remark}
		\begin{enumerate}
			\item[(i)] $\mathcal{H}({_{gr}\mathcal{I}}^{\beta,\rho}_{a^+}\widetilde{\phi}(t))
			=I_{a^+}^{\beta,\rho}\phi^{gr}(t,\beta_{\phi}),
			\mathcal{H}({_{gr}\mathcal{I}}^{\beta,\rho}_{b^-}\widetilde{\phi}(t))
			=I_{b^-}^{\beta,\rho}\phi^{gr}(t,\beta_{\phi})$.
			\item[(ii)] $\mathcal{H}({_{gr}\mathcal{D}}^{\beta,\rho}_{a^+}\widetilde{\phi}(t))=
			D^{\beta,\rho}_{a^+}
			\phi^{gr}(t,\beta_{\phi}), \mathcal{H}({_{gr}\mathcal{D}}^{\beta,\rho}_{b^-}\widetilde{\phi}(t))=
			D^{\beta,\rho}_{b^-}
			\phi^{gr}(t,\beta_{\phi})$.
			\item[(iii)] $\mathcal{H}({_{gr}^C\mathcal{D}}^{\beta,\rho}_{a^+}\widetilde{\phi}(t))=
			{^CD}^{\beta,\rho}_{a^+}
			\dot{\phi}^{gr}(t,\beta_{\phi}), \mathcal{H}({_{gr}^C\mathcal{D}}^{\beta,\rho}_{b^-}\widetilde{\phi}(t))=
			{^CD}^{\beta,\rho}_{b^-}
			\dot{\phi}^{gr}(t,\beta_{\phi})$.
		\end{enumerate}
	\end{remark}

Using the granular approach to consider the fuzzy fractional derivatives has some advantages
for other approaches in some cases. Firstly, the calculus rules of the fuzzy fractional
derivatives have some advantages in some cases, see e.g. \cite{BGMNST24}, \cite{DLK20}, \cite{HP24} ,\cite{MPK17},  \cite{TGBMH25}. Moreover, the granular fractional derivatives could be existed in some
cases that $gH$-derivatives do not exist, see e.g. \cite{MPK17,NZ17}. The following example shows that our approach based on granular derivatives has advantages over the approach using $gH$-derivatives in \cite{HVD19} in some cases.
\begin{example} (i) Consider following triangular fuzzy function in $t\in [0,2]$, see \cite{MPK17,NZ17} and references therein,
$$\tilde{\phi}(t)=\left(\frac{t^3}{3},\frac{t^3}{3}+t+3,
\frac{2t^3}{3}+4\right).$$
Hence, the parametric form of $\tilde{\phi}(t)$ is
$$[\tilde{\phi}(t)]^\mu=[\phi^L_\mu(t),\phi^R_\mu(t)]
=\left[\frac{1}{3}t^3+\mu t+3\mu,
\left(-\frac{1}{3}\mu+\frac{2}{3}\right)t+\mu t+4-\mu\right].$$
We get that, for each $\mu\in [0,1]$, $\tilde{\phi}(t)$ is $d$-monotone in $[0,2]$, since
$$diam([\tilde{\phi}(t)]^\mu)=\phi^R_\mu(t)-\phi^L_\mu(t)
=\left(-\frac{\mu}{3}+\frac{1}{3}\right)t^3+4-3\mu$$
has $\frac{d}{dt}\left(diam([\tilde{\phi}(t)]^\mu)\right)
=(1-\mu)t^2\ge 0$ for all $t\in [0,2]$. Moreover, for each $\mu\in [0,1]$, $\phi^L_\mu(t)$ and $\phi^R_\mu(t)$ are in $AC([0,2],\mathbb{R})$,
leading that $\tilde{\phi}(t)\in AC([0,2],\mathbf{F}_C(\mathbb{R}))$,
see \cite{L15}. However, $gH$-derivative of $\tilde{\phi}(t)$ does not exist (see  \cite{MPK17,NZ17}). Hence, we can not utilize Theorem 2.4 in \cite{HVD19} to calculate both left and right $gH$-CK fractional derivatives of $\tilde{\phi}(t)$.

\noindent (ii)   Now, since
$$\mathcal{H}(\tilde{\phi}(t))
=\phi^{gr}(t,\mu,\alpha_{\phi})=\frac{1}{3}t^3+\mu t+3\mu
+\left(\left(-\frac{\mu}{3}+\frac{1}{3}\right)t^3+4
-3\mu\right)\alpha_{\phi},$$
we get that
$$\frac{\phi^{gr}}{dt}=\dot{\phi}^{gr}=\left(
(1-\mu)\alpha_{\phi}+1\right)t^2+\mu.$$

For $\beta\in (0,1)$ and $\rho>0$,
\begin{eqnarray*}
\mathcal{H}\left({_{gr}^C\mathcal{D}}^{\beta,\rho}_{0^+}
\widetilde{\phi}(t)\right)
&=&{^CD}^{\beta,\rho}_{0^+}
			\dot{\phi}^{gr}(t,\beta_{\phi})\\
&=&\frac{\rho^{\beta}}{\Gamma(1-\beta)}\left(
\left((1-\mu)\alpha_{\phi}+1\right)
\int_0^t \frac{\tau^2}{(t^{\rho}-\tau^\rho)^\beta}d\tau + \mu
\int_0^t \frac{1}{(t^{\rho}-\tau^\rho)^\beta}d\tau \right)\\
&=&\frac{\rho^{\beta}}{\Gamma(1-\beta)}
\left(\left((1-\mu)\alpha_{\phi}+1\right)I_1 + \mu I_2 \right).
\end{eqnarray*}
By setting $u=\frac{\tau^\rho}{t^\rho}$, one gets that $\tau=t.u^{1/\rho}, d\tau=\frac{t}{\rho}u^{\frac{1}{\rho}-1}du$,
 $\tau=0\rightarrow u=0, \tau=t\rightarrow u=1$, and hence,
 $$I_1=\int_0^t \frac{\tau^2}{(t^{\rho}-\tau^\rho)^\beta}d\tau
 =\int_0^1 \frac{(t.u^{1/\rho})^2}{(t^\rho-t^\rho.u)^\beta}.
 \frac{t}{\rho}u^{\frac{1}{\rho}-1}du$$
 $$=\frac{t^{3-\rho\beta}}{\rho}\int_0^1 u^{\frac{3}{\rho}-1}(1-u)^{-\beta}du
 =\frac{t^{3-\rho\beta}}{\rho}B(\frac{3}{\rho},1-\beta),$$
 $$I_2=\int_0^t \frac{1}{(t^{\rho}-\tau^\rho)^\beta}d\tau=
\frac{t^{1-\rho\beta}}{\rho}B(\frac{1}{\rho},1-\beta).$$
Thus,
\begin{eqnarray*}
\mathcal{H}\left({_{gr}^C\mathcal{D}}^{\beta,\rho}_{0^+}
\widetilde{\phi}(t)\right)&=&\frac{\rho^{\beta}}{\Gamma(1-\beta)}
\left(\left((1-\mu)\alpha_{\phi}+1\right)
\frac{t^{3-\rho\beta}}{\rho}B(\frac{3}{\rho},1-\beta)
+\mu\frac{t^{1-\rho\beta}}{\rho}B(\frac{1}{\rho},1-\beta)\right)\\
&=&\frac{\rho^{\beta-1}}{\Gamma(1-\beta)}
\left(\left((1-\mu)\alpha_{\phi}+1\right)
B(\frac{3}{\rho},1-\beta).t^{3-\rho\beta}
+\mu B(\frac{1}{\rho},1-\beta).t^{1-\rho\beta}\right).
\end{eqnarray*}
Setting
$$f(\alpha_{\phi}):=\frac{\rho^{\beta-1}}{\Gamma(1-\beta)}
\left(\left((1-\mu)\alpha_{\phi}+1\right)
B(\frac{3}{\rho},1-\beta).t^{3-\rho\beta}
+\mu B(\frac{1}{\rho},1-\beta).t^{1-\rho\beta}\right),$$
one gets that $f(\alpha_{\phi})=a_1.\alpha_{\phi}+b_1$ is a linear function with
$$a_1=\frac{\rho^{\beta-1}}{\Gamma(1-\beta)}(1-\gamma)
B(\frac{3}{\rho},1-\beta).t^{3-\rho\beta}\ge 0$$
for all $t\in [0,2]$. Therefore, for each $t\in [0,2]$,
$$h^L(\gamma)=\min\limits_{\alpha_{\phi}}		\phi^{gr}(t,\gamma,\alpha_{\phi})=f(0)
=\frac{\rho^{\beta-1}}{\Gamma(1-\beta)}
\left(B(\frac{3}{\rho},1-\beta)t^{3-\rho\beta}
+\gamma B(\frac{1}{\rho},1-\beta)t^{1-\rho\beta}\right),$$
$$h^R(\gamma)=\max\limits_{\alpha_{\phi}}		\phi^{gr}(t,\gamma,\alpha_{\phi})=f(1)
=\frac{\rho^{\beta-1}}{\Gamma(1-\beta)}
\left((2-\gamma)
B(\frac{3}{\rho},1-\beta)t^{3-\rho\beta}
+\gamma B(\frac{1}{\rho},1-\beta)t^{1-\rho\beta}\right).$$
Similarly, we deduce from
$\frac{d}{d\gamma}\left(h^L(\gamma)\right)\ge 0$ that
$$\inf\limits_{\gamma\ge \mu}\min\limits_{\alpha_{\phi}}		\phi^{gr}(t,\gamma,\alpha_{\phi})=h^L(\mu)
=\frac{\rho^{\beta-1}}{\Gamma(1-\beta)}
\left(B(\frac{3}{\rho},1-\beta)t^{3-\rho\beta}
+\mu B(\frac{1}{\rho},1-\beta)t^{1-\rho\beta}\right).$$
On the others hand,
\begin{eqnarray*}
\frac{d}{d\gamma}\left(h^R(\gamma)\right)
&=&\frac{\rho^{\beta-1}}{\Gamma(1-\beta)}
\left(-
B(\frac{3}{\rho},1-\beta)t^{3-\rho\beta}
+B(\frac{1}{\rho},1-\beta)t^{1-\rho\beta}\right)\ge 0\\
&\Leftrightarrow& B(\frac{3}{\rho},1-\beta)t^{3-\rho\beta}\le
B(\frac{1}{\rho},1-\beta)t^{1-\rho\beta}\\
&\Leftrightarrow& t^2\le \frac{B(\frac{1}{\rho},1-\beta)}
{B(\frac{3}{\rho},1-\beta)}
\Leftrightarrow t\le\sqrt{\frac{B(\frac{1}{\rho},1-\beta)}
{B(\frac{3}{\rho},1-\beta)}} (t\in [0,2]).
\end{eqnarray*}
Setting $r=\frac{1}{\rho}\in (0,\infty)$ and $f(\beta,r):=\frac{B(r,1-\beta)}
{B(3r,1-\beta)}$, we get that
$$f(\beta,r)=\frac{\Gamma(r)\Gamma(1-\beta)}
{\Gamma(r+1-\beta)}.
\frac{\Gamma(3r+1-\beta)}
{\Gamma(3r)\Gamma(1-\beta)}=
\frac{\Gamma(r)}
{\Gamma(3r)}.
\frac{\Gamma(3r+1-\beta)}
{\Gamma(r+1-\beta)}$$
$$\Rightarrow \ln f(\beta,r)=\ln \Gamma(3) -\ln \Gamma(3r) +\ln \Gamma(3r+1-\beta)-
\ln \Gamma(r+1-\beta) $$
$$\Rightarrow \frac{\partial f(\beta,r)}{\partial \beta}=
f(\beta,r)\left(\psi(3r+1-\beta)-
\psi(r+1-\beta)\right)>0,$$
since the digamma function $\psi$ is increasing on $(0,+\infty)$ and $r+1-\beta<3r+1-\beta$.
Hence, $f(\beta,r)$ has no stationary point and its supremum obtains on the boundary of $D=\{(\beta,r) \mid 0<\beta<1,r>0\}$.

If $\rho\to 0^+$ or $r\to\infty$, we deduce from Remark \ref{rem2.6} (ii) that
$$\lim\limits_{r\to\infty}f(\beta,r)=\lim\limits_{r\to\infty}
\frac{\Gamma(r)}{\Gamma(3r)}.\frac{\Gamma(3r+1-\beta)}
{\Gamma(r+1-\beta)}=\frac{(3r)^{1-\beta}}{r^{1-\beta}}=3^{1-\beta}<3.$$

If $\rho\to \infty$ or $r\to 0^+$, we deduce from
$\lim\limits_{x\to 0^+}x\Gamma(x)=\lim\limits_{x\to 0^+}\Gamma(x+1)
=\Gamma(1)=1$ that
$$\lim\limits_{r\to 0^+}f(\beta,r)=\lim\limits_{r\to 0^+}
\frac{\Gamma(r)}{\Gamma(3r)}.\frac{\Gamma(3r+1-\beta)}
{\Gamma(r+1-\beta)}=
\lim\limits_{r\to 0^+}
\frac{3r}{r}.\frac{r+1-\beta}
{3r+1-\beta}=\frac{(3r)^{1-\beta}}{r^{1-\beta}}=1.$$

 Hence, for all $(\beta,\rho)$, one gets $\sqrt{\frac{B(\frac{1}{\rho},1-\beta)}
{B(\frac{3}{\rho},1-\beta)}}
=\sqrt{f(\beta,\frac{1}{\rho})}<\sqrt{3}<2$
implying that
$$\sup\limits_{\gamma\ge \mu}\max\limits_{\alpha_{\phi}}		\phi^{gr}(t,\gamma,\alpha_{\phi})=\left\{\begin{array}{c}
h^R(1),\mbox{if}\; 0\le t\le \sqrt{\frac{B(\frac{1}{\rho},1-\beta)}
{B(\frac{3}{\rho},1-\beta)}}  \\
h^R(\mu), \mbox{if}\; \sqrt{\frac{B(\frac{1}{\rho},1-\beta)}
{B(\frac{3}{\rho},1-\beta)}}\le t\le 2.
\end{array}
\right.$$
Hence, ${_{gr}^C\mathcal{D}}^{\beta,\rho}_{0^+}
\widetilde{\phi}(t)$ exists and its $\mu$-level sets are
$${\scriptsize\left[{_{gr}^C\mathcal{D}}^{\beta,\rho}_{0^+}
\widetilde{\phi}(t)\right]^\mu=\left\{\begin{array}{c}
\left[\frac{\rho^{\beta-1}}{\Gamma(1-\beta)}
\left(B(\frac{3}{\rho},1-\beta)t^{3-\rho\beta}
+\mu B(\frac{1}{\rho},1-\beta)t^{1-\rho\beta}\right),
h^R(1)\right],\mbox{if}\; 0\le t\le \sqrt{\frac{B(\frac{1}{\rho},1-\beta)}
{B(\frac{3}{\rho},1-\beta)}}  \\
\left[\frac{\rho^{\beta-1}}{\Gamma(1-\beta)}
\left(B(\frac{3}{\rho},1-\beta)t^{3-\rho\beta}
+\mu B(\frac{1}{\rho},1-\beta)t^{1-\rho\beta}\right),
h^R(\mu)\right], \mbox{if}\; \sqrt{\frac{B(\frac{1}{\rho},1-\beta)}
{B(\frac{3}{\rho},1-\beta)}}\le t\le 2.
\end{array}
\right.}$$
\begin{figure}[htbp]
    \centering
    \begin{tabular}{cc}

        \includegraphics[width=0.45\textwidth]{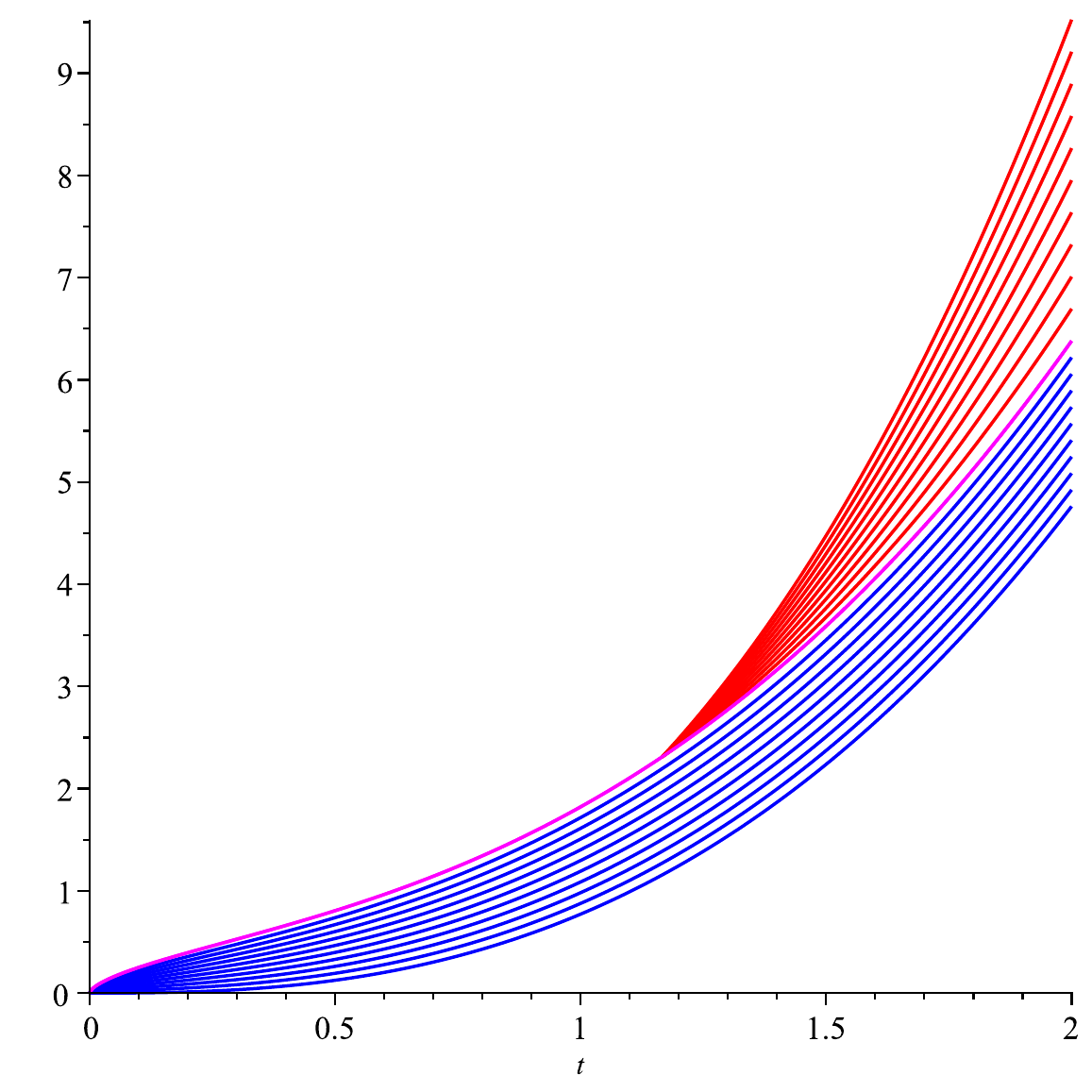} &
        \includegraphics[width=0.45\textwidth]{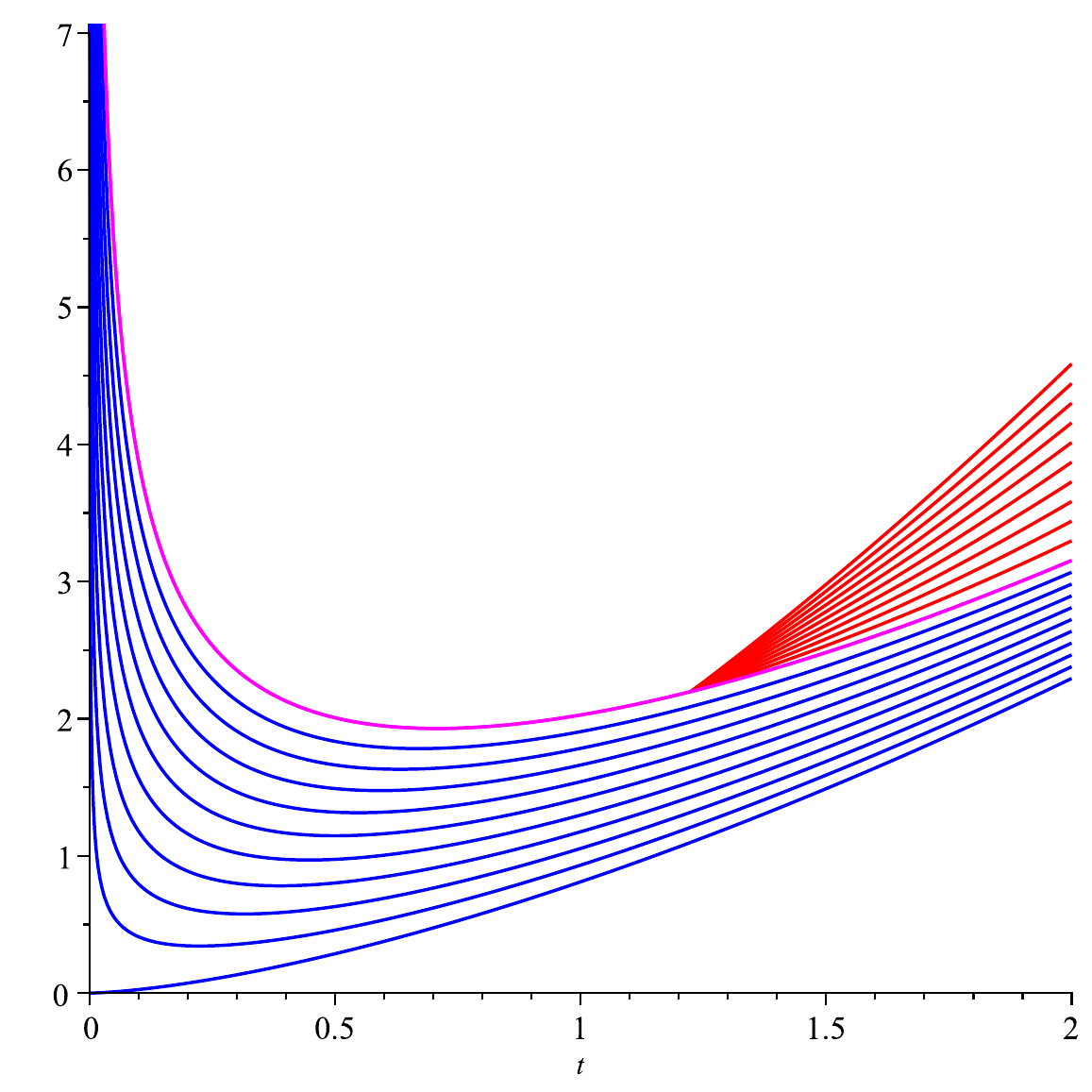} \\

        Figure 2.1.1. ${_{gr}^C\mathcal{D}}^{\frac{3}{4},\frac{1}{2}}_{0^+}
\widetilde{\phi}(t)$&
        Figure 2.1.2. ${_{gr}^C\mathcal{D}}^{\frac{3}{4},2}_{0^+}
\widetilde{\phi}(t)$ \\
    \end{tabular}
\end{figure}
It should be noted that when $\rho=1$, we get that
$$\sqrt{\frac{B(1,1-\beta)}
{B(3,1-\beta)}}=
\sqrt{\frac{\Gamma(1)\Gamma(1-\beta)}{\Gamma(2-\beta)}.
\frac{\Gamma(4-\beta)}{\Gamma(3)\Gamma(1-\beta)}}=
\sqrt{\frac{(3-\beta)(2-\beta)}{2}}.
$$
This implies that
$${\left[{_{gr}^C\mathcal{D}}^{\beta,1}_{0^+}
\widetilde{\phi}(t)\right]^\mu=\left\{\begin{array}{c}
\left[\frac{2t^{3-\beta}}{\Gamma(4-\beta)}+
\frac{\mu t^{1-\beta}}{\Gamma(2-\beta)},
\frac{2t^{3-\beta}}{\Gamma(4-\beta)}+
\frac{t^{1-\beta}}{\Gamma(2-\beta)}\right],\mbox{if}\; 0\le t\le \sqrt{\frac{(3-\beta)(2-\beta)}{2}}  \\
\left[\frac{2t^{3-\beta}}{\Gamma(4-\beta)}+
\frac{\mu t^{1-\beta}}{\Gamma(2-\beta)},
\frac{2(2-\mu)t^{3-\beta}}{\Gamma(4-\beta)}+
\frac{\mu t^{1-\beta}}{\Gamma(2-\beta)}\right], \mbox{if}\; \sqrt{\frac{(3-\beta)(2-\beta)}{2}}\le t\le 2,
\end{array}
\right.}$$
which coincides with the results in Example 2 in \cite{NZ17}.

\noindent (iii)  For $\beta\in (0,1)$ and $\rho>0$,
\begin{eqnarray*}
\mathcal{H}\left({_{gr}^C\mathcal{D}}^{\beta,\rho}_{2^-}
\widetilde{\phi}(t)\right)
&=&{^CD}^{\beta,\rho}_{2^-}
			\dot{\phi}^{gr}(t,\beta_{\phi})\\
&=&-\frac{\rho^{\beta}}{\Gamma(1-\beta)}\left(
\left((1-\mu)\alpha_{\phi}+1\right)
\int_t^2 \frac{\tau^2}{(\tau^\rho-t^{\rho})^\beta}d\tau + \mu
\int_t^2 \frac{1}{(\tau^\rho-t^{\rho})^\beta}d\tau \right)\\
&=&-\frac{\rho^{\beta}}{\Gamma(1-\beta)}
\left(\left((1-\mu)\alpha_{\phi}+1\right)J_1 + \mu J_2 \right).
\end{eqnarray*}
By setting $\tau^\rho=2^\rho-(2^\rho-t^\rho)u$, one gets that $\tau^\rho-t^\rho=(2^\rho-t^\rho)(1-u)$, $$\rho\tau^{\rho-1}d\tau=-(2^\rho-t^\rho)du \rightarrow d\tau=-\frac{2^\rho-t^\rho}{\rho\tau^{\rho-1}}du=
-\frac{2^\rho-t^\rho}
{\rho\left(2^\rho-(2^\rho-t^\rho)u\right)^{\frac{\rho-1}{\rho}}}du$$
 and $\tau=t\rightarrow u=1, \tau=2\rightarrow u=0$. Hence,
 \begin{eqnarray*}
J_1&=&\int_t^2 \frac{\tau^2}{(\tau^\rho-t^{\rho})^\beta}d\tau
 =\int_0^1 \frac{\left(2^\rho-(2^\rho-t^\rho)u\right)^{\frac{2}{\rho}}}
 {((2^\rho-t^\rho)(1-u))^\beta}.
 \frac{2^\rho-t^\rho}
{\rho\left(2^\rho-(2^\rho-t^\rho)u\right)^{\frac{\rho-1}{\rho}}}du\\
&=&\frac{(2^\rho-t^\rho)^{1-\beta}}{\rho}\int_0^1 (1-u)^{-\beta}(2^\rho-(2^\rho-t^\rho)u)^{\frac{3}{\rho}-1}du\\
 &=&\frac{(2^\rho-t^\rho)^{1-\beta}}{\rho}.2^{3-\rho}\int_0^1 (1-u)^{-\beta}
 \left(1-\left(1-\frac{t^\rho}{2^\rho}\right)u\right)^{-
 \left(1-\frac{3}{\rho}\right)}du\\
 &=&\frac{(2^\rho-t^\rho)^{1-\beta}}{\rho}.2^{3-\rho}\int_0^1 u^{1-1}(1-u)^{(2-\beta)-1-1}
 \left(1-\left(1-\frac{t^\rho}{2^\rho}\right)u\right)^{-
 \left(1-\frac{3}{\rho}\right)}du\\
 &=&\frac{(2^\rho-t^\rho)^{1-\beta}}{\rho}2^{3-\rho}.
 \frac{\Gamma(2-\beta-1)}{\Gamma(1)\Gamma(2-\beta)}
 {_2F_1}\left(1,1-\frac{3}{\rho};2-\beta;
 1-\frac{t^\rho}{2^\rho}\right)\\
 & =&\frac{(2^\rho-t^\rho)^{1-\beta}}{\rho(1-\beta)}2^{3-\rho}.
 {_2F_1}\left(1,1-\frac{3}{\rho};2-\beta;
 1-\frac{t^\rho}{2^\rho}\right).
\end{eqnarray*}
Similarly,
$$J_2=\int_t^2 \frac{1}{(\tau^\rho-t^{\rho})^\beta}d\tau=
\frac{(2^\rho-t^\rho)^{1-\beta}}{\rho(1-\beta)}2^{1-\rho}.
 {_2F_1}\left(1,1-\frac{1}{\rho};2-\beta;
 1-\frac{t^\rho}{2^\rho}\right),$$
and hence,
$$\mathcal{H}\left({_{gr}^C\mathcal{D}}^{\beta,\rho}_{2^-}
\widetilde{\phi}(t)\right)
=A\left(4\left((1-\mu)\alpha_{\phi}+1\right).
 {_2F_1}(1,1-\frac{3}{\rho};2-\beta;
 1-\frac{t^\rho}{2^\rho}) + \mu .
 {_2F_1}(1,1-\frac{1}{\rho};2-\beta;
 1-\frac{t^\rho}{2^\rho})\right).$$
where $A:=-\frac{\rho^{\beta-1}(2^\rho-t^\rho)^{1-\beta}2^{1-\rho}}
{(1-\beta)\Gamma(1-\beta)}=
-\frac{\rho^{\beta-1}(2^\rho-t^\rho)^{1-\beta}2^{1-\rho}}
{\Gamma(2-\beta)}\le 0$. Setting
$$g(\alpha_{\phi}):=A
\left(4\left((1-\gamma)\alpha_{\phi}+1\right).
 {_2F_1}(1,1-\frac{3}{\rho};2-\beta;
 1-\frac{t^\rho}{2^\rho}) + \gamma .
 {_2F_1}(1,1-\frac{1}{\rho};2-\beta;
 1-\frac{t^\rho}{2^\rho})\right),$$
one gets that $g(\alpha_{\phi})=a_2.\alpha_{\phi}+b_2$ is a linear function with
$$a_2=A.4(1-\gamma){_2F_1}(1,1-\frac{3}{\rho};2-\beta;
 1-\frac{t^\rho}{2^\rho})\le 0$$
for all $t\in [0,2]$, since
$${_2F_1}(1,1-\frac{3}{\rho};2-\beta;
 1-\frac{t^\rho}{2^\rho})=
 \frac{\Gamma(1)\Gamma(2-\beta)}{\Gamma(1-\beta)}\int_0^1 (1-u)^{-\beta}
 \left(1-\left(1-\frac{t^\rho}{2^\rho}\right)u\right)^{-
 \left(1-\frac{3}{\rho}\right)}du\ge 0$$
when $0<u<1, 1-u>0,0<1-\frac{t^\rho}{2^\rho}<1, 1-\left(1-\frac{t^\rho}{2^\rho}\right)u>0$.

Therefore, for each $t\in [0,2]$,
$$l^L(\gamma)=\min\limits_{\alpha_{\phi}}		\phi^{gr}(t,\gamma,\alpha_{\phi})=g(1)={\scriptsize A
{[}4(2-\gamma)
 {_2F_1}(1,1-\frac{3}{\rho};2-\beta;
 1-\frac{t^\rho}{2^\rho}) + \gamma
 {_2F_1}(1,1-\frac{1}{\rho};2-\beta;
 1-\frac{t^\rho}{2^\rho}){]}}$$
$$l^R(\gamma)=\max\limits_{\alpha_{\phi}}		\phi^{gr}(t,\gamma,\alpha_{\phi})=g(0)=A
{[}4{_2F_1}(1,1-\frac{3}{\rho};2-\beta;
 1-\frac{t^\rho}{2^\rho}) + \gamma
 {_2F_1}(1,1-\frac{1}{\rho};2-\beta;
 1-\frac{t^\rho}{2^\rho}){]}.$$
 We deduce from $\frac{d}{d\gamma}\left(l^R(\gamma)\right)\le 0$ that
 $$\sup\limits_{\gamma\ge \mu}\max\limits_{\alpha_{\phi}}		\phi^{gr}(t,\gamma,\alpha_{\phi})=l^R(\mu)=A{\scriptsize
\left(4{_2F_1}(1,1-\frac{3}{\rho};2-\beta;
 1-\frac{t^\rho}{2^\rho}) + \mu
 {_2F_1}(1,1-\frac{1}{\rho};2-\beta;
 1-\frac{t^\rho}{2^\rho})\right)}.$$
The sign of expression
$$\frac{d}{d\gamma}\left(l^L(\gamma)\right)
=A
\left(-4.
 {_2F_1}(1,1-\frac{3}{\rho};2-\beta;
 1-\frac{t^\rho}{2^\rho}) +
 {_2F_1}(1,1-\frac{1}{\rho};2-\beta;
 1-\frac{t^\rho}{2^\rho})\right)$$
 depends on $\rho$, hence, it is difficult to analyze the general case $\rho>0$.

 However, we could consider the sign of  $\frac{d}{d\gamma}\left(l^L(\gamma)\right)$ for any $\rho=\bar{\rho}>0$. For simplicity, we only consider 3 cases of $\rho =\bar{\rho}$: $\bar{\rho}=1,
 \bar{\rho}=2,\bar{\rho}=1/2$.

\noindent $*${\it Case 1:} $\bar{\rho}=1$. Then, $A=-\frac{(2-t)^{1-\beta}}
{\Gamma(2-\beta)}\le 0$ and, since $(-2)_n=0 \; (n\ge 3)$,
\begin{eqnarray*}
\frac{d}{d\gamma}\left(l^L(\gamma)\right)
&=&A\left(-4
 {_2F_1}(1,-2;2-\beta;
 1-\frac{t}{2}) +
 {_2F_1}(1,0;2-\beta;
 1-\frac{t}{2})\right)\\
&=&A\left(-4\sum\limits_{n=0}^2\frac{(1)_n(-2)_n}{(2-\beta)_n}.
\frac{1}{n!}\left(1-\frac{t}{2}\right)^n +
 1\right)\\
&=&A\left(-4\left(1+\frac{-2}{2-\beta}\left(1-\frac{t}{2}\right)
+\frac{2}{(2-\beta)(3-\beta)}\left(1-\frac{t}{2}\right)^2\right) +
 1\right)\\
&=&\frac{A}{(2-\beta)(3-\beta)}
(-2t^2+4(\beta-1)t-3\beta^2+7\beta-2)\\
&=&\frac{-2A}{(2-\beta)(3-\beta)}.
\left(t-\frac{4\beta-4+\sqrt{8(-\beta^2+3\beta)}}{4}\right)
\left(t-\frac{4\beta-4-\sqrt{8(-\beta^2+3\beta)}}{4}\right)\\
&=&\frac{-2A}{(2-\beta)(3-\beta)}.\left(t-t_1\right)
\left(t-t_2\right),\;\mbox{where}\; t_{1,2}=\frac{4\beta-4\pm\sqrt{8(-\beta^2+3\beta)}}{4}.
\end{eqnarray*}
Since $t_2\le 0$ for all $\beta\in (0,1)$, $t_1\le 0$ when $\beta\le \frac{1}{3}$ and $t_1\in (0,1)$ when $\frac{1}{3}<\beta<1$, one has
$$\inf\limits_{\gamma\ge \mu}\min\limits_{\alpha_{\phi}}		u^{gr}(t,\gamma,\alpha_{\phi})=\left\{\begin{array}{c}
l^L(\mu),\mbox{if}\; 0\le \beta\le \frac{1}{3}  \\
l^L(1), \mbox{if}\;  \frac{1}{3}<\beta<1\; \mbox{and}\; t\in [0,t_1]
 \\
l^L(\mu), \mbox{if}\;  \frac{1}{3}<\beta<1\; \mbox{and}\; t\in [t_1,2].
\end{array}
\right.$$
Hence, we get 3 cases of $\left[{_{gr}^C\mathcal{D}}^{\beta,1}_{2^-}
\widetilde{u}(t)\right]^\mu$ respectively,
which coincides with the results in Example 2 in \cite{NZ17}.

\noindent $*${\it Case 2:} $\bar{\rho}=2$. Then, $A=-\frac{2^{\beta-1}(4-t^2)^{1-\beta}2^{-1}}
{\Gamma(2-\beta)}\le 0$ and
\begin{eqnarray*}
\frac{d}{d\gamma}\left(l^L(\gamma)\right)
&=&A{\scriptsize
\left(-4
 {_2F_1}(1,-\frac{1}{2};2-\beta;
 1-\frac{t^2}{4}) +
 {_2F_1}(1,\frac{1}{2};2-\beta;
 1-\frac{t^2}{4})\right)}\\
&=&A{\scriptsize
\left(-4
 {_2F_1}(1,-\frac{1}{2};2-\beta;
 z) +
 {_2F_1}(1,\frac{1}{2};2-\beta;
 z)\right)}=A.p(z),
 \end{eqnarray*}
where $z=1-\frac{t^2}{4}\in [0,1]$ when $t\in [0,2]$.

Now, we consider the sign of $p(z)$ in $[0,1]$ and $\frac{d}{d\gamma}\left(l^L(\gamma)\right)=Ap(z)$ with $A\le 0$. Since
 $$p'(z)=\frac{1}{2-\beta}
 \left(2{_2F_1}(2,\frac{1}{2};3-\beta;
z) +\frac{1}{2}
 {_2F_1}(2,\frac{3}{2};3-\beta;
 z)\right)>0,$$
$p(z)$ is increasing on $[0,1]$. At $z=0$, one has
 $$p(0)=-4
 {_2F_1}(1,-\frac{1}{2};2-\beta;
0) +
 {_2F_1}(1,\frac{1}{2};2-\beta;
 0)=-4.1+1=-3<0.$$
At $z=1$, if  $c-b-a=2-\beta-1-(-\frac{1}{2})=\frac{1}{2}-\beta>0$, one has
\begin{eqnarray*}
p(1)&=&-4
 {_2F_1}(1,-5;2-\beta;
1) +
 {_2F_1}(1,-1;2-\beta;
 1)\\
 &=&-4
 \frac{2-\beta-1}
{2-\beta-(-1/2)-1} +
 \frac{2-\beta-1}
{2-\beta-1/2-1}\\
&=&
 \frac{(1-\beta)(3\beta-1/2)}
{(1/2-\beta)(3/2-\beta)}>0\Leftrightarrow \frac{1}{6}<\beta<\frac{1}{2}.
\end{eqnarray*}
Notice that when $c-b-a=\frac{1}{2}-\beta\le  0$, the function  ${_2F_1}(1,\frac{1}{2};2-\beta;
 1)\to +\infty$ by Remark \ref{rem2.7} (v), leading that $p(1)>0$ when $1/2<\beta<1$.

If $0 < \beta <\frac{1}{6}$,  $p(z)\le p(0) <0$ for all $z\in [0,1]$, leading that $\frac{d}{d\gamma}\left(l^L(\gamma)\right)>0$ and
  $$\inf\limits_{\gamma\ge \mu}\min\limits_{\alpha_{\phi}}		\phi^{gr}(t,\gamma,\alpha_{\phi})=l^L(\mu)=A
\left(4(2-\mu)
 {_2F_1}(1,-\frac{1}{2};2-\beta;
 1-\frac{t^2}{4}) + \mu
 {_2F_1}(1,\frac{1}{2};2-\beta;
 1-\frac{t^2}{4})\right).$$

If $\frac{1}{6} < \beta <1$, $p(0)=-3<0,p(1)>0$, then
 $p(z)=0 (\frac{d}{d\gamma}\left(l^L(\gamma)\right)=0)$  has an unique solution $z_0\in [0,1] (t_0\in [0,2]$, respectively). For example, with $\beta=3/4$, using

\noindent{$>>$pkg load gsl}
		
\noindent{$>>$ fsolve(@(t) -4*gsl$\_$sf$\_$
	hyperg$\_$2F1(1,-1/2,5/4,1-t.\textasciicircum 2/4)

\indent +
gsl$\_$sf$\_$
	hyperg$\_$2F1(1,1/2,5/4,1-t.\textasciicircum 2/4),1)}

\noindent in Octave, we get $t_0\approx {\bf 0.7404}$.

 Therefore, $\frac{d}{d\gamma}\left(l^L(\gamma)\right)<0$ when $t\in [0,t_0]$  and $\frac{d}{d\gamma}\left(l^L(\gamma)\right)>0$ when $t\in [t_0,2]$ , leading that
 $$\inf\limits_{\gamma\ge \mu}\min\limits_{\alpha_{\phi}}		\phi^{gr}(t,\gamma,\alpha_{\phi})=\left\{\begin{array}{c}
l^L(1), \mbox{if}\;   t\in [0,t_0]
 \\
l^L(\mu), \mbox{if}\;  t\in [t_0,2].
\end{array}
\right.$$
Hence, ${_{gr}^C\mathcal{D}}^{\beta,2}_{2^-}
\widetilde{\phi}(t)$ exists,  see Figure 2.1.4, and its $\mu$-level sets are
   $$\left[{_{gr}^C\mathcal{D}}^{\beta,2}_{2^-}
\widetilde{\phi}(t)\right]^\mu=\left\{\begin{array}{c}
{[}l^L(\mu),l^R(\mu){]},\mbox{if}\; 0\le \beta\le \frac{1}{6}  \\
{[}l^L(1),l^R(\mu){]}, \mbox{if}\;  \frac{1}{6}<\beta<1\; \mbox{and}\; t\in [0,t_0] \\
{[}l^L(\mu),l^R(\mu){]}, \mbox{if}\;  \frac{1}{6}<\beta<1\; \mbox{and}\; t\in [t_0,2].
\end{array}
\right.$$
\begin{figure}[htbp]
    \centering

    \begin{tabular}{cc}

         \includegraphics[width=0.45\textwidth]{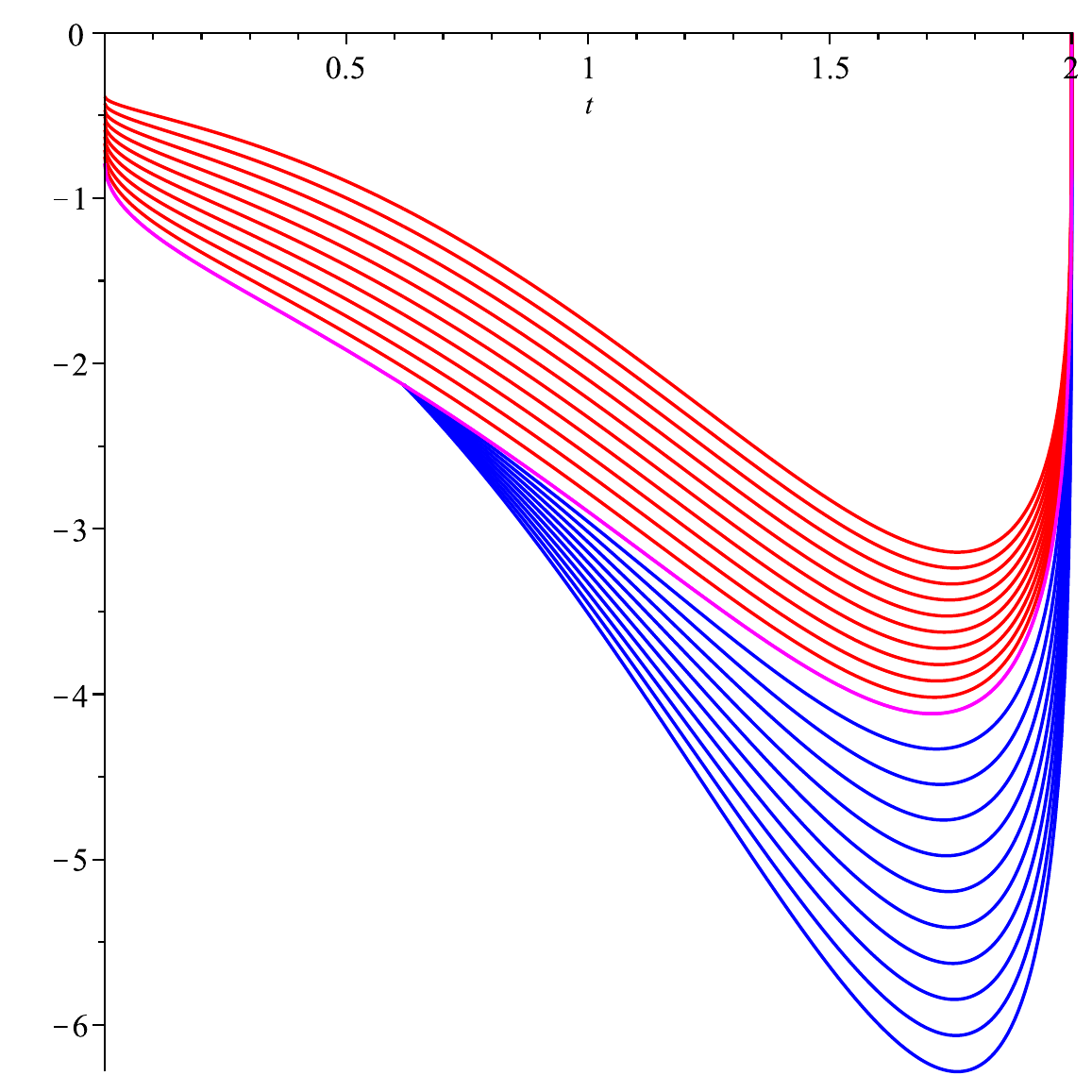} &
        \includegraphics[width=0.45\textwidth]{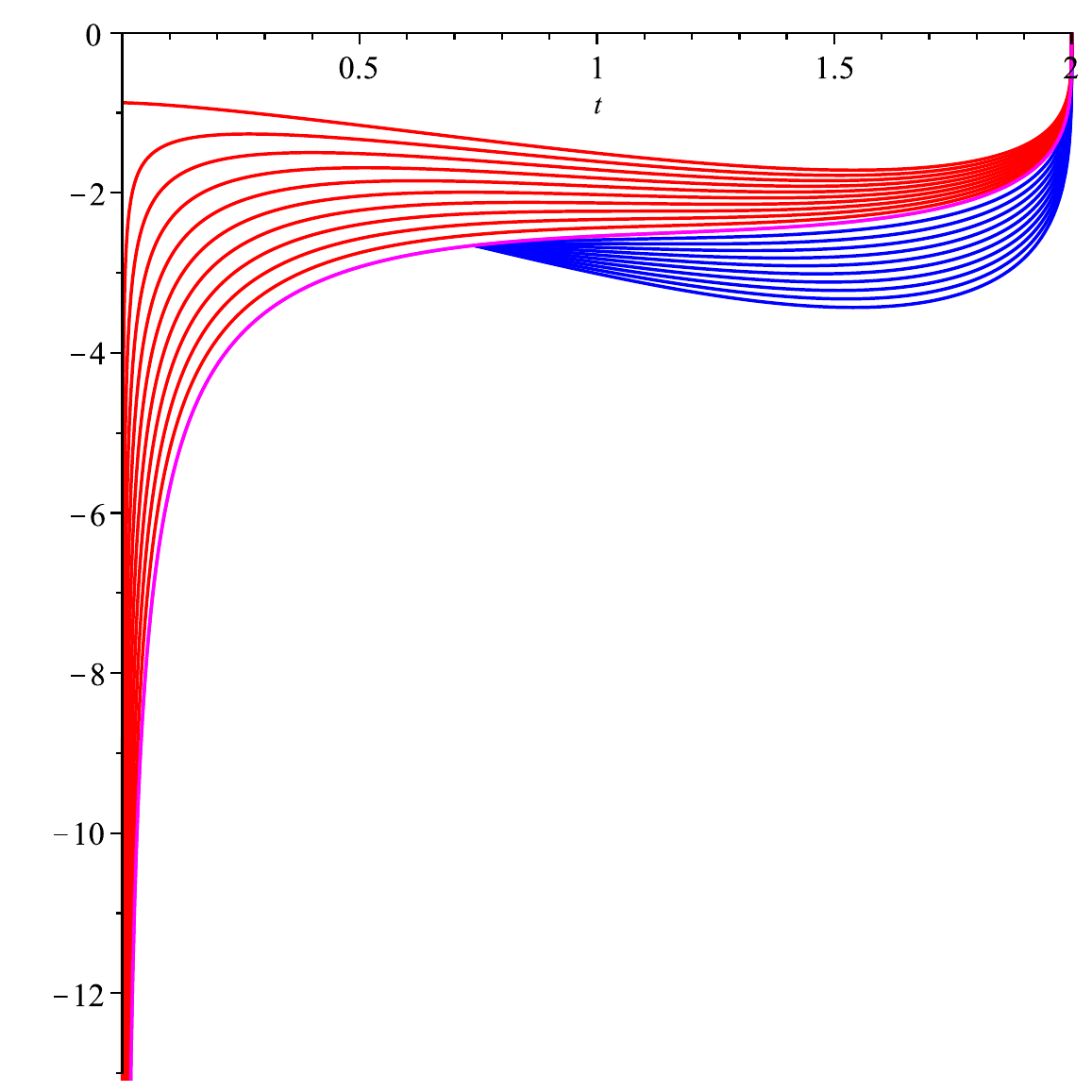} \\

        Figure 2.1.3. ${_{gr}^C\mathcal{D}}^{\frac{3}{4},\frac{1}{2}}_{2^-}
\widetilde{\phi}(t)$ &
        Figure 2.1.4. ${_{gr}^C\mathcal{D}}^{\frac{3}{4},2}_{2^-}
\widetilde{\phi}(t)$ \\
    \end{tabular}
\end{figure}

  \noindent $*${\it Case 3:} $\bar{\rho}=\frac{1}{2}$. Then, $A=-\frac{(1/2)^{\beta-1}(\sqrt{2}-\sqrt{t})^{1-\beta}
 2^{1/2}}
{\Gamma(2-\beta)}\le 0$ and
\begin{eqnarray*}
\frac{d}{d\gamma}\left(l^L(\gamma)\right)
&=&A
\left(-4
 {_2F_1}(1,-5;2-\beta;
 1-\sqrt{\frac{t}{2}}) +
 {_2F_1}(1,-1;2-\beta;
 1-\sqrt{\frac{t}{2}})\right)\\
 &=&A{\scriptsize
\left(-4
 {_2F_1}(1,-5;2-\beta;
 z) +
 {_2F_1}(1,-1;2-\beta;
 z)\right)}=A.q(z),
 \end{eqnarray*}
where $z=1-\sqrt{\frac{t}{2}}\in [0,1]$ when $t\in [0,2]$.
We also get that $q'(z)>0$ in $[0,1]$ and
 $$q(0)=-4
 {_2F_1}(1,-5;2-\beta;
0) +
 {_2F_1}(1,-1;2-\beta;
 0)=-4.1+1=-3<0,$$
\begin{eqnarray*}
q(1)&=&-4
 {_2F_1}(1,-5;2-\beta;
1) +
 {_2F_1}(1,-1;2-\beta;
 1)\\
 &=&-4
 \frac{2-\beta-1}
{2-\beta-(-5)-1} +
 \frac{2-\beta-1}
{2-\beta-(-1)-1}\\
&=&
 \frac{(1-\beta)(3\beta-2)}
{(2-\beta)(6-\beta)}>0\Leftrightarrow \frac{2}{3}<\beta<1,
\end{eqnarray*}
since $c-a-b=2-\beta-1-(-5)=6-\beta>0$ for all $\beta\in (0,1)$.

By the analysis similar to the {\it Case} 2, we get that
 ${_{gr}^C\mathcal{D}}^{\beta,\frac{1}{2}}_{2^-}
\widetilde{\phi}(t)$ exists and its $\mu$-level sets are
$$\left[{_{gr}^C\mathcal{D}}^{\beta,\frac{1}{2}}_{2^-}
\widetilde{\phi}(t)\right]^\mu=\left\{\begin{array}{c}
{[}l^L(\mu),l^R(\mu){]},\mbox{if}\; 0< \beta\le \frac{2}{3}  \\
{[}l^L(1),l^R(\mu){]}, \mbox{if}\;  \frac{2}{3}<\beta<1\; \mbox{and}\; t\in [0,t_1] \\
{[}l^L(\mu),l^R(\mu){]}, \mbox{if}\;  \frac{2}{3}<\beta<1\; \mbox{and}\; t\in [t_1,2],
\end{array}
\right.$$
where $t_1\in [0,2]$ is the unique solution of $\frac{d}{d\gamma}\left(l^L(\gamma)\right)=0$, see Figure 2.1.3.
\end{example}

\section{Optimality conditions for fuzzy variational problems}
	\label{sect3}
For simplicity, this section focuses on the case $n=2$; the proof extends naturally to the general case.

\begin{lemma}\label{lem2.2} (see e.g. \cite{GH96,H75}) Let $\Omega\subset \mathbb{R}^2$ be a region and
		let $f : \Omega \to \mathbb{R}$ and $h: \Omega \to \mathbb{R}$  be continuously differentiable  functions. Suppose that $f$ has a local extremum at ${\bf x}^*\in \Omega$ subject to the conditions that $h({\bf x}^*) = 0$ and  $\nabla h({\bf x}^{*})\neq {\bf 0}$. Then there are the numbers $\lambda\in\mathbb{R}$ such that
		$$\nabla f({\bf x}^*)+\lambda\nabla h({\bf x}^*)=0\Leftrightarrow
		\partial_if({\bf x}^*)+\lambda \partial_ih({\bf x}^*)=0, \forall i=1,2,$$
		where $\partial_i=\frac{\partial}{\partial x_i}(i=1,2)$.
	\end{lemma}

\begin{lemma}\label{lem2.1} (see e.g. \cite{GH96,H75}) Assume that $\phi(t)$ is a continuous real-valued function on $[a,b]$. If it holds
		$$\int_{a}^{b}\phi(t)\eta(t)dt=0,$$
		for any $\eta(t)\in C^1[a,b]$ such that $\eta(a)=\eta(b)=0$, then $\phi(t)=0$ for all $t\in [a,b]$.
	\end{lemma}
	\begin{lemma}\label{lem4.5} (see \cite{A17}) Let $\beta\in (0,1), \rho>0$ and $u,v\in C^1([a,b])$. Then,
		$$\int_a^bu(t).{^CD}^{\beta,\rho}_{a^+}v(t)dt
		=I_{b^-}^{1-\beta,\rho}u(t).v(t)\mid_{a}^b
		+\int_a^bv(t).{D}^{\beta,\rho}_{b^-}u(t)dt.$$
	\end{lemma}
	
\begin{lemma} \label{lem4.4} Let $\beta\in (0,1)$ and $\rho>0$.
		\begin{enumerate}
			\item[(i)] \cite{K11,K14} Let $\phi,\varphi \in X^p_c[a,b]$. Then,

$$I^{\beta,\rho}_{a^+}(\phi(t)+\varphi(t))=I^{\beta,\rho}_{a^+}\phi(t)+
			I^{\beta,\rho}_{a^+}\varphi(t)$$
$$D^{\beta,\rho}_{a^+}(\phi(t)+\varphi(t))=D^{\beta,\rho}_{a^+}\phi(t)+
			D^{\beta,\rho}_{a^+}\varphi(t),$$
and
$$D^{\beta,\rho}_{a^+}\left(I^{\beta,\rho}_{a^+}\phi(t)
\right)=\phi(t), D^{\beta,\rho}_{b^-}\left(I^{\beta,\rho}_{b^-}\phi(t)
\right)=\phi(t).$$

\item[(ii)] \cite{AMO16} If $\phi\in C[a,b]$, then
$^CD^{\beta,\rho}_{a^+}\left(I^{\beta,\rho}_{a^+}\phi(t)
\right)=\phi(t)$. If $\phi\in C^1[a,b]$, then $$I^{\beta,\rho}_{a^+}\left({^CD}^{\beta,\rho}_{a^+}\phi(t)
\right)=\phi(t)-\phi(a).$$

\item[(iii)] \cite{JAB17}  $$I^{\beta,\rho}_{b^-}
    \left(\frac{b^{\rho}-t^{\rho}}{\rho}\right)^\nu		=\frac{\Gamma(\nu+1)}{\Gamma(\nu+1+\beta)}
    \left(\frac{b^{\rho}-t^{\rho}}{\rho}\right)^{\nu+\beta},$$
$$			D^{\beta,\rho}_{b^-}\left(\frac{b^{\rho}-t^{\rho}}{\rho}\right)^\nu			=\frac{\Gamma(\nu+1)}{\Gamma(\nu+1-\beta)}
\left(\frac{b^{\rho}-t^{\rho}}{\rho}\right)^{\nu-\beta}.
			$$
			In particular, if $\nu=0$ then $$I^{\beta,\rho}_{b^-}1
			=\frac{1}{\Gamma(1+\beta)}
\left(\frac{b^{\rho}-t^{\rho}}{\rho}\right)^{\beta},
D^{\beta,\rho}_{b^-}1
			=\frac{1}{\Gamma(1-\beta)}
\left(\frac{b^{\rho}-t^{\rho}}{\rho}\right)^{-\beta}.$$
		
		\end{enumerate}
	\end{lemma}
 By the proof similar to Theorem 5, p. 22 in \cite{MDM18}, we have
			$$I^{\beta,\rho}_{b^-}(C_1\phi(t)+C_2\varphi(t))=
C_1I^{\beta,\rho}_{b^-}\phi(t)+
			C_2I^{\beta}_{b^-}\varphi(t),$$
			$$D^{\beta,\rho}_{b^-}(C_1\phi(t)+C_2\varphi(t))
=C_1D^{\beta,\rho}_{b^-}\phi(t)+
			C_2D^{\beta,\rho}_{b^-}\varphi(t).$$	
\begin{lemma}\cite{KST06}  Let $\beta\in (0,1)$ and $\phi \in C[a,b]$. The solution of the equation
			${^CD}^{\beta}_{a^+}\phi(t)=\phi(t),\phi(a)=C_1 (C_1\in\mathbb{R})$
			has the form
$$\phi(t)=C_1+
			\frac{1}{\Gamma(\beta)}\int_a^t
			\left(t-\tau\right)^{\beta-1}\phi(\tau)d\tau.$$
\end{lemma}
	\begin{lemma} \label{lem4.6} Let $\beta\in (0,1)$, $\rho>0$ and $\phi \in C[a,b]$.
		\begin{enumerate}
			\item[(i)]  $D^{\beta,\rho}_{b^-}\phi(t)=0$ if and only if $\phi(t)=C\left(\frac{b^{\rho}-t^{\rho}}{\rho}\right)^{\beta-1}$ with any $C\in\mathbb{R}$.
\item[(ii)]  $D^{\beta,\rho}_{b^-}\phi(t)=1$ if and only if $\phi(t)=\frac{1}{\Gamma(1+\beta)}
    \left(\frac{b^{\rho}-t^{\rho}}{\rho}\right)^{\beta}
    +C\left(\frac{b^{\rho}-t^{\rho}}{\rho}\right)^{\beta-1}$ with any $C\in\mathbb{R}$.	
			\item[(iii)]  The solution of the equation
			$${^CD}^{\beta,\rho}_{a^+}\phi(t)=\phi(t),\phi(a)=C_1 (C_1\in\mathbb{R})$$
			has the form
$$\phi(t)=C_1+
			\frac{1}{\Gamma(\beta)}\int_a^t
			\left(\frac{t^{\rho}-
\tau^{\rho}}{\rho}\right)^{\beta-1}\tau^{\rho-1}\phi(\tau)d\tau.$$
		
		\end{enumerate}
	\end{lemma}
\begin{proof} (i)  ($\Leftarrow$)
$D^{\beta,\rho}_{b^-}\phi(t)
=C.D^{\beta,\rho}_{b^-}
\left(\frac{b^{\rho}-t^{\rho}}{\rho}\right)^{\beta-1}$
$$=-C.\frac{\rho^{\beta}}{\Gamma(1-\beta)}
t^{1-\rho}				.\frac{d}{dt}\left(\int_t^b\frac{\tau^{\rho-1}}{(\tau^{\rho}-t^{\rho}
)^{\beta}}
\left(\frac{b^{\rho}-\tau^{\rho}}{\rho}\right)^{\beta-1}d\tau\right).$$
Setting $u=\frac{b^{\rho}-\tau^{\rho}}{\rho}$, one has $du=-\tau^{\rho-1}d\tau$, $\tau^{\rho}-t^{\rho}
=b^{\rho}-\rho.u-t^{\rho}
=\rho\left(\frac{b^{\rho}-t^{\rho}}{\rho}-u\right)$
and $\tau=t\Rightarrow u=\frac{b^{\rho}-t^{\rho}}{\rho}$,
$\tau=b\Rightarrow u=0.$ Hence,
$$D^{\beta,\rho}_{b^-}\phi(t)
=-\frac{C\rho^{\beta}}{\Gamma(1-\beta)}
\frac{d}{dt}\left(\int_{\frac{b^{\rho}-t^{\rho}}{\rho}}^0
\frac{1}{\left(\rho\left(\frac{b^{\rho}-t^{\rho}}
{\rho}-u\right)\right)^{\beta}}
u^{\beta-1}(-du)\right)$$
$$=-\frac{C}{\Gamma(1-\beta)}
\frac{d}{dt}\left(\int_0^{\frac{b^{\rho}-t^{\rho}}{\rho}}
\left(\frac{b^{\rho}-t^{\rho}}{\rho}-u\right)^{-\beta}
u^{\beta-1}du\right) \;\left(\mbox{setting}\; v=\frac{u}{\frac{b^{\rho}-t^{\rho}}{\rho}} \right)$$
$$=-\frac{C}{\Gamma(1-\beta)}
\frac{d}{dt}\left(\int_0^1
\left(\frac{b^{\rho}-t^{\rho}}{\rho}-
\frac{b^{\rho}-t^{\rho}}{\rho}v\right)^{-\beta}
\left(\frac{b^{\rho}-t^{\rho}}{\rho}v\right)^{\beta-1}
.\frac{b^{\rho}-t^{\rho}}{\rho}dv\right)$$
$$=-\frac{C}{\Gamma(1-\beta)}
\frac{d}{dt}\left(\int_0^1
(1-v)^{-\beta}v^{\beta-1}dv\right)
=-\frac{C}{\Gamma(1-\beta)}\frac{d}{dt}\left(B(\beta,1-\beta)\right)=0.$$

\noindent($\Rightarrow$) We deduce from $D^{\beta,\rho}_{b^-}\phi(t)=0$ that
$$-\frac{\rho^{\beta}}{\Gamma(1-\beta)}
t^{1-\rho}				.\frac{d}{dt}\left(\int_t^b\frac{\tau^{\rho-1}}{(\tau^{\rho}-t^{\rho}
)^{\beta}}
\phi(\tau)d\tau\right)=0,$$
which together with $\frac{\rho^{\beta}}{\Gamma(1-\beta)}
t^{1-\rho}\neq 0$ implies that
$$\int_t^b\frac{\tau^{\rho-1}}{(\tau^{\rho}-t^{\rho}
)^{\beta}}
\phi(\tau)d\tau=C.$$
Equivalently,
$$\frac{\rho^{1-(1-\beta)}}{\Gamma (1-\beta)}\int_t^b\frac{\tau^{\rho-1}}{(\tau^{\rho}-t^{\rho}
)^{1-(1-\beta)}}
\phi(\tau)d\tau
=\frac{\rho^{1-(1-\beta)}}{\Gamma (1-\beta)}C\Leftrightarrow I^{1-\beta,\rho}_{b^-}\phi(t) =
\frac{\rho^{\beta}}{\Gamma (1-\beta)}C.$$
By applying $D^{1-\beta,\rho}_{b^-}$ for both sides of the above equation, we get that
$$\phi(t)=D^{1-\beta,\rho}_{b^-}(\frac{\rho^{\beta}}{\Gamma (1-\beta)}C)=\frac{C\rho^{\beta}}{\Gamma (1-\beta)}D^{1-\beta,\rho}_{b^-}(1)$$
$$=
\frac{C\rho^{\beta}}{\Gamma (1-\beta)}.\frac{1}{\Gamma(1-(1-\beta))}
\left(\frac{b^{\rho}-t^{\rho}}{\rho}\right)^{-(1-\beta)}$$
$$=C_1
\left(\frac{b^{\rho}-t^{\rho}}{\rho}\right)^{\beta-1}, \;\mbox{where}\; C_1=\frac{C\rho^{\beta}}{\Gamma (1-\beta)}.\frac{1}{\Gamma(\beta)}.$$

\noindent(ii)  ($\Leftarrow$)
$D^{\beta,\rho}_{b^-}\phi(t)
=\frac{1}{\Gamma(1+\beta)}.
    D^{\beta,\rho}_{b^-}
    \left(\frac{b^{\rho}-t^{\rho}}{\rho}\right)^{\beta}
    +C.D^{\beta,\rho}_{b^-}
    \left(\frac{b^{\rho}-t^{\rho}}{\rho}\right)^{\beta-1}$
$$=\frac{1}{\Gamma(1+\beta)}.\frac{\Gamma(\beta+1)}{\Gamma(\beta+1-\beta)}
\left(\frac{b^{\rho}-t^{\rho}}{\rho}\right)^{\beta-\beta}+0=1.$$

\noindent($\Rightarrow$) We deduce from $D^{\beta,\rho}_{b^-}\phi(t)=1$ that
$$-\frac{\rho^{\beta}}{\Gamma(1-\beta)}
t^{1-\rho}				.\frac{d}{dt}\left(\int_t^b\frac{\tau^{\rho-1}}{(\tau^{\rho}-t^{\rho}
)^{\beta}}
\phi(\tau)d\tau\right)=1,$$
which together with $\frac{\rho^{\beta}}{\Gamma(1-\beta)}
t^{1-\rho}\neq 0$ implies that
$$\int_t^b\frac{\tau^{\rho-1}}{(\tau^{\rho}-t^{\rho}
)^{\beta}}
\phi(\tau)d\tau=-\frac{\Gamma(1-\beta)}{\rho^{\beta}}
\int\left(t^{\rho-1}\right)dt+C=
-\frac{\Gamma(1-\beta)}{\rho^{\beta+1}}t^{\rho}+C.$$
Substituting $t=b$ in the above equation, one has
$$0=-\frac{\Gamma(1-\beta)}{\rho^{\beta+1}}t^{\rho}+C\Rightarrow
C=\frac{\Gamma(1-\beta)}{\rho^{\beta+1}}b^{\rho},$$
leading that
$$\int_t^b\frac{\tau^{\rho-1}}{(\tau^{\rho}-t^{\rho}
)^{\beta}}
\phi(\tau)d\tau=\frac{\Gamma(1-\beta)}{\rho^{\beta+1}}
(b^{\rho}-t^{\rho})$$
Equivalently,
$$\frac{\rho^{1-(1-\beta)}}{\Gamma (1-\beta)}\int_t^b.\frac{\tau^{\rho-1}}{(\tau^{\rho}-t^{\rho}
)^{1-(1-\beta)}}
\phi(\tau)d\tau
=\frac{\rho^{1-(1-\beta)}}{\Gamma (1-\beta)}.\frac{\Gamma(1-\beta)}{\rho^{\beta+1}}
(b^{\rho}-t^{\rho})$$
$$\Leftrightarrow I^{1-\beta,\rho}_{b^-}\phi(t) =
\frac{b^{\rho}-t^{\rho}}{\rho}.$$
By applying $D^{1-\beta,\rho}_{b^-}$ for both sides of the above equation, we get that
$$\phi(t)=D^{1-\beta,\rho}_{b^-}
\left(\frac{b^{\rho}-t^{\rho}}{\rho}\right)
=\frac{\Gamma(1+1)}{\Gamma(1+1-(1-\beta))}
\left(\frac{b^{\rho}-t^{\rho}}{\rho}\right)^{1-(1-\beta)}$$
$$
=\frac{\Gamma(2)}{\Gamma(1+\beta)}
\left(\frac{b^{\rho}-t^{\rho}}{\rho}\right)^{\beta}
=\frac{1}{\Gamma(1+\beta)}
\left(\frac{b^{\rho}-t^{\rho}}{\rho}\right)^{\beta}.$$
Combining this and (i),
we get that
$$\phi(t)=\frac{1}{\Gamma(1+\beta)}
\left(\frac{b^{\rho}-t^{\rho}}{\rho}\right)^{\beta}
+C\left(\frac{b^{\rho}-t^{\rho}}{\rho}\right)^{\beta-1}.$$
\noindent (iii) By applying $I^{\beta,\rho}_{a^+}$, we deduce from Lemma \ref{lem4.4} (iii) that
$$\phi(t)-\phi(a)=I^{\beta,\rho}_{a^+}\left({^CD}^{\beta,\rho}_{a^+}
\phi(t)\right)
=I^{\beta,\rho}_{a^+}\phi(t),$$
which together with $\phi(a)=C_1$ leads to
$$\phi(t)=C_1+\frac{\rho^{1-\beta}}{\Gamma(\beta)}
\int_a^t\frac{\tau^{\rho-1}}{(t^{\rho}-\tau^{\rho})^{1-\beta}}
\phi(\tau)d\tau,$$
which completes the proof.
\end{proof}
	
	\subsection{Fuzzy fractional variational problems}
	
	In the sequel, let $\widetilde{f}:\mathbf{F}_C(\mathbb{R})^2\times\mathbf{F}_C(\mathbb{R})^2 \to \mathbf{F}_C(\mathbb{R})$ be a fuzzy function of class $\mathcal{C}^1$,
	$\widetilde{J}:\mathcal{C}^1([a,b])^2\to \mathbf{F}_C(\mathbb{R})$, $\beta\in (0,1)$ and $\rho>0$.
	In this section, we consider the fuzzy fractional variational problem of several dependent variables as follows:
	
	\noindent (P) $\widetilde{\min}$ $\displaystyle\widetilde{J}(\widetilde{{\bf x}})={\widetilde{\int}}_{a}^{b}
	\widetilde{f}(\widetilde{x}_1(t),\widetilde{x}_2(t),
	{_{gr}^C\mathcal{D}}^{\beta}_{a^+}
	\widetilde{x}_1(t),{_{gr}^C\mathcal{D}}^{\beta,\rho}_{a^+}
	\widetilde{x}_2(t),t)dt$
	
	\indent s.t. $\widetilde{x}_i(a)=\widetilde{x}_{0,i},
	\widetilde{x}_i(b)=\widetilde{x}_{1,i}$ for all $i\in I:=\{1,2\}$.
	
	\noindent Denote $\widetilde{{\bf x}}=(\widetilde{x}_1,\widetilde{x}_2)$,  ${_{gr}^C\mathcal{D}}^{\beta,\rho}_{a^+}\widetilde{{\bf x}}
	=({_{gr}^C\mathcal{D}}^{\beta,\rho}_{a^+}\widetilde{x}_1,
	{_{gr}^C\mathcal{D}}^{\beta,\rho}_{a^+}\widetilde{x}_2)$,
	$\widetilde{{\bf x}}(t)=(\widetilde{x}_1(t),\widetilde{x}_2(t))$ and $${_{gr}^C\mathcal{D}}^{\beta,\rho}_{a^+}\widetilde{{\bf x}}(t)
	=({_{gr}^C\mathcal{D}}^{\beta,\rho}_{a^+}\widetilde{x}_1(t),
	{_{gr}^C\mathcal{D}}^{\beta,\rho}_{a^+}\widetilde{x}_2(t)),$$
	$$f[\widetilde{{\bf x}},{_{gr}^C\mathcal{D}}^{\beta,\rho}_{a^+}\widetilde{{\bf x}}](t)=\widetilde{f}(\widetilde{x}_1(t),\widetilde{x}_2(t),
	{_{gr}^C\mathcal{D}}^{\beta,\rho}_{a^+}
	\widetilde{x}_1(t),{_{gr}^C\mathcal{D}}^{\beta,\rho}_{a^+}
	\widetilde{x}_2(t),t).$$
	We propose the following definition in the line of \cite{SDL19,DLK20}.
	\begin{definition}
		A fuzzy function $\widetilde{{\bf x}}^*$ in $\widetilde{\Omega}:=\{\widetilde{{\bf x}}\in \mathcal{C}^1([a,b])^2\mid \widetilde{x}_i(a)=\widetilde{x}_{0,i},\widetilde{x}_i(b)
	=\widetilde{x}_{1,i}\}$ is a minimal solution of (P) if the increment
		of $\widetilde{J}$ has to be fuzzy non-negative, that is,
		$$\Delta\widetilde{J}:=\widetilde{J}(\widetilde{{\bf x}})\ominus_{gr}
		\widetilde{J}(\widetilde{{\bf x}}^*)\ge \widetilde{0}, 		\forall\widetilde{{\bf x}}\in \widetilde{\Omega}.$$
	\end{definition}
	
	\noindent The granular problem (GP) corresponding to the problem (P) is
	
	\noindent min $\displaystyle J^{gr}({\bf x}^{gr})=\int_{a}^{b}
	f^{gr}(x^{gr}_1,x^{gr}_2,
	{^CD}^{\beta,\rho}_{a^+}x_1^{gr},
	{^CD}^{\beta,\rho}_{a^+}x_2^{gr},t,\mu,\alpha_f)dt$
	
	\noindent s.t. $x^{gr}_i(a,\mu,\alpha_{x_i})=x_{0,i}^{gr}(\mu,
	\alpha_{x_{0,i}}),
	x^{gr}_i(b,\mu,\alpha_{x_i})=x_{1,i}^{gr}(\mu,\alpha_{x_{1,i}}),\forall \alpha_{x_i}=\alpha_{x_{0,i}}=\alpha_{x_{1,i}}$
	and ${\bf x}^{gr}=(x^{gr}_1,x^{gr}_2), x^{gr}_i=x^{gr}_i(t,\mu,\alpha_{x_i})$ for $i\in I$.
	
	Denote ${^CD}^{\beta,\rho}_{a^+}{\bf x}^{gr}=({^CD}^{\beta,\rho}_{a^+}x_1^{gr},
	{^CD}^{\beta,\rho}_{a^+}x_2^{gr})$, ${^CD}^{\beta,\rho}_{a^+}x_i^{gr}={^CD}^{\beta,\rho}_{a^+}x_i^{gr}
(t,\mu,\alpha_{x_i})$ for $i\in I$, and
	$$f^{gr}[{\bf x}^{gr},{^CD}^{\beta,\rho}_{a^+}{\bf x}^{gr}](t,\mu,\alpha_f)=f^{gr}(x_1^{gr},x_2^{gr},
	{^CD}^{\beta,\rho}_{a^+}x_1^{gr},
	{^CD}^{\beta,\rho}_{a^+}x_2^{gr},t,\mu,\alpha_f).$$
	Moreover, for $\widetilde{{\bf x}}(t)=(\widetilde{x}_1(t),\widetilde{x}_2(t))\in \mathbf{F}_C(\mathbb{R})^2$ and ${\bf x}^{gr}=(x_1^{gr},x_2^{gr})\in \mathbb{R}^2$ with $x_i^{gr}=x_i^{gr}(t,\mu,\alpha_{x_i})$ for $i=1,2$, denote
	$$\mathcal{H}(\widetilde{{\bf x}}(t))=(\mathcal{H}(\widetilde{x}_1(t)),
	\mathcal{H}(\widetilde{x}_2(t)))={\bf x}^{gr},$$
	$$\mathcal{H}^{-1}({\bf x}^{gr})=(\mathcal{H}^{-1}(x_1^{gr}),
	\mathcal{H}^{-1}(x_2^{gr}))=\widetilde{{\bf x}}(t).$$
	\begin{definition} Let $\mu,\alpha_{x_i}\in [0,1]$ for $i=1,2$.
					A granular  function ${\bf x}^{*gr}=\mathcal{H}(\widetilde{{\bf x}}^*(t))$ in admissible granular functions $\Omega^{gr}$ is a minimal solution of (GP) if the increment
			of $J^{gr}$ has to be  non-negative, that is,
			$$\Delta J^{gr}:=J^{gr}({\bf x}^{gr})-
			J^{gr}({\bf x}^{*gr})\ge 0, \forall {\bf x}^{gr}\in \Omega^{gr}.$$
 \end{definition}
	\begin{proposition}\label{prop3.1} Let $\mathcal{H}:\widetilde{\Omega}\subset\mathbf{F}_C(\mathbb{R})^n\to \mathbb{R}^2$ be defined by $\mathcal{H}(\widetilde{{\bf x}}(t))={\bf x}^{gr}$. Suppose that $\mathcal{H}$ is a bijective map and $\mathcal{H}^{-1}:\mathbb{R}^2\to \mathbf{F}_C(\mathbb{R})^n$ is defined by $\mathcal{H}^{-1}({\bf x}^{gr})=\widetilde{{\bf x}}(t)$. Then,
		\begin{enumerate}
			\item[(i)] $\mathcal{H}(\widetilde{\Omega})=\Omega^{gr}$ and $\mathcal{H}^{-1}(\Omega^{gr})=\widetilde{\Omega}$.
			\item[(ii)] If $\widetilde{{\bf x}}^*(t)$ is a minimal solution of (P) then $\mathcal{H}(\widetilde{{\bf x}}^*(t))={\bf x}^{*gr}$ is a minimal solution of (GP).
			\item[(iii)] If ${\bf x}^{*gr}$ is a minimal solution of (GP) then $\widetilde{{\bf x}}^*(t)=\mathcal{H}^{-1}({\bf x}^{*gr})$ is a minimal solution of (GP).
		\end{enumerate}
	\end{proposition}
	\begin{proof} The proof is similar to the proof of Proposition 9 in \cite{TTKN25}.
	\end{proof}
	
	\begin{proposition}\label{prop3.2} If $\widetilde{{\bf x}}^*\in \widetilde{\Omega}$ is a minimal solution of (P), then
		${\bf x}^{*gr}=\mathcal{H}(\widetilde{{\bf x}}^*(t))$ satisfies the following granular Euler-Lagrange equations
		\begin{equation}\label{prop3.1-0}
			\partial_if^{gr}+{D}^{\beta,\rho}_{b^-}\partial_{i+n}f^{gr}=0,\forall \mu,\alpha_{f}\in[0,1], \mbox{for each}\; i\in I,
		\end{equation}
		where  $f^{gr}=f^{gr}[{\bf x}^{*gr},{^CD}^{\beta,\rho}_{a^+}{\bf x}^{*gr}](t,\mu,\alpha_f)$,  and ${^CD}^{\beta,\rho}_{a^+}{\bf x}^{*gr}=({^CD}^{\beta,\rho}_{a^+}{x}_1^{*gr},
		{^CD}^{\beta,\rho}_{a^+}{x}_2^{*gr})$.
	\end{proposition}

	\begin{proof} Let $\widetilde{{\bf w}}(t)=(\widetilde{w}_1(t),\widetilde{w}_2(t))$ be an arbitrary fuzzy function with $\widetilde{w}_i(t)\in \mathcal{C}^1[a,b]$, satisfying
		$\widetilde{{\bf x}}^*(t)+\epsilon\widetilde{{\bf w}}(t)\in \widetilde{\Omega}$
		for any $\epsilon>0$ small enough and $\widetilde{w}_i(a)=\widetilde{w}_i(b)=\widetilde{0}$ for all $i\in I$. Then, we deduce from $\widetilde{{\bf x}}^*(t)$ is a minimal of (P) that
		$$\Delta\widetilde{J}({\bf x}^*,{\bf w})=\widetilde{J}(\widetilde{{\bf x}}^*+
		\epsilon\widetilde{{\bf w}})\ominus_{gr}
		\widetilde{J}(\widetilde{{\bf x}}^*)\ge \widetilde{0},$$
		or equivalent to
		$$\Delta\widetilde{J}({\bf x}^*,{\bf w})={\widetilde{\int}}_{a}^{b}
		\widetilde{f}[\widetilde{{\bf x}}^*+\epsilon {\bf w},{_{gr}^C\mathcal{D}}^{\beta,\rho}_{a^+}\widetilde{{\bf x}}^*+\epsilon {_{gr}^C\mathcal{D}}^{\beta,\rho}_{a^+}\widetilde{{\bf w}}](t)dt
		\ominus_{gr}
		{\widetilde{\int}}_{a}^{b}
		\widetilde{f}[\widetilde{{\bf x}}^*,{_{gr}^C\mathcal{D}}^{\beta,\rho}_{a^+}\widetilde{{\bf x}}^*](t)dt \ge\widetilde{0}.$$
		This arrives at
		$$\mathcal{H}\left(
		{\widetilde{\int}}_{a}^{b}
		\widetilde{f}[\widetilde{{\bf x}}^*+\epsilon {\bf w},{_{gr}^C\mathcal{D}}^{\beta,\rho}_{a^+}\widetilde{{\bf x}}^*+\epsilon {_{gr}^C\mathcal{D}}^{\beta,\rho}_{a^+}\widetilde{{\bf w}}](t)dt
		\ominus_{gr}
		{\widetilde{\int}}_{a}^{b}
		\widetilde{f}[\widetilde{{\bf x}}^*,{_{gr}^C\mathcal{D}}^{\beta,\rho}_{a^+}\widetilde{{\bf x}}^*](t)dt\right)\ge
		\mathcal{H}(\widetilde{0}),$$
		which in turn implies that
		$$\mathcal{H}\left({\widetilde{\int}}_{a}^{b}
		\widetilde{f}[\widetilde{{\bf x}}^*+\epsilon {\bf w},{_{gr}^C\mathcal{D}}^{\beta,\rho}_{a^+}\widetilde{{\bf x}}^*+\epsilon {_{gr}^C\mathcal{D}}^{\beta,\rho}_{a^+}\widetilde{{\bf w}}](t)dt\right)-
		\mathcal{H}\left({\widetilde{\int}}_{a}^{b}
		\widetilde{f}[\widetilde{{\bf x}}^*,{_{gr}^C\mathcal{D}}^{\beta,\rho}_{a^+}\widetilde{{\bf x}}^*](t)dt\right)\ge
		\mathcal{H}(\widetilde{0}).$$
		From the definition of $gr$-integrable functions, we have
		$$\int_{a}^{b}
		\mathcal{H}\left(\widetilde{f}[\widetilde{{\bf x}}^*+\epsilon {\bf w},{_{gr}^C\mathcal{D}}^{\beta,\rho}_{a^+}\widetilde{{\bf x}}^*+\epsilon {_{gr}^C\mathcal{D}}^{\beta,\rho}_{a^+}\widetilde{{\bf w}}](t)\right)dt-\int_{a}^{b}
		\mathcal{H}\left(\widetilde{f}[\widetilde{{\bf x}}^*,{_{gr}^C\mathcal{D}}^{\beta,\rho}_{a^+}\widetilde{{\bf x}}^*](t)
		\right)dt\ge 0,$$
		which is equivalent to the following inequality
		$$\int_{a}^{b}
		f^{gr}[{\bf x}^{*gr}+\epsilon {\bf w}^{gr},{^CD}^{\beta,\rho}_{a^+}{\bf x}^{*gr}+\epsilon {^CD}^{\beta,\rho}_{a^+}{\bf w}^{gr}](t,\mu,\alpha_f)dt$$
		\begin{equation}\label{prop3.1-4}
			-\int_{a}^{b}
			f^{gr}[{\bf x}^{*gr},{^CD}^{\beta,\rho}_{a^+}{\bf x}^{*gr}](t,\mu,\alpha_f)
			dt\ge 0,
		\end{equation}
		Setting
		$F(\epsilon)=\int_{a}^{b}
		f^{gr}[{\bf x}^{*gr}+\epsilon {\bf w}^{gr},{^CD}^{\beta,\rho}_{a^+}{\bf x}^{*gr}+\epsilon {^CD}^{\beta,\rho}_{a^+}{\bf w}^{gr}](t,\mu,\alpha_f)dt$, we deduce from (\ref{prop3.1-4}) that $\bar{\epsilon}=0$ is a minimum of the real-valued function $F(.)$. It follows from Remark \ref{rem2.1} that
		$f^{gr}$, $x_i^{gr},w_i^{gr}(i=1,2)$ are  real functions of class $C^1$. This together with the fact that  $\phi_i(\epsilon):=x_i^{gr}+\epsilon w_i^{gr}, \phi_i(\epsilon):={^CD}^{\beta,\rho}_{a^+}{ x}_i^{gr}+\epsilon {^CD}^{\beta,\rho}_{a^+}{ w}_i^{gr}(i=1,2)$ are  real functions of class $C^1$ (see e.g. \cite{GH96}) imply that $F$ is differentiable at $\bar{\epsilon}=0$. By applying the chain rules for $f^{gr}(\phi_1(\epsilon),\phi_2(\epsilon),
		\phi_1(\epsilon),\phi_2(\epsilon),t)$, we obtain
		\begin{equation}\label{prop3.1-5}
			\delta J^{gr}({\bf x}^{*gr},{\bf w}^{gr}):=\frac{d F(\epsilon)}{d\epsilon}\mid _{\epsilon=0}=\int_{a}^{b}\sum\limits_{k=1}^n
			\left(
			\partial_kf^{gr}.w_k^{gr}+\partial_{k+n}f^{gr}.{^CD}^{\beta,\rho}_{a^+}{ w}_k^{gr}\right)dt.
		\end{equation}
		From Lemma \ref{lem4.5} and the fact that $w^{gr}_k(a,\mu,\alpha_{w_k})=w_k^{gr}(b,\mu,\alpha_{w_k})=0$ for all $k\in I$ implies
		\begin{align*}
			\int_{a}^{b}
			\partial_{k+n}f^{gr}.{^CD}^{\beta,\rho}_{a^+}{ w}_k^{gr}dt&=
			I_{b^-}^{1-\beta,\rho}\partial_{k+n}f^{gr}.w^{gr}_k\mid_{a}^{b}
			+\int_{a}^{b}
			{D}^{\beta,\rho}_{b^-}\partial_{k+n}f^{gr}.w_k^{gr}dt\\
			&=\int_{a}^{b}
			{D}^{\beta,\rho}_{b^-}\partial_{k+n}f^{gr}.w_k^{gr}dt.
		\end{align*}
		Combining all of this information, (\ref{prop3.1-5}), and  the fact $\bar{\epsilon}=0$ is a minimum of $F$, we entail that
		$$\int_{a}^{b}\sum\limits_{k=1}^n\left(
		\partial_kf^{gr}+{D}^{\beta,\rho}_{b^-}\partial_{k+n}f^{gr}\right)w_k^{gr}dt=0.$$
		For each $i\in I$, applying Lemma \ref{lem2.1} with $w_k(t)=0$ for all $k\in I\setminus \{i\}$, it is easy to get that
		$$\partial_if^{gr}+{D}^{\beta,\rho}_{b^-}\partial_{i+n}f^{gr}=0,\forall \mu,\alpha_f\in [0,1], \forall i\in I.$$
		The proof is complete.
	\end{proof}

	\begin{definition} (see e.g. \cite{A17,MO20,TAT15,TR20,TT23,TTKN25})
		The real-valued function $\phi:\mathbb{R}^2\times \mathbb{R}^2\times [a, b] \times [0,1]^2 \to \mathbb{R}$ is said to be convex with respect to first $4$-arguments in $S\subset \mathbb{R}^{7}$ if
		$$\phi[{\bf u}+{\bf p},{\bf v}+{\bf q}](t,\mu,\alpha_{\phi})
		-\phi[{\bf u},{\bf v}](t,\mu,\alpha_{\phi})$$
		$$\ge \sum\limits_{k=1}^2\left(\partial_k \phi[{\bf u},{\bf v}](t,\mu,\alpha_{\phi})p_k+
		\partial_{k+2}\phi[{\bf u},{\bf v}](t,\mu,\alpha_{\phi})q_k\right)$$
		for all $(u_1,u_{2},v_1,v_2,t,\mu,\alpha_{\phi}), (u_1+p_1,u_{2}+p_{2},v_1+q_1,v_2+q_2,t,\mu,\alpha_{\phi})$ in $S$ and $\partial_k (k=1,...,4)$ is the partial derivative of function $\phi(.,.,.,.,t,\mu,\alpha_{\phi})$ with respect to its $k$th argument.
	\end{definition}

	\begin{proposition}\label{prop3.3} Suppose that $\mathcal{H}:\widetilde{\Omega}\subset\mathbf{F}_C(\mathbb{R})\to \Omega^{gr}\subset\mathbb{R}$, defined by $\mathcal{H}(\widetilde{{\bf x}}(t))={\bf x}^{gr}$, is a bijective map and ${\bf x}^{*gr}$
		satisfies
		\begin{equation}\label{prop3.15-1}
			\partial_if^{gr}+{D}^{\beta,\rho}_{b^-}\partial_{i+n}f^{gr}=0, \forall \mu,\alpha_f\in[0,1], \forall i\in I.
		\end{equation}
		If $f^{gr}[{\bf x}^{gr},{^CD}^{\beta,\rho}_{a^+}{\bf x}^{gr}](t,\mu,\alpha_f)$ is convex with respect to first $4$-arguments in $X_{ad}^{gr}$, then ${\bf x}^{*gr}$ is a minimal solution of (GP) and $\widetilde{{\bf x}}^*(t)=\mathcal{H}^{-1}({\bf x}^{*gr})$ is a minimal solution of (P).
	\end{proposition}
	\begin{proof} Since $f^{gr}[{\bf x}^{gr},{^CD}^{\beta,\rho}_{a^+}{\bf x}^{gr}](t,\mu,\alpha_f)$ is convex with respect to first $4$-arguments  in $\Omega_{ad}$, for any admissible function ${\bf x}^{*gr}+{\bf w}^{gr}$ with $w_i^{gr}(a,\mu,\alpha_{w_i})=w_i^{gr}(b,\mu,\alpha_{w_i})=0$ for all $i\in I$,  we have
		
		\noindent $J^{gr}({\bf x}^{*gr}+{\bf w}^{gr})-J^{gr}({\bf x}^{*gr})$
		$$= \int_{a}^{b}
		f^{gr}[{\bf x}^{*gr}+
		{\bf w}^{gr},{^CD}^{\beta,\rho}_{a^+}{\bf x}^{*gr}+
		{^CD}^{\beta,\rho}_{a^+}{\bf w}^{gr}](t,\mu,\alpha_f)dt-\int_{a}^{b}
		f^{gr}[{\bf x}^{*gr},{^CD}^{\beta,\rho}_{a^+}{\bf x}^{*gr}](t,\mu,\alpha_f)
		dt
		$$
		$$= \int_{a}^{b}
		\left(f^{gr}[{\bf x}^{*gr}+
		{\bf w}^{gr},{^CD}^{\beta,\rho}_{a^+}{\bf x}^{*gr}+
		{^CD}^{\beta,\rho}_{a^+}{\bf w}^{gr}](t,\mu,\alpha_f)-
		f^{gr}[{\bf x}^{*gr},{^CD}^{\beta,\rho}_{a^+}{\bf x}^{*gr}](t,\mu,\alpha_f)\right)
		dt
		$$
		$$\ge  \int_{a}^{b}
		\sum\limits_{k=1}^2
		\left(\partial_kf^{gr}[{\bf x}^{*gr},{^CD}^{\beta,\rho}_{a^+}{\bf x}^{*gr}](t,\mu,\alpha_f)w_k^{gr}
		+\partial_{k+2}f^{gr}[{\bf x}^{*gr},{^CD}^{\beta,\rho}_{a^+}{\bf x}^{*gr}](t,\mu,\alpha_f) {^CD}^{\beta,\rho}_{a^+}{w}_k^{gr}\right)dt.
		$$
		Proceeding exactly as in the proof of Proposition \ref{prop3.2}, we deduce from (\ref{prop3.15-1}) that
		\begin{eqnarray*}
			J^{gr}({\bf x}^{*gr}+{\bf w}^{gr})-J^{gr}({\bf x}^{*gr})&\ge&
			\int_{a}^{b}
			\sum\limits_{k=1}^2
			\left(\partial_kf^{gr}.w_k^{gr}
			+\partial_{k+2}f^{gr}. {^CD}^{\beta,\rho}_{a^+}{w}_k^{gr}\right)dt\\
			&\ge&\int_{a}^{b}\sum\limits_{k=1}^2		\left(\partial_kf^{gr}+{D}^{\beta,\rho}_{b^-}\partial_{k+2}f^{gr}
			\right)w_k^{gr}dt\\
			&=& 0.
		\end{eqnarray*}
		We conclude that ${\bf x}^{*gr}$ is a minimal of (GP) and $\widetilde{{\bf x}}^*(t)=\mathcal{H}^{-1}({\bf x}^{*gr})$ is a minimal of (P) by Proposition \ref{prop3.1}.
	\end{proof}

	\begin{example}\label{exa3.1} Consider the (P) with $\beta\in\left[\frac{1}{2},1\right]$ as follows
		
		\noindent (P) $\widetilde{\min}$ $\displaystyle\widetilde{J}(\widetilde{{\bf x}})={\widetilde{\int}}_{0}^{1}
		\left(({_{gr}^C\mathcal{D}}^{\beta,\rho}_{0^+}
		\widetilde{x}_1(t))^2+({_{gr}^C\mathcal{D}}^{\beta,\rho}_{0^+}
		\widetilde{x}_2(t))^2+2
		\widetilde{x}_2(t)\right)dt$
		
		\noindent s.t. $\widetilde{x}_1(0)=\widetilde{x}_2(0)=(0,0,0),\widetilde{x}_1(1)=(1,2,3),
		\widetilde{x}_2(1)=(0,1,2)$.
		
		\noindent The granular problem corresponding to the problem (P) is (GP)
		
		\noindent  min $\displaystyle J^{gr}({\bf x}^{gr})=\int_{0}^{1}
		\left(({^CD}^{\beta,\rho}_{0^+}
		x_1^{gr}(t,\mu,\alpha_{x_1}))^2+({^CD}^{\beta,\rho}_{0^+}
		x_2^{gr}(t,\mu,\alpha_{x_2}))^2+2
		x_2^{gr}(t,\mu,\alpha_{x_2})\right)dt$
		
		\noindent s.t. $x_1^{gr}(0,\mu,\alpha_{x_1})=x_2^{gr}(0,\mu,\alpha_{x_1})=0,
		x_1^{gr}(1,\mu,\alpha_{x_1})=1+\mu+(2-2\mu)\alpha_{u_1},$
		$x_2^{gr}(1,\mu,\alpha_{x_2})=\mu+(2-2\mu)\alpha_{u_2},\;\forall \alpha_{x_1}=\alpha_{u_1}, \alpha_{x_2}=\alpha_{u_2}$ ,where
		$\widetilde{u}_1=(1,2,3), \widetilde{u}_2=(0,1,2)$.
		
		\noindent Since $f^{gr}[{\bf x}^{gr},{^CD}^{\beta,\rho}_{0^+}
		{\bf x}^{gr}](t,\mu,\alpha_f)=({^CD}^{\beta,\rho}_{0^+}
		x_1^{gr})^2+({^CD}^{\beta,\rho}_{0^+}
		x_2^{gr})^2+2
		x_2^{gr}$, we deduce from the granular Euler-Lagrange conditions that
		$$0=\partial_1f^{gr}+{D}^{\beta,\rho}_{1^-}\partial_3f^{gr}=0+{D}^{\beta,\rho}_{1^-}
		(2.{^CD}^{\beta,\rho}_{0^+}
		x^{gr}),$$
		$$0=\partial_2f^{gr}+{D}^{\beta,\rho}_{1^-}\partial_4f^{gr}=2+{D}^{\beta,\rho}_{1^-}
		(2.{^CD}^{\beta,\rho}_{0^+}
		x^{gr}).$$
		So, we have the granular equations
		\begin{equation}\label{exa3.11.-1}
			\left\{\begin{array}{l}
				{D}^{\beta,\rho}_{1^-}
				({^CD}^{\beta,\rho}_{0^+}
				x_1^{gr}(t,\mu,\alpha_{x_1}))=0,\\
				x_1^{gr}(0,\mu,\alpha_{x_1})=0,
				x_1^{gr}(1,\mu,\alpha_{x_1})=1+\mu+(2-2\mu)\alpha_{u_1},
			\end{array}\right.
		\end{equation}
		and
		\begin{equation}\label{exa3.11.-2}
			\left\{\begin{array}{l}
				{D}^{\beta,\rho}_{1^-}
				({^CD}^{\beta,\rho}_{0^+}
				x_2^{gr}(t,\mu,\alpha_{x_2}))=-1,\\
				x_2^{gr}(0,\mu,\alpha_{x_2})=0,
				x_2^{gr}(1,\mu,\alpha_{x_2})=\mu+(2-2\mu)\alpha_{u_2}.
			\end{array}\right.
		\end{equation}
		Suppose that ${\bf x}^{*gr}=(x_1^{*gr}(t,\mu,\alpha_{x_1^*}),x_2^{*gr}(t,\mu,\alpha_{x_2^*}))$ is the solution of the above equations. It follows from  Lemma \ref{lem4.6} (i) that ${D}^{\beta,\rho}_{1^-}
		({^CD}^{\beta,\rho}_{0^+}
		x_1^{*gr}(t,\mu,\alpha_{x_1^*}))=0$ iff
		$${^CD}^{\beta,\rho}_{0^+}
		x_1^{*gr}(t,\mu,\alpha_{x_1^*})
=C_1\left(\frac{1-t^\rho}{\rho}\right)^{\beta-1}.$$
		Moreover, using the results in Lemma \ref{lem4.6} (iii), we deduce from the above equation, $x_1^{*gr}(0,\mu,\alpha_{x_1^*})=0$ and Lemma \ref{lem2.15} (i) that
\begin{eqnarray*}
x_1^{*gr}(t,\mu,\alpha_{x_1^*})&=&\frac{1}{\Gamma(\beta)}
		\int_0^t\left(\frac{t^{\rho}-
\tau^{\rho}}{\rho}\right)^{\beta-1}\tau^{\rho-1}.		C_1\left(\frac{1-\tau^\rho}{\rho}\right)^{\beta-1}d\tau\\
&=&C_1\frac{\rho^{2-2\beta}}{\Gamma(\beta)}.
		\int_0^t\left(t^{\rho}-
\tau^{\rho}\right)^{\beta-1}\tau^{\rho-1}.		\left(1-\tau^\rho\right)^{\beta-1}d\tau\\
&=&C_1\frac{\rho^{2-2\beta}}{\Gamma(\beta)}.
		\frac{t^{\beta \rho}}{\beta\rho}.
			{_2F_1}(1,1-\beta;1+\beta;t^\rho)\\
&=&C_1\frac{\rho^{1-2\beta}}{\beta\Gamma(\beta)}.
		t^{\beta \rho}
			{_2F_1}(1,1-\beta;1+\beta;t^\rho).
\end{eqnarray*}
		It follows from Lemma \ref{lem2.15} (ii) and $x_1^{*gr}(1,\mu,\alpha_{x^*_1})=1+\mu+(2-2\mu)\alpha_{u_1}$ for all $\alpha_{x^*_1}=\alpha_{u_1}$ that
		$$C_1\frac{\rho^{1-2\beta}}{\beta\Gamma(\beta)}.	1^{\beta\rho}\frac{\beta\rho}{2\beta-1}=1+\mu+(2-2\mu)\alpha_{u_1}\;
\mbox{or}\;		C_1=(1+\mu+(2-2\mu)\alpha_{u_1})\rho^{2-2\beta}\Gamma(\beta)(2\beta-1)$$
		which in turn implies that the solution of (\ref{exa3.11.-1}) is
		$$x_1^{*gr}(t,\mu,\alpha_{x_1^*})=(1+\mu+(2-2\mu)\alpha_{x_1^*})
		\frac{(2\beta-1)\rho^{3-4\beta}}{\beta}.t^{\beta\rho}.
		{_2F_1}(1,1-\beta;1+\beta;t^\rho).$$
Since
$${D}^{\beta,\rho}_{1^-}\left({^CD}^{\beta,\rho}_{0^+}
			x_2^{gr}(t,\mu,\alpha_{x_2})\right)=-1\Leftrightarrow
{D}^{\beta,\rho}_{1^-}\left(-{^CD}^{\beta,\rho}_{0^+}
				x_2^{gr}(t,\mu,\alpha_{x_2}))\right)=1,$$
we deduce from Lemma \ref{lem4.6} that
	$${^CD}^{\beta,\rho}_{0^+}	x_2^{*gr}(t,\mu,\alpha_{x_2^*})=-
\left(\frac{1}{\Gamma(1+\beta)}\left(\frac{1-t^\rho}{\rho}\right)^{\beta}
		+C_2\left(\frac{1-t^\rho}{\rho}\right)^{\beta-1}\right).$$
		Moreover, using the results in Lemma \ref{lem4.6} (iii), we deduce from the above equation and $x_2^{*gr}(0,\mu,\alpha_{x_2^*})=0$ that
		$$x_2^{*gr}(t,\mu,\alpha_{x_2^*})=-\frac{1}{\Gamma(\beta)}		\int_0^t\left(\frac{t^\rho-\tau^\rho}{\rho}\right)^{\beta-1}\tau^{\rho-1}
.\left(\frac{1}{\Gamma(1+\beta)}\left(\frac{1-\tau^\rho}{\rho}\right)^{\beta}
+C_2\left(\frac{1-\tau^\rho}{\rho}\right)^{\beta-1}\right)d\tau
		$$
$$=-\frac{\rho^{1-2\beta}}
{\Gamma(\beta)\Gamma(1+\beta)}\int_0^t(t-\tau)^{\beta-1}	\tau^{\rho-1}(1-\tau^\rho)^{\beta}d\tau-
\frac{C_2\rho^{2-2\beta}}{\Gamma(\beta)}
\int_0^t(t^\rho-\tau^\rho)^{\beta-1}\tau^{\rho-1}
		(1-\tau^\rho)^{\beta-1}d\tau.$$
		It follows from Lemma \ref{lem2.15} (i) that
		$$x_2^{*gr}(t,\mu,\alpha_{x_2^*})=
-\frac{\rho^{1-2\beta}}{\Gamma(\beta)\Gamma(1+\beta)}.
	\frac{t^{\beta\rho}}{\beta\rho}{_2F_1}(1,-\beta;1+\beta;t^\rho)
		-\frac{C_2\rho^{2-2\beta}}{\Gamma(\beta)}.
\frac{t^{\beta\rho}}{\beta\rho}
		{_2F_1}(1,1-\beta;1+\beta;t^\rho).$$
$$=
-\frac{\rho^{-2\beta}}{\beta\Gamma(\beta)\Gamma(1+\beta)}.
	t^{\beta\rho}{_2F_1}(1,-\beta;1+\beta;t^\rho)
		-\frac{C_2\rho^{1-2\beta}}{\beta\Gamma(\beta)}.
t^{\beta\rho}
		{_2F_1}(1,1-\beta;1+\beta;t^\rho).$$
		By using the endpoint conditions $x_2^{*gr}(1,\mu,\alpha_{x_2^*})=\mu+(2-2\mu)\alpha_{u_2}$ with $\alpha_{x_2^*}=\alpha_{u_2}$, one has
		$$		\mu+(2-2\mu)\alpha_{u_2}=-\frac{\rho^{-2\beta}}
{\Gamma(\beta)\Gamma(1+\beta)\beta}.
		{_2F_1}(1,-\beta;1+\beta;1)		-\frac{C_2\rho^{1-2\beta}}{\Gamma(\beta)\beta}.
{_2F_1}(1,1-\beta;1+\beta;1),
		$$
		which together with Lemma \ref{lem4.6} (ii) entails that
		$$\mu+(2-2\mu)\alpha_{u_2}=
-\frac{\rho^{-2\beta}}{\Gamma(\beta)\Gamma(1+\beta)\beta}.
		\frac{1}{2}		-\frac{C_2\rho^{1-2\beta}}{\Gamma(\beta)\beta}.\frac{\beta}{2\beta-1}.$$
		Hence,		$$C_2=-(2\beta-1)\left(\rho^{1-2\beta}\Gamma(\beta)
(\mu+(2-2\mu)\alpha_{u_2})
		+\frac{1}
		{2\beta\rho\Gamma(1+\beta)}\right).$$
		Substituting into $x_2^{*gr}(t,\mu,\alpha_{x_2^*})$, one has
		$$x_2^{*gr}(t,\mu,\alpha_{x_2^*})=
		-\frac{\rho^{-2\beta}}{\beta\Gamma(\beta)\Gamma(1+\beta)}.
		t^{\beta\rho}{_2F_1}(1,-\beta;1+\beta;t^\rho)$$
		$$+\frac{\rho^{1-2\beta}(2\beta-1)}{\Gamma(\beta)\beta}
	\left(\rho^{2\beta-1}\Gamma(\beta)(\mu+(2-2\mu)\alpha_{x_2^*})
		+\frac{1}		{2\beta\rho\Gamma(1+\beta).\beta}\right).t^{\beta\rho}
{_2F_1}(1,1-\beta;1+\beta;t^\rho).$$
		Moreover, since
		
		$f^{gr}[{\bf u}+{\bf p},{\bf v}+{\bf q}](t,\mu,\alpha_f)
		-f^{gr}[{\bf u}, {\bf v}](t,\mu,\alpha_f)$
		\begin{eqnarray*}
			&=&(v_1+q_1)^2+(v_2+q_2)^2+2(u_2+p_2)-(v_1^2+v_2^2+2u_2)
			\\
			&=& 2v_1q_1+2v_2q_2+2p_2+q_1^2+q_2^2
			\\
			&\ge& 0.p_1+2p_2+2v_1q_1+2v_2q_2
			\\
			&=& \partial_1 f^{gr}p_1+
			\partial_2f^{gr}p_2+\partial_3 f^{gr}q_1+
			\partial_4f^{gr}q_2,
		\end{eqnarray*}
		$f^{gr}$ is convex with respect to first 4-arguments in $\Omega^{gr}$. Invoking Proposition \ref{prop3.3}, we obtain that ${\bf x}^{*gr}=(x_1^{*gr}(t,\mu,\alpha_{x_1^*}),x_2^{*gr}(t,\mu,\alpha_{x_2^*}))$ is  a minimum of (GP) and
		$\widetilde{{\bf x}}^*(t)=\mathcal{H}^{-1}({\bf x}^{*gr}))$ whose $\mu$-level sets are
		$$[\widetilde{x}_1^*(t)]^{\mu}=\left[\frac{(1+\mu)
			(2\beta-1)\rho^{3-4\beta}w_2(\beta,t)}
		{\beta},\frac{(3-\mu)(2\beta-1)\rho^{3-4\beta}w_2(\beta,t)}
		{\beta}\right],$$
		$$[\widetilde{x}_2^*(t)]^{\mu}
=\left[-\frac{\rho^{-2\beta}w_1(\beta,t)}
		{\Gamma(\beta)\Gamma(1+\beta)\beta}
+\frac{\rho^{1-2\beta}(2\beta-1)}{\Gamma(\beta)\beta}
		\left(\rho^{2\beta-1}\Gamma(\beta)\mu
		+\frac{1}
		{2\Gamma(1+\beta).\beta\rho}\right)w_2(\beta,t),
		\right.$$
		$$\left. -\frac{\rho^{-2\beta}w_1(\beta,t)}
		{\Gamma(\beta)\Gamma(1+\beta)\beta}
+\frac{\rho^{1-2\beta}(2\beta-1)}{\Gamma(\beta)\beta}
		\left(\rho^{2\beta-1}\Gamma(\beta)(2-\mu)
		+\frac{1}
		{2\Gamma(1+\beta).\beta\rho}\right)w_2(\beta,t)\right]$$
		is a minimal solution of (P), where $w_1(\beta,t)=t^{\beta\rho}.{_2F_1}(1,-\beta;1+\beta;t^\rho)$ and $w_2(\beta,t)=t^{\beta\rho}.{_2F_1}(1,1-\beta;1+\beta;t^\rho)$. See the graph of $(\widetilde{x}_1^*(t),\widetilde{x}_2^*(t))$ with $\beta=\frac{3}{4},\rho=4$ in Figure 3.1.1 and Figure 3.1.2 and the graph of $(\widetilde{x}_1^*(t),\widetilde{x}_2^*(t))$ with $\beta=\frac{3}{4},\rho=\frac{1}{2}$ in Figure 3.1.3 and Figure 3.1.4.
\begin{figure}[htbp]
    \centering
      \begin{tabular}{cc}
        \includegraphics[width=0.45\textwidth]{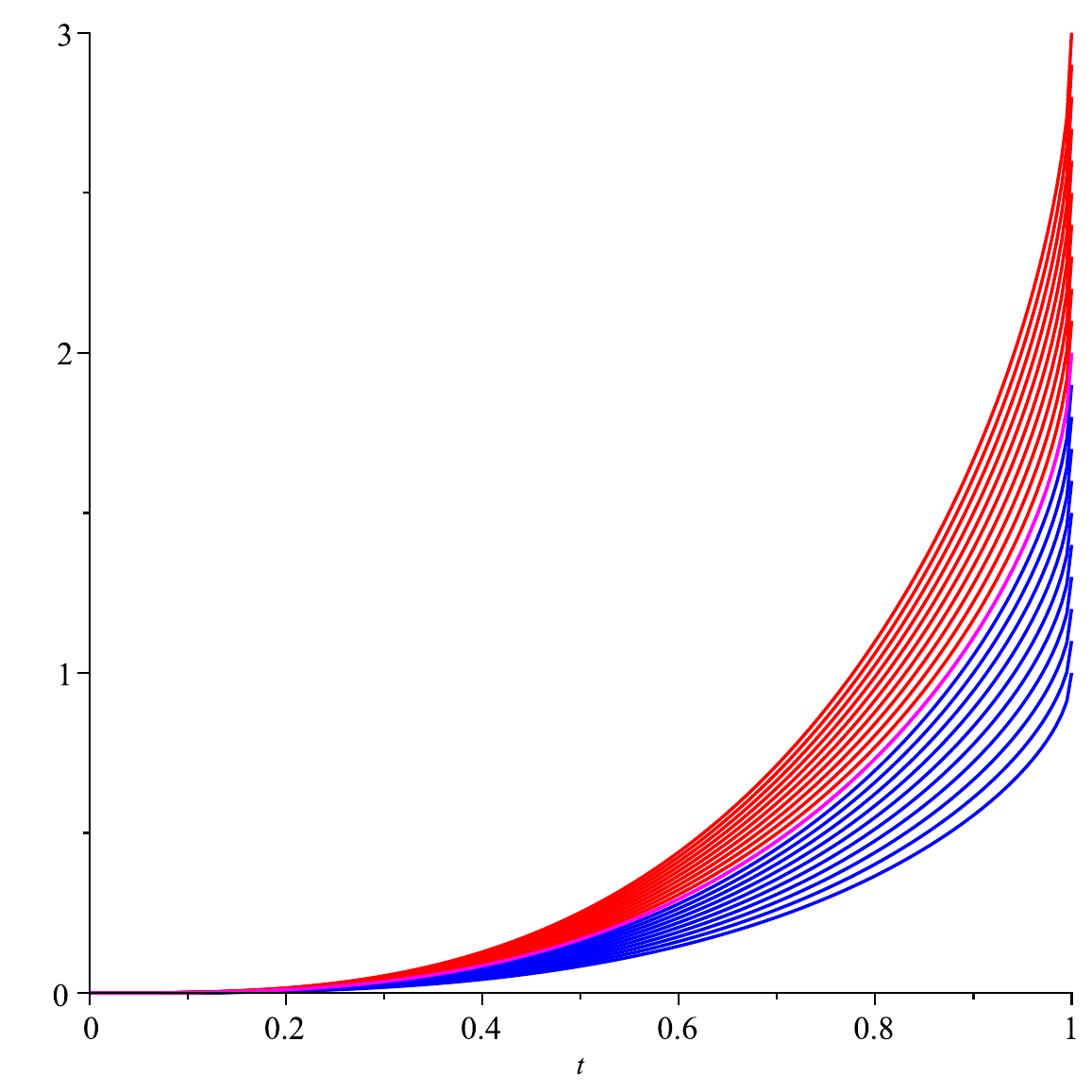} &
        \includegraphics[width=0.45\textwidth]{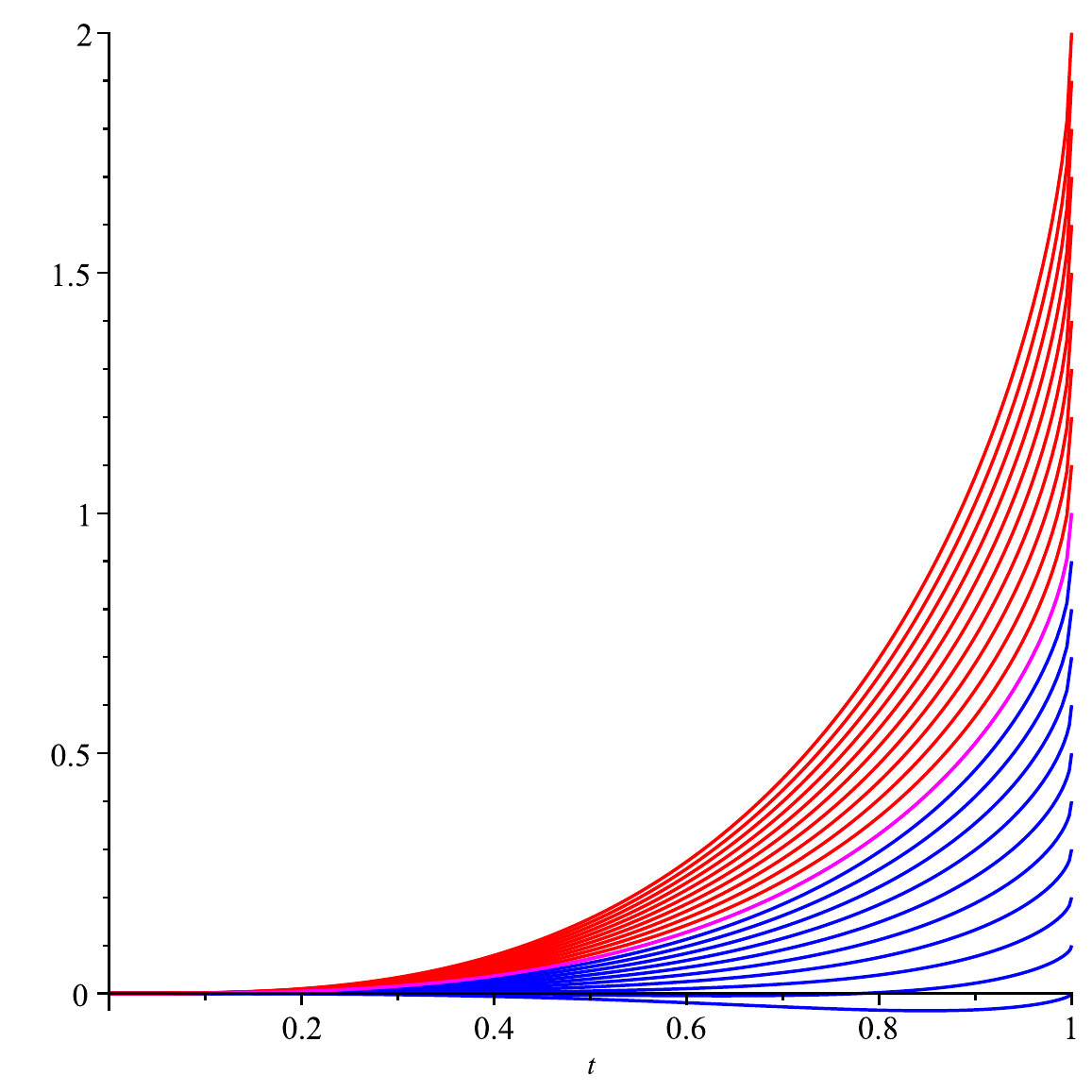} \\
        Figure 3.1.1. $\widetilde{x}_1^*(t)$ with  $\beta=\frac{3}{4},\rho=4$ &
        Figure 3.1.2. $\widetilde{x}_2^*(t)$ with $\beta=\frac{3}{4},\rho=4$
    \end{tabular}
\end{figure}

\begin{figure}[htbp]
    \centering
    \begin{tabular}{cc}
        \includegraphics[width=0.45\textwidth]{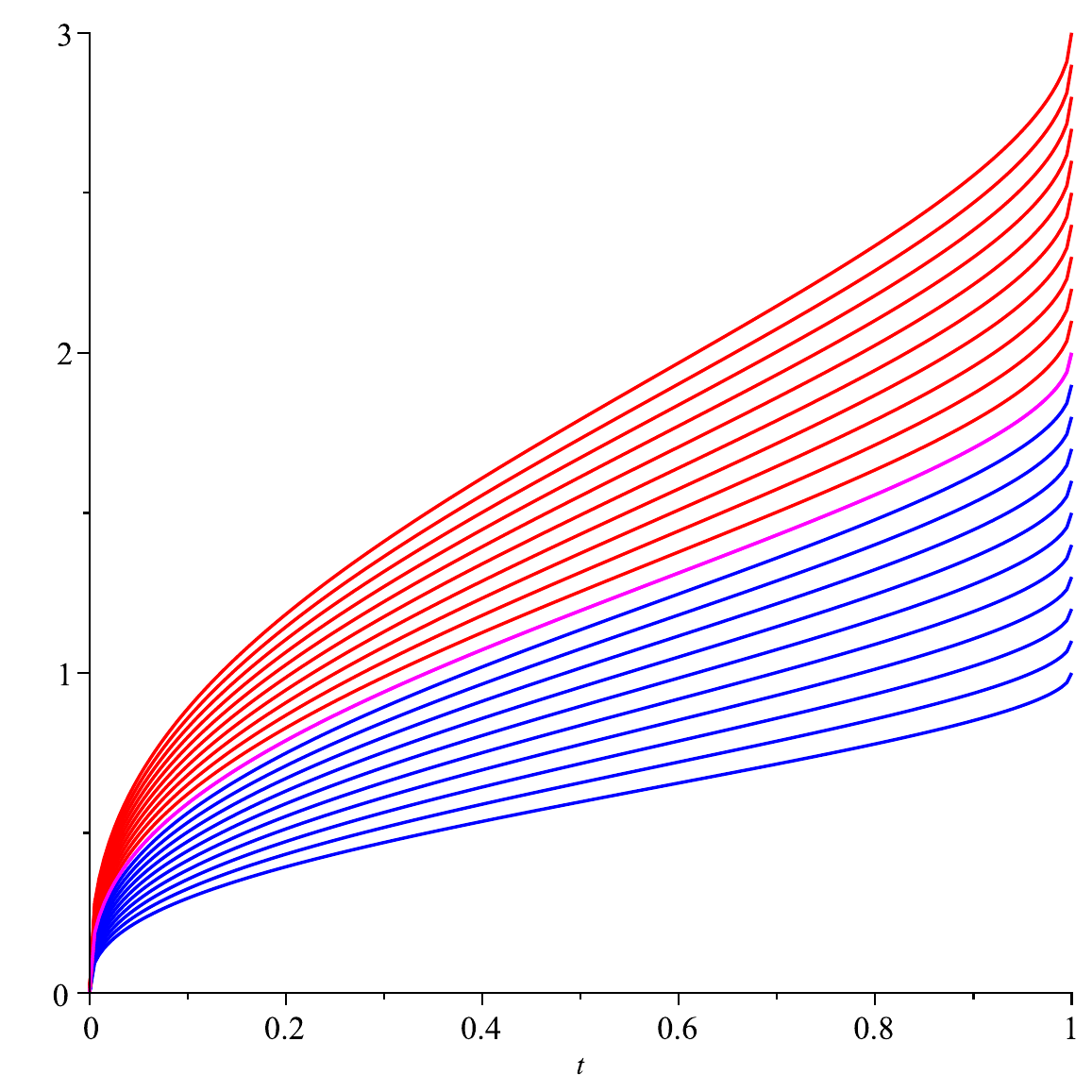} &
        \includegraphics[width=0.45\textwidth]{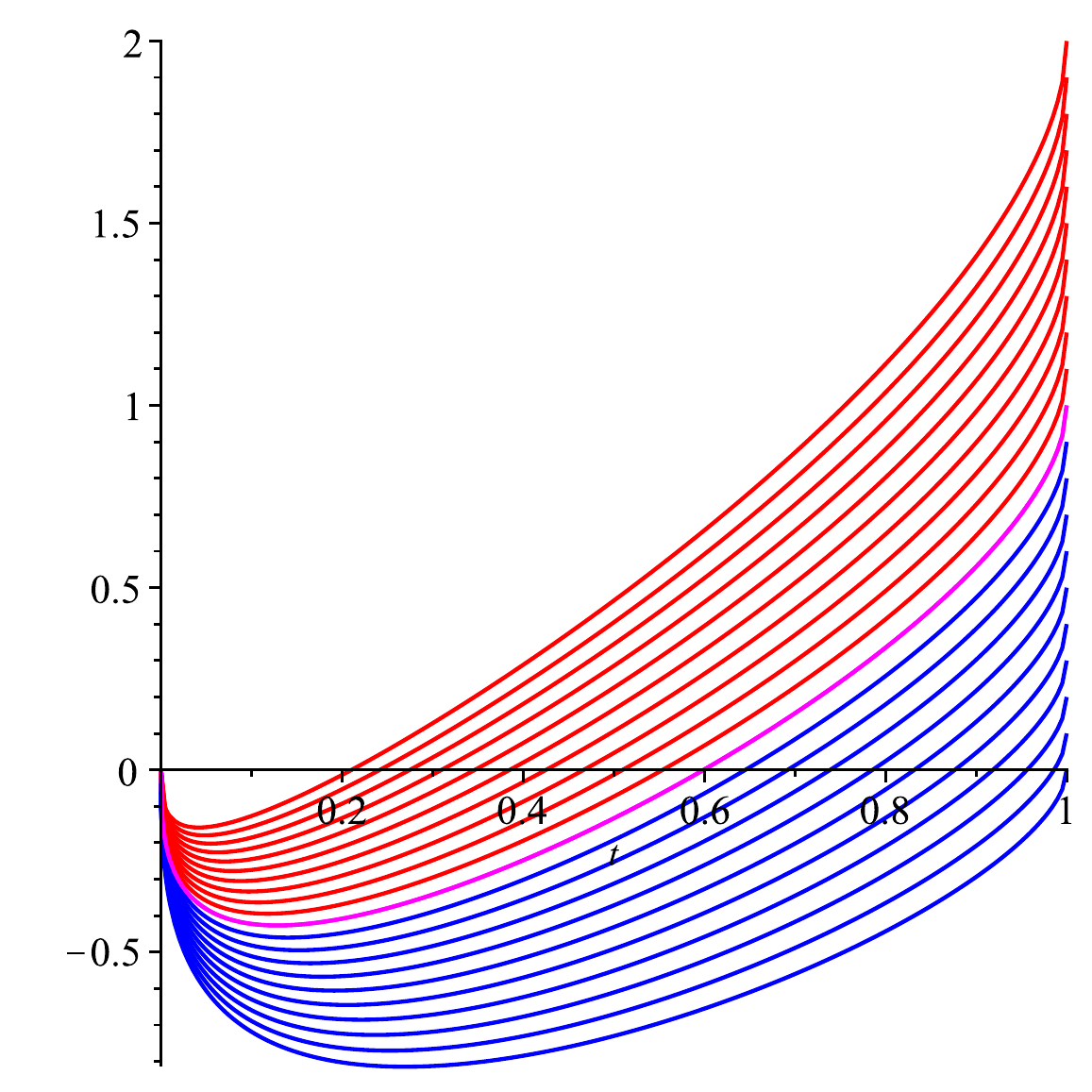} \\
        Figure 3.1.3. $\widetilde{x}_1^*(t)$ with $\beta=\frac{3}{4},\rho=\frac{1}{2}$ &
        Figure 3.1.4. $\widetilde{x}_2^*(t)$ with $\beta=\frac{3}{4},\rho=\frac{1}{2}$
    \end{tabular}
   \end{figure}	
\end{example}	
	
\subsection{Fuzzy fractional isoperimetric problem}
	In this section, let $\tilde{h}: \mathbf{F}_C(\mathbb{R})^{4} \times \mathbb{R}\to \mathbf{F}_C(\mathbb{R})$ be a fuzzy function
	of class $\mathcal{C}^1$ and $\widetilde{H}:\mathcal{C}^1([a,b])^2\to \mathbf{F}_C(\mathbb{R})$ . We consider the fuzzy fractional isoperimetric problem as follows:
	
	\noindent (FIP) $\widetilde{\min}$ $\displaystyle\widetilde{J}(\widetilde{{\bf x}})={\widetilde{\int}}_{a}^{b}
	\widetilde{f}[\widetilde{{\bf x}},{_{gr}^C\mathcal{D}}^{\beta,\rho}_{a^+}\widetilde{{\bf x}}](t)dt$
	
	\indent  s.t. $\displaystyle\widetilde{H}(\widetilde{{\bf x}})={\widetilde{\int}}_{a}^{b}
	\widetilde{h}[\widetilde{{\bf x}},{_{gr}^C\mathcal{D}}^{\beta,\rho}_{a^+}\widetilde{{\bf x}}](t)dt=\widetilde{l}$,
	
	\indent\indent$\widetilde{x}_i(a)=\widetilde{x}_{0,i},\widetilde{x}_i(b)=
	\widetilde{x}_{1,i}, \forall i\in I=\{1,2\}.$

	\begin{definition} A fuzzy function $\widetilde{{\bf x}}^*=\widetilde{{\bf x}}^*(t)$ in $\widetilde{\Theta}$ is a minimal solution of (FIP) if
			$$\Delta\widetilde{J}:=\widetilde{J}(\widetilde{{\bf x}})\ominus_{gr}
			\widetilde{J}(\widetilde{{\bf x}}^*)\ge \widetilde{0},$$
			for all admissible fuzzy functions $\widetilde{{\bf x}}\in \widetilde{\Theta}$.
	\end{definition}

	\noindent The granular problem corresponding to the problem (FIP) is
	
	\noindent (GIP) min $\displaystyle J^{gr}({\bf x}^{gr})=\int_{a}^{b}
	f^{gr}[{\bf x}^{gr},
	{^CD}^{\beta,\rho}_{a^+}{\bf x}^{gr}](t,\mu,\alpha_f)dt$
	
	\noindent  s.t. $\displaystyle H^{gr}({\bf x}^{gr})=\int_{a}^{b}
	h^{gr}[{\bf x}^{gr},{^CD}^{\beta,\rho}_{a^+}{\bf x}^{gr}](t,\mu,\alpha_{h})dt=l^{gr}(\mu,\alpha_{l}),
	\forall \alpha_{h}=\alpha_{l},$
	
	\indent\indent$x_i^{gr}(a,\mu,\alpha_{x_i})=x_{0,i}^{gr}(\mu,\alpha_{x_{0,i}}),
	x_i^{gr}(b,\mu,\alpha_{x_i})=x_{1,i}^{gr}(\mu,\alpha_{x_{1,i}}),$
	
	\indent\indent$\forall \alpha_{x_i}=\alpha_{x_{0,i}}=\alpha_{x_{1,i}},\forall i\in I.$

\begin{definition} Let $\mu,\alpha_{x_i}\in [0,1] (i=1,2)$.
		 A granular  function $\mathcal{H}(\widetilde{{\bf x}}^*(t))={\bf x}^{*gr}$ in the admissible set $\Theta^{gr}$ is a minimal solution of (GIP) if
$\Delta J^{gr}:=J^{gr}({\bf x}^{gr})-
			J^{gr}({\bf x}^{*gr})\ge 0,\forall {\bf x}^{gr}\in\; \Theta^{gr}$.
	\end{definition}

	\begin{proposition}\label{prop3.4}
		Suppose that $\mathcal{H}:\;\widetilde{\Theta}\subset\mathbf{F}_C(\mathbb{R})\to \Theta^{gr}\subset\mathbb{R}$, defined by $\mathcal{H}(\widetilde{{\bf x}}(t))={\bf x}^{gr}$, is bijective map. Then,
		\begin{enumerate}
			\item[(i)] $\mathcal{H}(\widetilde{\Theta})=\;\Theta^{gr}$ and $\mathcal{H}^{-1}( \;\Theta^{gr})=\widetilde{\Theta}$.
			\item[(ii)] If $\widetilde{{\bf x}}^*$ is a minimal solution of (FIP) then $\mathcal{H}(\widetilde{{\bf x}}^*)={\bf x}^{*gr}$ is a minimal solution of (GIP).
			\item[(iii)] If ${\bf x}^{*gr}$ is a minimal solution of (GIP) then $\widetilde{{\bf x}}^*=\mathcal{H}^{-1}({\bf x}^{*gr})$ is a minimal solution of (FIP).
		\end{enumerate}
	\end{proposition}
\begin{proposition}\label{prop3.5} Let $\widetilde{{\bf x}}^*\in \widetilde{\Theta}$ be a minimizer for the problem (FIP). Suppose that there are functions ${\bf w}_1^{gr}\in (C^1[a,b])^2$ with ${\bf w}_1^{gr}(a)={\bf w}_1^{gr}(b)={\bf 0}=(0,0) $ such that
		$\delta H^{gr}({\bf x}^{*gr},{\bf w}_1^{gr})\neq 0$. Then, there exist  fuzzy  numbers  $\lambda=\lambda^{gr}(\mu,\alpha_{\lambda})\in \mathbb{R}$ such that,  ${\bf x}^{*gr}=\mathcal{H}(\widetilde{{\bf x}}^*)$ is a solution of the equations
		$$\partial_iK^{gr}[{\bf x}^{*gr},{^CD}^{\beta,\rho}_{a^+}{\bf x}^{*gr}](t,\mu,\alpha_K)+{D}^{\beta,\rho}_{b^-}\left(
		\partial_{i+2}K^{gr}[{\bf x}^{*gr},{^CD}^{\beta,\rho}_{a^+}{\bf x}^{*gr}](t,\mu,\alpha_K)\right)=0,\forall \mu,\alpha_K, i\in I,$$
		
\noindent where ${ K^{gr}[{\bf x}^{gr},{^CD}^{\beta,\rho}_{a^+}{\bf x}^{gr}](t,\mu,\alpha_K)=
			f^{gr}[{\bf x}^{gr},{^CD}^{\beta,\rho}_{a^+}{\bf x}^{gr}](t,\mu,\alpha_f)+
			\lambda
			h^{gr}[{\bf x}^{gr},{^CD}^{\beta,\rho}_{a^+}{\bf x}^{gr}](t,\mu,\alpha_{h})}$.
\end{proposition}

	\begin{proof} Consider a variation of $\widetilde{{\bf x}}^*(t)$ with $2$ parameters $\widetilde{{\bf x}}^*(t)+\epsilon_1\widetilde{{\bf w}}_1(t)+\epsilon_2\widetilde{{\bf w}}_2(t)$
		where $\epsilon_i$ are  small real numbers, $\widetilde{{\bf w}}_j(t)=\mathcal{H}^{-1}({\bf w}_j^{gr})(j=1,2)$ with $\delta H^{gr}({\bf x}^{*gr},{\bf w}_1^{gr})\neq 0$, and $\widetilde{{\bf w}}_2(t)\in \mathcal{C}^1([a,b])$ satisfying $\widetilde{{\bf w}}_2(a)=\widetilde{{\bf w}}_{2}(b)=\widetilde{{\bf 0}}$. The maps $F,C:\mathbb{R}^n\to \mathbb{R}$, for $2$ parameters  $(\epsilon_1,\epsilon_2)$ in a neighborhood of zero, defined by
		$$F(\epsilon_1,\epsilon_2)=
		\mathcal{H}(\widetilde{J}(\widetilde{{\bf x}}^*+\epsilon_1\widetilde{{\bf w}}_1+
		\epsilon_2\widetilde{{\bf w}}_2))$$
		\begin{eqnarray*}
			&=&\mathcal{H}\left({\widetilde{\int}}_{a}^{b}
			\widetilde{f}[\widetilde{{\bf x}}^*+\epsilon_1\widetilde{{\bf w}}_1+
			\epsilon_2\widetilde{{\bf w}}_2,{_{gr}^C\mathcal{D}}^{\beta,\rho}_{a^+}\widetilde{{\bf x}}^*+
			\epsilon_1{_{gr}^C\mathcal{D}}^{\beta,\rho}_{a^+}{\bf w}_1+	\epsilon_2{_{gr}^C\mathcal{D}}^{\beta,\rho}_{a^+}{\bf w}_2](t)dt\right)\\
			&=&\int_{a}^{b}\mathcal{H}\left(\widetilde{f}
			[\widetilde{{\bf x}}^*+\epsilon_1\widetilde{{\bf w}}_1+
			\epsilon_2\widetilde{{\bf w}}_2,{_{gr}^C\mathcal{D}}^{\beta,\rho}_{a^+}\widetilde{{\bf x}}^*+
			\epsilon_1{_{gr}^C\mathcal{D}}^{\beta,\rho}_{a^+}{\bf w}_1+\epsilon_2{_{gr}^C\mathcal{D}}^{\beta,\rho}_{a^+}{\bf w}_2](t)\right)dt\\
			&=&\int_{a}^{b}f^{gr}[{\bf x}^{*gr}+\epsilon_1 {\bf w}_1^{gr}
			+\epsilon_2 {\bf w}_2^{gr},{^CD}^{\beta,\rho}_{a^+}{\bf x}^{*gr}+\epsilon_1 {^CD}^{\beta,\rho}_{a^+}{\bf w}^{*gr}_1
			+\epsilon_2 {^CD}^{\beta,\rho}_{a^+}{\bf w}^{gr}_2](t,\mu,\alpha_f)dt,
		\end{eqnarray*}
		$$C(\epsilon_1,\epsilon_2)
		=\mathcal{H}\left({\widetilde{\int}}_{a}^{b}
		\widetilde{h}[\widetilde{{\bf x}}^*+\epsilon_1\widetilde{{\bf w}}_1+
		\epsilon_2\widetilde{{\bf w}}_2,{_{gr}^C\mathcal{D}}^{\beta,\rho}_{a^+}\widetilde{{\bf x}}^*+
		\epsilon_1{_{gr}^C\mathcal{D}}^{\beta,\rho}_{a^+}{\bf w}_1+		\epsilon_2{_{gr}^C\mathcal{D}}^{\beta,\rho}_{a^+}{\bf w}_2](t)dt\ominus_{gr}\widetilde{l}\right)$$
		$$=\int_{a}^{b}h^{gr}[{\bf x}^{*gr}+\epsilon_1 {\bf w}_1^{gr}
		+\epsilon_2 {\bf w}_n^{gr},{^CD}^{\beta,\rho}_{a^+}{\bf x}^{*gr}+\epsilon_1 {^CD}^{\beta,\rho}_{a^+}{\bf w}^{gr}_1
		+\epsilon_2 {^CD}^{\beta,\rho}_{a^+}{\bf w}^{gr}_2](t,\mu,\alpha_{h_j})dt-l^{gr}(\mu,\alpha_{l}),$$
where
\begin{equation}\label{prop3.23-1}		w_{i,k}^{gr}(a,\mu,\alpha_{w_{i,k}})
=w_{i,k}^{gr}(b,\mu,\alpha_{w_{i,k}})=0, \forall i\in I,\forall k=1,2.
\end{equation}
		The differentiability of $F$ and $C$ at $\bar{\epsilon}_i=0(i=1,2)$ could be proven in the same way as the proof of Proposition \ref{prop3.2}.
		Note that the equation $C(\epsilon_1,\epsilon_2)=0$ is not true in general. Hence, we need to find the region  of $(\epsilon_1,\epsilon_2)$ such that $C(\epsilon_1,\epsilon_2)=0$ on that region.

By  the calculation similar to that of Proposition \ref{prop3.2} and (\ref{prop3.23-1}),
		\begin{eqnarray*}
			\partial_1C({\bf 0})&=&\int_{a}^{b}\sum\limits_{k=1}^2
\left(\partial_kh^{gr}.w_{1,k}^{gr}+			\partial_{k+2}h^{gr}.{^CD}^{\beta,\rho}_{a^+}
w_{1,k}^{gr}\right)dt\\
&=&\int_{a}^{b}\sum\limits_{k=1}^2\left(\partial_kh^{gr}
			+{D}^{\beta,\rho}_{b^-}
			\partial_{k+2}h^{gr}\right)w_{1,k}^{gr}dt\\
			&=&\delta H^{gr}({\bf x}^*,{\bf{w}}_1)\neq 0.
		\end{eqnarray*}
It should be noted that the assumption ${\bf x}^{*gr}$ is not an extremal of $H^{gr}$ in \cite{A17} also derives $\delta H^{gr}({\bf x}^*,{\bf{w}}_1)\neq 0$.
By applying the Implicit Function Theorem for $C(\epsilon_1,\epsilon_2)=0$ with $C({\bf 0})=0$, $\partial_1C({\bf 0})\neq 0$, we deduce that there exists the function $\epsilon_1=\epsilon_1(\epsilon_{2})$ in a neighborhood $U$ of zero such that
		$$C(\epsilon_1(\epsilon_2), \epsilon_{2})=0, \forall \epsilon_{2}\in U.$$

		Combining the above equalities and the fact that ${\bf x}^{*gr}$ is a minimizer of (GIP), we get that
		
		\noindent$F(\epsilon_1,\epsilon_2)-
		F(0,0)=$
		$$\int_{a}^{b}f^{gr}[{\bf x}^{*gr}+\epsilon_1 {\bf w}_1^{gr}
		+\epsilon_2 {\bf w}_n^{gr},{^CD}^{\beta,\rho}_{a^+}{\bf x}^{*gr}+\epsilon_1 {^CD}^{\beta,\rho}_{a^+}{\bf w}^{*gr}_1
		+\epsilon_2 {^CD}^{\beta,\rho}_{a^+}{\bf w}^{gr}_2](t,\mu,\alpha_f)dt$$
		$$-\int_{a}^{b}f^{gr}[{\bf x}^{*gr}
		,{^CD}^{\beta,\rho}_{a^+}{\bf x}^{*gr}](t,\mu,\alpha_f)dt\ge 0,$$
		for the above $\epsilon_2\in U$ and $\epsilon_1=\epsilon_1(\epsilon_2)$, i.e., $C(\epsilon_1,\epsilon_2)=0$.
		This ensures that ${\bf 0}=(0,0)$ is a minimizer of the following constrained nonlinear programming
		
		 $$(CNP):\;\;\min\limits_{\epsilon_1,\epsilon_2} \{ F(\epsilon_1,\epsilon_2), C(\epsilon_1,\epsilon_2)=0\}.$$
		
		\noindent Since $\partial_1C({\bf 0})\neq 0$, we deduce from Lemma \ref{lem2.2} that there exists $\lambda=\lambda^{gr}(\mu,\alpha_{\lambda})$ such that
		\begin{equation}\label{prop3.9-5}
			\partial_iF(0,0)+
			\lambda \partial_iC(0,0)=0, \forall i=1,2.
		\end{equation}
		By the calculations similar to that of $\partial_1C_1({\bf 0})$, one has
		$$\partial_2F(0,0)=\int_{a}^{b}
		\sum\limits_{k=1}^2\left(\partial_kf^{gr}
		+{D}^{\beta,\rho}_{b^-}\partial_{k+2}f^{gr}
		\right)w_{2,k}^{gr}
		dt$$
		$$\partial_2C(0,0)=\int_{a}^{b}		\sum\limits_{k=1}^2\left(\partial_kh^{gr}+{D}^{\beta,\rho}_{b^-}
		\partial_{k+2}h^{gr}
		\right)w_{2,k}^{gr}
		dt.$$
		This together with (\ref{prop3.9-5}) entails that,
		$$\int_{a}^{b}
		\sum\limits_{k=1}^2\left(\partial_k(f^{gr}
		+\lambda^{gr}h^{gr})
		+{D}^{\beta,\rho}_{b^-}\partial_{k+2}(f^{gr}+
		\lambda^{gr}h^{gr})
		\right)w_{2,k}^{gr}
		dt=0.$$
		Combining this with the arbitrary of ${\bf w}_2^{gr}=(w_{2,1}^{gr},w_{2,2}^{gr})$ and Lemma \ref{lem2.1}, we could take $w_{2,2}^{gr}=0$ to obtain
		$$\partial_1(f^{gr}+\lambda^{gr}h^{gr})
		+{D}^{\beta,\rho}_{b^-}\partial_{3}(f^{gr}
		+\lambda^{gr}h^{gr})=0,$$
and $w_{2,1}^{gr}=0$ to obtain
$$\partial_2(f^{gr}+\lambda^{gr}h^{gr})
		+{D}^{\beta,\rho}_{b^-}\partial_{4}(f^{gr}
		+\lambda^{gr}h^{gr})=0.$$
Consequently,	$$\partial_iK^{gr}+{D}^{\beta,\rho}_{b^-}\partial_{i+2}K^{gr}=0,\;
		\mbox{for each}\; i\in I,$$
		which completes the proof.
	\end{proof}

	\begin{proposition}\label{prop3.6} Suppose that $\mathcal{H}:{\widetilde{\Theta}}\subset\mathbf{F}_C(\mathbb{R})\to \Theta^{gr}\subset\mathbb{R}$, defined by $\mathcal{H}(\widetilde{{\bf x}}(t))={\bf x}^{gr}$, is bijective map and ${\bf x}^{*gr}$
		satisfies
		$$\partial_iK^{gr}[{\bf x}^{*gr},{^CD}^{\beta,\rho}_{a^+}{\bf x}^{*gr}](t,\mu,\alpha_f)+{D}^{\beta,\rho}_{b^-}\left(
		\partial_{i+2}K^{gr}[{\bf x}^{*gr},{^CD}^{\beta,\rho}_{a^+}{\bf x}^{*gr}](t,\mu,\alpha_f)\right)=0,$$
		$ \forall \alpha_K,i\in I$, where $K^{gr}=f^{gr}+\lambda^{gr}
		h^{gr}$.

If  $\lambda^{gr}=\lambda^{gr}(\mu,\alpha_{\lambda})\ge 0$ and $f^{gr}[{\bf x}^{gr},{^CD}^{\beta,\rho}_{a^+}{\bf x}^{gr}](t,\mu,\alpha_f)$,
		$h^{gr}[{\bf x}^{gr},{^CD}^{\beta,\rho}_{a^+}{\bf x}^{gr}](t,\mu,\alpha_{h})$ are convex with respect to first $4$-arguments in ${^iX}_{ad}^{gr}$, then ${\bf x}^{*gr}$ is a minimal solution of (GIP) and $\widetilde{{\bf x}}^*(t)=\mathcal{H}^{-1}({\bf x}^{*gr})$ is a minimal solution of (FIP).
	\end{proposition}

	\begin{proof} Since $f^{gr},h^{gr}$ are convex with respect to first $4$-arguments and $\lambda(\mu, \alpha_{\lambda})\ge 0$,
		$K^{gr}$ is convex with respect to first $4$-arguments on ${^iX}_{ad}^{gr}$. By Proposition \ref{prop3.3}, ${\bf x}^{*gr}$ is a minimal solution of (GP) with the function $K^{gr}$ instead of $f^{gr}$. This leads to
		$$\int_{a}^{b}K^{gr}({\bf x}^{*gr}+{\bf w}^{gr})dt- \int_{a}^{b} K^{gr}({\bf x}^{*gr})dt\ge 0,$$
		for any admissible function ${\bf x}^{*gr}+ {\bf w}^{gr}$ with  $w_i^{gr}(a,\mu,\alpha_{w_i})=w_i^{gr}(b,\mu,\alpha_{w_i})=0$ for all $i\in I$,
		or equivalently,
		
		$\displaystyle\int_{a}^{b}
		\left(f^{gr}[{\bf x}^{*gr}+
		{\bf w}^{gr},{^CD}^{\beta,\rho}_{a^+}{\bf x}^{*gr}+ {^CD}^{\beta,\rho}_{a^+}{\bf w}^{gr}](t,\mu,\alpha_f)\right)dt$
		$$+\int_{a}^{b}
		\left(\lambda^{gr}
		h^{gr}[{\bf x}^{*gr}+
		{\bf w}^{gr},{^CD}^{\beta,\rho}_{a^+}{\bf x}^{*gr}+ {^CD}^{\beta,\rho}_{a^+}{\bf w}^{gr}](t,\mu,\alpha_{h})\right)dt$$
		$$\ge \int_{a}^{b}\left(
		f^{gr}[{\bf x}^{*gr},{^CD}^{\beta,\rho}_{a^+}{\bf x}^{*gr}](t,\mu,\alpha_f)
		+\lambda^{gr}h^{gr}[{\bf x}^{*gr},{^CD}^{\beta,\rho}_{a^+}{\bf x}^{*gr}](t,\mu,\alpha_{h})\right)dt.
		$$
		Utilizing the fractional isoperimetric constraints, we conclude that
		$$\int_{a}^{b}
		\left(f^{gr}[{\bf x}^{*gr}+
		{\bf w}^{gr},{^CD}^{\beta,\rho}_{a^+}{\bf x}^{*gr}+ {^CD}^{\beta,\rho}_{a^+}{\bf w}^{gr}](t,\mu,\alpha_f)\right)dt+\lambda^{gr}
		l^{gr}$$
		$$\ge \int_{a}^{b}\left(
		f^{gr}[{\bf x}^{*gr},{^CD}^{\beta,\rho}_{a^+}{\bf x}^{*gr}](t,\mu,\alpha_f)\right)dt+\lambda^{gr}
		l^{gr},$$
		or equivalently,
		$$\int_{a}^{b}
		\left(f^{gr}[{\bf x}^{*gr}+
		{\bf w}^{gr},{^CD}^{\beta,\rho}_{a^+}{\bf x}^{*gr}+ {^CD}^{\beta,\rho}_{a^+}{\bf w}^{gr}](t,\mu,\alpha_f)\right)dt\ge \int_{a}^{b}\left(
		f^{gr}[{\bf x}^{*gr},{^CD}^{\beta,\rho}_{a^+}{\bf x}^{*gr}](t,\mu,\alpha_f)\right)dt.$$
		Hence, ${\bf x}^{*gr}$ is a minimal solution of (GIP) and $\widetilde{{\bf x}}(t)$ is a minimal solution of (FIP) by Proposition \ref{prop3.4}.
	\end{proof}

	\begin{example} Let $\rho\in [\frac{1}{3},2]$. Consider the (FIP) as follows
		
		\noindent (FIP) $\widetilde{\min}$ $\displaystyle\widetilde{J}(\widetilde{x})={\widetilde{\int}}_{0}^{1}
		({_{gr}^C\mathcal{D}}^{\frac{2}{3},\rho}_{0^+}
		\widetilde{x}(t))^2dt$
		
		s.t. $\displaystyle\widetilde{H}(\widetilde{x})={\widetilde{\int}}_{0}^{1}
		\widetilde{x}(t)dt=\widetilde{l}=(-1,0,1)$,
		
		\hskip1cm $\widetilde{x}(0)=(0,0,0),\widetilde{x}(1)=(3,4,5)$.
		
		\noindent The granular problem corresponding to the problem (FIP) is
		
		\noindent (GIP) min $\displaystyle J^{gr}(x^{gr})=\int_{0}^{1}
		({^CD}^{\frac{2}{3},\rho}_{0^+}
		x^{gr}(t,\mu,\alpha_x))^2dt$
		
		\noindent s.t. $\displaystyle H^{gr}(x^{gr})=\int_{0}^{1}
		x^{gr}(t,\mu,\alpha_{x})dt=-1+\mu+(2-2\mu)\alpha_l,\forall \alpha_{x}=\alpha_l,$
		
		\noindent $x^{gr}(0,\mu,\alpha_x)=0,x^{gr}(1,\mu,\alpha_x)=3+\mu+(2-2\mu)\alpha_{u}, \forall \alpha_x=\alpha_{u}$, where  $\widetilde{u}=(3,4,5)$.
		
		\noindent Since $h^{gr}(x^{gr},{^CD}^{\frac{2}{3},\rho}_{0^+}
		x^{gr},t,\mu,\alpha_h)=x^{gr}$, one has
		$$\delta H(x^{gr},w^{gr})=\int_0^1(\partial_1h^{gr}+{D}^{\frac{2}{3},\rho}_{1^-}
		(\partial_2h^{gr}))dt=\int_0^11dt=1\neq 0,\forall x^{gr}.$$
		Applying the granular Euler-Lagrange condition for $K^{gr}=f^{gr}+\lambda^{gr}h^{gr}=({^CD}^{\frac{2}{3},\rho}_{0^+}
		x^{gr})^2+\lambda^{gr}x^{gr}$, we obtain that
		$$0=\partial_1K^{gr}+{D}^{\frac{2}{3},\rho}_{1^-}(
		\partial_2K^{gr})=\lambda^{gr}+{D}^{\frac{2}{3},\rho}_{1^-}(
		2{^CD}^{2/3}_{0^+}
		x^{gr}).$$
		This implies the following equation
		$${D}^{\frac{2}{3},\rho}_{1-}(
		{^CD}^{\frac{2}{3},\rho}_{0^+}
		x^{gr})=-\frac{\lambda^{gr}}{2}.$$
		Since
$${D}^{\frac{2}{3},\rho}_{1-}(
		{^CD}^{\frac{2}{3},\rho}_{0^+}
		x^{gr})=-\frac{\lambda^{gr}}{2}\Leftrightarrow
{D}^{\beta,\rho}_{1^-}\left(-\frac{2}{\lambda^{gr}}.
{^CD}^{\beta,\rho}_{0^+}
				x_2^{gr}(t,\mu,\alpha_{x_2}))\right)=1,$$
we deduce from Lemma \ref{lem4.6} (ii) that
		$$-\frac{2}{\lambda^{gr}}.{^CD}^{\frac{2}{3},\rho}_{0^+}
		x^{*gr}(t,\mu,\alpha_{x^*})
=\frac{1}{\Gamma(1+\frac{2}{3})}\left(\frac{1-t^\rho}
{\rho}\right)^{\frac{2}{3}}
+C\left(\frac{1-t^\rho}{\rho}\right)^{-\frac{1}{3}},$$
$$\Leftrightarrow{^CD}^{\frac{2}{3},\rho}_{0^+}
		x^{*gr}(t,\mu,\alpha_{x^*})
=-\frac{\lambda^{gr}}{2}\left(\frac{3}{2\Gamma(\frac{2}{3})}
\left(\frac{1-t^\rho}{\rho}\right)^{\frac{2}{3}}
+C\left(\frac{1-t^\rho}{\rho}\right)^{-\frac{1}{3}}\right) $$
		By applying Lemma \ref{lem4.6} (iii), we deduce from the above equation, $x^{*gr}(0,\mu,\alpha_{x^*})=0$  that
		$$x^{*gr}(t,\mu,\alpha_{x^*})
=-\frac{\lambda^{gr}}{2}
\left(\frac{1}{\Gamma(\frac{2}{3})}\int_0^t
(t^\rho-\tau^\rho)^{-\frac{1}{3}}.
		\tau^{\rho-1}\left(\frac{3}{2\Gamma(\frac{2}{3})}
\left(\frac{1-\tau^\rho}{\rho}\right)^{\frac{2}{3}}
+C\left(\frac{1-\tau^\rho}{\rho}\right)^{-\frac{1}{3}}\right)d\tau\right)
		$$
		$$=-\frac{3\lambda^{gr}\rho^{-\frac{2}{3}}}
{4(\Gamma(\frac{2}{3}))^2}
\int_0^t(t^\rho-\tau^\rho)^{-\frac{1}{3}}.
		\tau^{\rho-1}(1-\tau^\rho)^{\frac{2}{3}}
d\tau-\frac{C\lambda^{gr}\rho^{\frac{1}{3}}}{2\Gamma(\frac{2}{3})}
\int_0^t(t^\rho-\tau^\rho)^{-\frac{1}{3}}.
		\tau^{\rho-1}(1-\tau^\rho)^{-\frac{1}{3}}d\tau.$$
		It follows from Lemma \ref{lem2.15} (i) that
		$$x^{*gr}(t,\mu,\alpha_{x^*})=
-\frac{3\lambda^{gr}\rho^{-\frac{2}{3}}}
{4(\Gamma(\frac{2}{3}))^2}.
		\frac{t^{\frac{2}{3}\rho}}{\frac{2}{3}\rho}
{_2F_1}\left(1,-\frac{2}{3};\frac{5}{3};t^\rho\right)-
\frac{\lambda^{gr}C_1\rho^{\frac{1}{3}}}
		{2\Gamma(\frac{2}{3})}.
\frac{t^{\frac{2}{3}\rho}}{\frac{2}{3}\rho}
		{_2F_1}\left(1,\frac{1}{3};\frac{5}{3};t^\rho\right)$$
$$=
-\frac{9\lambda^{gr}\rho^{-\frac{5}{3}}}
{8(\Gamma(\frac{2}{3}))^2}.t^{\frac{2}{3}\rho}
{_2F_1}\left(1,-\frac{2}{3};\frac{5}{3};t^\rho\right)-
\frac{3\lambda^{gr}C_1\rho^{-\frac{2}{3}}}
		{4\Gamma(\frac{2}{3})}.
t^{\frac{2}{3}\rho}
		{_2F_1}\left(1,\frac{1}{3};\frac{5}{3};t^\rho\right).$$
		By using the endpoint conditions $x^{*gr}(1,\alpha_{x^*})=3+\mu+(2-2\mu)\alpha_{u}$ with $\alpha_{x^*}=\alpha_{u}$ and Lemma \ref{lem2.15} (ii), one has
		\begin{equation}\label{exa3.26-1}
			3+\mu+(2-2\mu)\alpha_{u}=-
\frac{9\lambda^{gr}\rho^{-\frac{5}{3}}}
{8(\Gamma(\frac{2}{3}))^2}.
\frac{\frac{2}{3}}{2.\frac{2}{3}}
			-\frac{3\lambda^{gr}C_1\rho^{-\frac{2}{3}}}
		{4\Gamma(\frac{2}{3})}.
\frac{\frac{2}{3}}{2.\frac{2}{3}-1}.
		\end{equation}
		Hence,		
$$C_1=-\frac{2\rho^{\frac{2}{3}}}{3\lambda^{gr}}
\left((3+\mu+(2-2\mu)
\alpha_{u})\Gamma(\frac{2}{3})
		+\frac{9\lambda^{gr}\rho^{-\frac{5}{3}}}
{16\Gamma(\frac{2}{3})}\right).$$
		Substituting into $x^{*gr}(t,\mu,\alpha_{x^*})$, one has
		$$x^{*gr}(t,\mu,\alpha_{x^*})
=-\frac{9\lambda^{gr}\rho^{-\frac{5}{3}}}
{8(\Gamma(\frac{2}{3}))^2}.p_1(t)
+\frac{1}{2}\left(3+\mu+(2-2\mu)\alpha_{x^*}
		+\frac{9\lambda^{gr}\rho^{-\frac{5}{3}}}
{16(\Gamma(\frac{2}{3}))^2}\right)
.p_2(t),$$
where $p_1(t)=t^{\frac{2}{3}\rho}
{_2F_1}\left(1,-\frac{2}{3};\frac{5}{3};t^\rho\right)$, $p_2(t)=t^{\frac{2}{3}\rho}.
		{_2F_1}\left(1,\frac{1}{3};\frac{5}{3};t^\rho\right)$.
		
We deduce from Lemma \ref{lem2.15} (iii) that
				$$\int_{0}^{1}t^{\frac{2}{3}\rho}{_2F_1}
\left(1,-\frac{2}{3};
\frac{5}{3};t^\rho\right)dt=
		\frac{1}{\frac{2}{3}\rho+1}.{_3F_2}
\left(1,-\frac{2}{3},\frac{2}{3}+\frac{1}{\rho};
\frac{5}{3},\frac{5}{3}+\frac{1}{\rho};1\right)=
		\frac{1}{\frac{2}{3}\rho+1}.F_1.$$
$$\int_{0}^{1}t^{\frac{2}{3}\rho}{_2F_1}\left(1,\frac{1}{3};
\frac{5}{3};t^\rho\right)dt
		=\frac{1}{\frac{2}{3}\rho+1}.{_3F_2}
\left(1,\frac{1}{3},\frac{2}{3}+\frac{1}{\rho};
\frac{5}{3},\frac{5}{3}+\frac{1}{\rho};1\right)=
\frac{1}{\frac{2}{3}\rho+1}.F_2$$
These results could be checked by using Octave or Matlab a specific value of $\rho$, for example $\rho=2$, or $\rho=\frac{1}{2}$. For instance, using
		
\noindent{$>>$pkg load gsl}
		
\noindent{$>>$integral(@(t) t.\textasciicircum(4/3).*gsl$\_$sf$\_$
			hyperg$\_$2F1(1,-2/3,5/3,t.\textasciicircum 2),0,1)}

\noindent{$>>$pkg load symbolic}

\noindent{$>>$(1/(4/3+1))$*$double(hypergeom(sym($[$1, -2/3,7/6$]$), sym($[$5/3,13/6$]$), sym(1)))}
		
\noindent in Octave, we get the results {\bf 0.3241}, i.e.,
$$\int_{0}^{1}t^{\frac{4}{3}}{_2F_1}
\left(1,-\frac{2}{3};
\frac{5}{3};t^2\right)dt=
		\frac{1}{\frac{4}{3}+1}.{_3F_2}
\left(1,-\frac{2}{3},\frac{7}{6};
\frac{5}{3},\frac{13}{6};1\right)\approx {\bf 0.3241}.$$

		Hence, it follows from these results the constraint and $\alpha_{x^*}=\alpha_l$ that
$$	-1+\mu+(2-2\mu)\alpha_{x^*}=\int_{0}^{1}
			x^{*gr}(t,\mu,\alpha_{x^*})dt$$
$$=-\frac{9\lambda^{gr}\rho^{-\frac{5}{3}}}
{8(\Gamma(\frac{2}{3}))^2}.\frac{1}{\frac{2}{3}\rho+1}.F_1
+\frac{1}{2}\left(3+\mu+(2-2\mu)\alpha_{x^*}
		+\frac{9\lambda^{gr}\rho^{-\frac{5}{3}}}
{16(\Gamma(\frac{2}{3}))^2}\right).\frac{1}{\frac{2}{3}\rho+1}.F_2,
$$
		which in turns derives that
		$$\lambda^{gr}(\mu,\alpha_{\lambda})=\frac{32(\Gamma(\frac{2}{3}))^2
\rho^{5/3}}{9\left(F_2-4.F_1\right)}\left((-1+\mu+(2-2\mu)\alpha_{\lambda})
(\frac{2}{3}\rho+1)-\frac{1}{2}
(3+\mu+(2-2\mu)\alpha_{\lambda})F_2\right).$$
		Setting $L:=(-1+\mu+(2-2\mu)\alpha_{\lambda})
(\frac{2}{3}\rho+1)-\frac{1}{2}
(3+\mu+(2-2\mu)\alpha_{\lambda})F_2$ and substituting $\lambda^{gr}$ into $x^{*gr}(t,\mu,\alpha_{\lambda})$, one has
		$$x^{*gr}(t,\alpha_{x^*})=-\frac{9\rho^{-\frac{5}{3}}}
{8(\Gamma(\frac{2}{3}))^2}.\frac{32(\Gamma(\frac{2}{3}))^2
\rho^{5/3}}{9\left(F_2-4.F_1\right)}L.p_1(t)
$$
$$+\frac{1}{2}\left(3+\mu+(2-2\mu)\alpha_{x^*}
		+\frac{9\rho^{-\frac{5}{3}}}
{16(\Gamma(\frac{2}{3}))^2}.\frac{32(\Gamma(\frac{2}{3}))^2
\rho^{5/3}}{9\left(F_2-4.F_1\right)}L\right)
.p_2(t).$$
$$=-\frac{4L}{F_2-4.F_1}.p_1(t)
+\frac{1}{2}\left(3+\mu+(2-2\mu)\alpha_{x^*}
		+\frac{2L}
{F_2-4.F_1}\right)
.p_2(t).$$
		
		Moreover, we can check that $f^{gr},h^{gr}$ convex with respect to first 2-arguments in $\Theta^{gr}$.

Now, we will prove that $\lambda^{gr}(\mu,\alpha_{\lambda})\ge 0$ for all $0<\rho\le 2$. Setting $a_3:=-1+\mu+(2-2\mu)\alpha_{\lambda}$, we get that $a_3\in [-1,1]$ and
$$L=a_3\left(\frac{2}{3}\rho+1\right)-\frac{1}{2}(4+a_3).F_2,$$
$$\lambda^{gr}(\mu,\alpha_{\lambda})=\frac{32(\Gamma(\frac{2}{3}))^2
\rho^{5/3}}{9\left(F_2-4.F_1\right)}\left(a_3\left(\frac{2}{3}\rho+1\right)
-\frac{1}{2}(4+a_3).F_2\right).$$
By using the graph of $\lambda^{gr}(\mu,\alpha_{\lambda})$ on $\rho\in [\frac{1}{3},2]$ and $D=\{(\rho,a_3)\mid \frac{1}{3}\le \rho \le 2, -1\le a_3\le 1\}$, we could check that $\lambda^{gr}(\mu,\alpha_{\lambda})\ge 0$.

Invoking Proposition \ref{prop3.6}, we obtain that $x^{*gr}(t,\mu,\alpha_{x^*})$ is  a minimum of (GIP) and
		$\widetilde{x}^*(t)=\mathcal{H}^{-1}(x^{*gr}(t,\mu,\alpha_{x^*}))$ with the $\mu$-level set
		$$[\widetilde{x}^*(t)]^{\mu}=\left[-\frac{4L_1}{F_2-4.F_1}.p_1(t)
+\frac{1}{2}\left(3+\mu
		+\frac{2L_1}
{F_2-4.F_1}\right)
.p_2(t),\right.
		$$
		$$\left.-\frac{4L_2}{F_2-4.F_1}.p_1(t)+
\frac{1}{2}\left(5-\mu
		+\frac{2L_2}
{F_2-4.F_1}\right)
.p_2(t)\right],$$
where $L_1=(-1+\mu)
(\frac{2}{3}\rho+1)-\frac{1}{2}
(3+\mu)F_2$ and $L_2=(1-\mu)
(\frac{2}{3}\rho+1)-\frac{1}{2}
(5-\mu)F_2$, is a minimum of (FIP), see Figure 3.2.1 and 3.2.2.

\begin{figure}[htbp]
    \centering
    \begin{tabular}{cc}

        \includegraphics[width=0.45\textwidth]{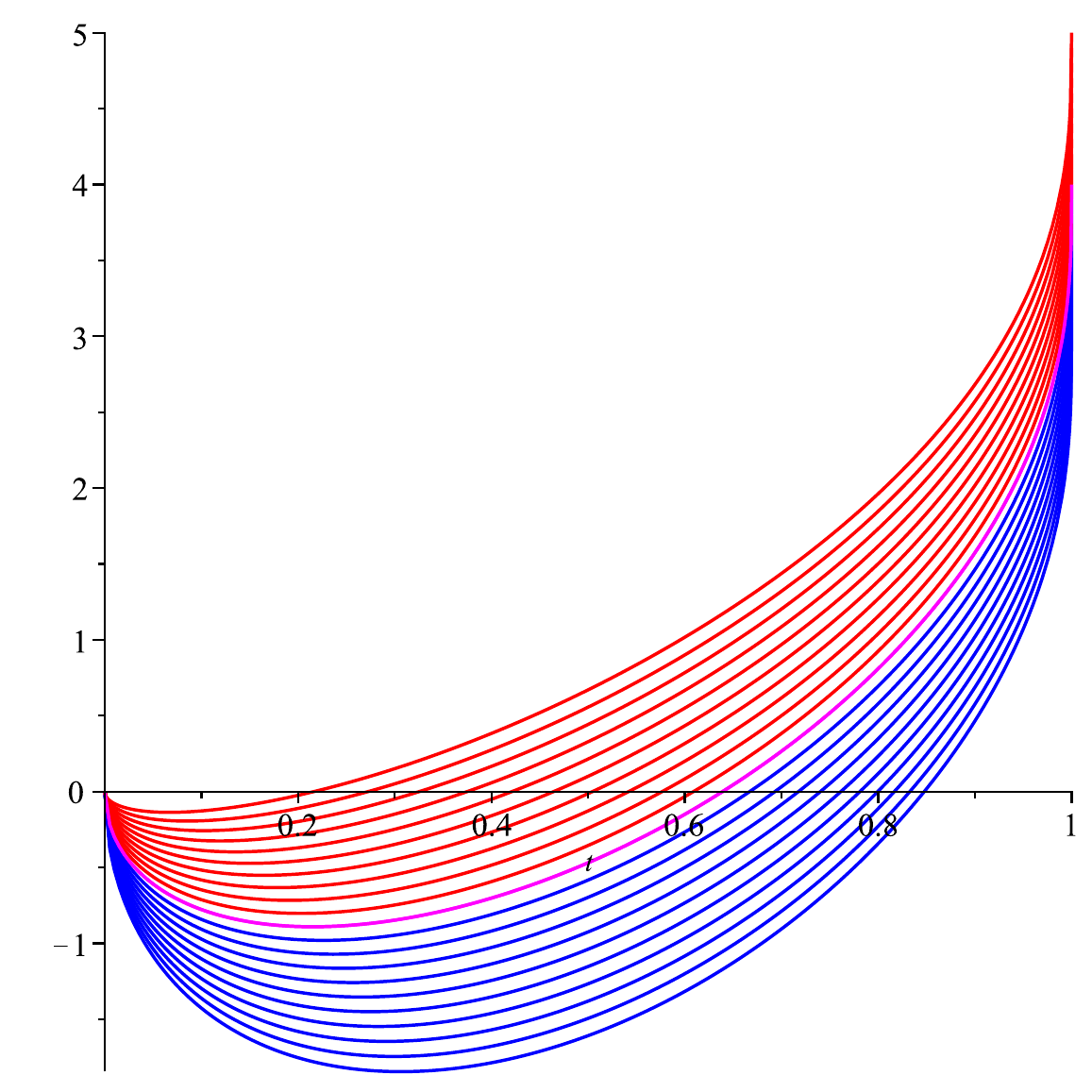} &
        \includegraphics[width=0.45\textwidth]{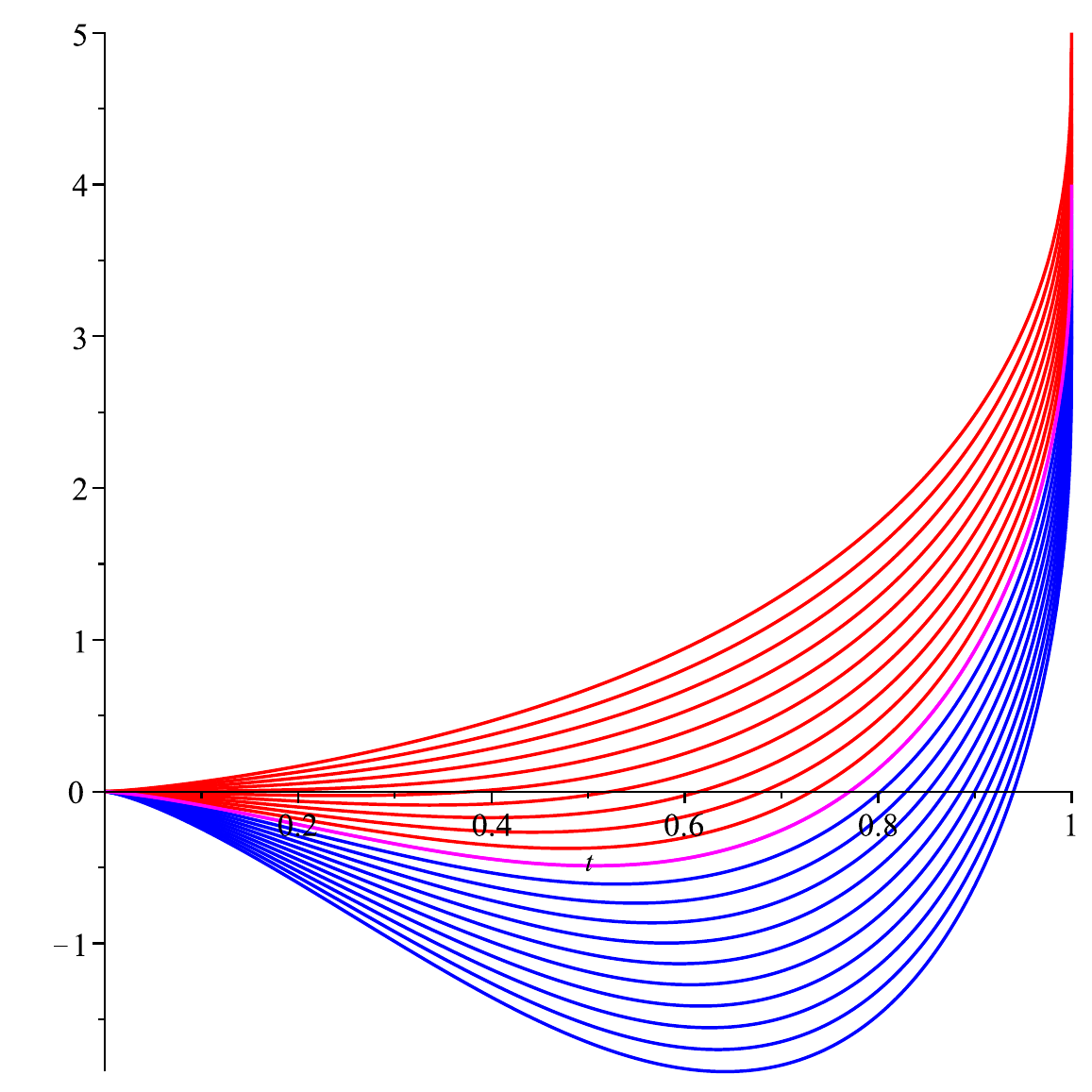} \\

        Figure 3.2.1. $\rho=\frac{3}{4}$ &
        Figure 3.2.2. $\rho=2$ \\
    \end{tabular}
\end{figure}
\end{example}
\section{Conclusions}
By utilizing the granular approach, the granular CK fractional derivatives of a fuzzy function could exist in Example 3.2, while the $gH$-CK fractional  derivatives \cite{HVD19} do not exist. Hence, the proposed results demonstrate clear advantages over previous ones in specific cases. When $\rho=1$, the results and the examples in Section 4 become the results and examples in \cite{TTKN25}.	When the considered functions are crisp, our results reduce to the results in  \cite{A17}. However, the examples in this paper are new, even in the crisp cases. In future work, the granular Caputo-Katugampola derivatives of the fuzzy functions could be defined similarly for other classes of fuzzy functions such as the type-2 fuzzy functions, the intuitionistic fuzzy functions or the neutrosophic fuzzy functions. Moreover, the results of the applications of the granular Caputo derivatives in the fuzzy fractional differential equations, the other fuzzy variational problems, the fuzzy optimal control problems in \cite{NZ17} and related papers  could be generalized for the granular Caputo-Katugampola derivatives, which  are also the interesting subjects for the next research.

\noindent{\bf Conflict of interest}
	
The authors declare that they have no conflict of interest.
	
	
	
\section*{Acknowledgement}
	This work was supported by National Foundation for Science and Technology Development (NAFOSTED) of Vietnam under grant number 101.01-2025.16.

	

\end{document}